\documentclass[a4paper, 15pt]{article}
\usepackage[utf8]{inputenc}
\usepackage[T1]{fontenc}
\usepackage{babel}  
\usepackage{tikz}
\usetikzlibrary{positioning}
\usepackage[left=2cm,right=2cm,top=2cm,bottom=2cm]{geometry}         \usetikzlibrary{decorations.text}
\usetikzlibrary{decorations.pathreplacing}

\usepackage{graphicx}			 	
\usepackage{amssymb}
\usepackage[backend=bibtex,style=numeric,
          maxnames=1000]{biblatex}
\usepackage{lineno}
\usepackage{amsthm}
\usepackage{amsmath}
\usepackage{amsfonts}

\usepackage{amssymb}
\usepackage{latexsym}
\usepackage{hyperref}
\usepackage{mathrsfs}
\usepackage[utf8]{inputenc}
\usepackage[parfill]{parskip}
\usepackage{xcolor}
\usepackage[affil-it]{authblk}
\usepackage{enumitem}

\usepackage{mathtools}

\usepackage[toc]{appendix}

\numberwithin{equation}{section}

\newtheorem{defi}{Definition}[section]
\newtheorem{unTheorem}{Theorem}[section]

\newtheorem{propal}{Proposition}[section]
\newtheorem{rem}{Remark}[section]
\newtheorem{cor}{Corollary}[section]
\newtheorem{lem}{Lemma}[section]

\newtheorem{nota}{Notation}[section]

\title{Semi-classical limit of the massive Klein-Gordon-Maxwell system toward the relativistic Euler-Maxwell system via an adapted modulated energy method}
\author[1]{Tony Salvi}
\affil[1]{Centre de mathématiques Laurent-Schwartz, École polytechnique, tony.salvi@polytechnique.edu, Département de mathématiques et applications, École normale supérieure, tony.salvi@ens.psl.eu}
\date{}
\begin{document}

\maketitle

\begin{center}
    \textbf{Abstract}
    \hfill\begin{minipage}{\dimexpr\textwidth-1cm}
    We show that the momentum, the density, and the electromagnetic field associated with the massive Klein-Gordon-Maxwell equations converge in the semi-classical limit towards their respective equivalents associated with the relativistic Euler-Maxwell equations. The proof relies on a modulated stress-energy method and a compactness argument. We also give a proof of the well-posedness of the relativistic Euler-Maxwell equations and show how this system, and so the semi-classical limit of Klein-Gordon-Maxwell, is related to the relativistic massive Vlasov-Maxwell equations.
\end{minipage}
\end{center}
\tableofcontents
\section{Introduction}
\label{section:Intro}
In this paper, we are interested in the semi-classical limit of the massive Klein-Gordon-Maxwell (mKGM)\footnote{We also call mKGM the usual version of the system, that is, with $\hbar\sim\varepsilon=1$.} equations in the (3+1)-dimensional Minkowski spacetime  
\begin{equation}
\begin{cases}
\label{eq:KGMscintro}
\boldsymbol{\nabla}_\alpha (F^\varepsilon)^{\alpha\beta}=-\Im
(\Phi^\varepsilon\overline{(\boldsymbol{D}^\varepsilon)^{\beta}\Phi^\varepsilon}),\\
(\boldsymbol{D}^\varepsilon)^\alpha(\boldsymbol{D}^\varepsilon)_\alpha\Phi^\varepsilon=\Phi^\varepsilon.
\end{cases}
\end{equation}
The complex function $\Phi^\varepsilon$ is the wave function, the 2-form $F^\varepsilon$ is the Faraday tensor representative of the electro-\\magnetic field, and $\varepsilon$ a small parameter representative of the Planck constant. The other important hidden quantity is the electromagnetic four-potential $\textbf{A}^\varepsilon$. The Faraday tensor $F^\varepsilon$ is its associated curvature tensor, that is,
\begin{align*}
F^\varepsilon=d\boldsymbol{A}^\varepsilon
\end{align*}
(or $F_{\alpha\beta}^\varepsilon=\boldsymbol{\nabla}_\alpha\textbf{A}^\varepsilon_{\beta}-\boldsymbol{\nabla}_\beta\textbf{A}^\varepsilon_\alpha$ in coordinates), and the operator $\textbf{D}^\varepsilon$ is its associated covariant derivative, that is
\begin{align*}
(\textbf{D}^\varepsilon)_\alpha=\varepsilon\boldsymbol{\nabla}_\alpha+i\boldsymbol{A}^\varepsilon_\alpha,
\end{align*}
 where $\boldsymbol{\nabla}$ is the standard flat spacetime gradient. \\
 We look at the behavior of the solutions at the semi-classical limit (when $\varepsilon$ goes to 0), when the quantum effects vanish. At this limit, the dynamics is given by the relativistic Euler-Maxwell (REM) equations 
\begin{equation}
\label{eq:REMintro}
        \begin{cases}
               \boldsymbol{\nabla}_\alpha F^{\alpha\beta}=\textbf{U}^\beta\rho,\\               \textbf{U}^\alpha\boldsymbol{\nabla}_\alpha\rho+\boldsymbol{\nabla}_\alpha\textbf{U}^\alpha\rho=0,\\              \textbf{U}^\alpha\boldsymbol{\nabla}_\alpha\textbf{U}_\beta=F_{\alpha\beta}\textbf{U}^\alpha,\\
            \textbf{U}_\alpha\textbf{U}^\alpha=-1,
        \end{cases}
    \end{equation}
where $\rho$ is the charge density, $F$ the Faraday tensor, and $\textbf{U}$ the four-velocity vector field. The REM system \eqref{eq:REMintro} describes the evolution of a pressureless charged fluid and its associated electromagnetic field. 
In this paper, we show that at the semi-classical limit the momentum $\textbf{J}^\varepsilon=-\Im
(\Phi^\varepsilon\overline{(\boldsymbol{D}^\varepsilon)^{\beta}\Phi^\varepsilon})$, the density $\rho^\varepsilon=\Phi^\varepsilon$, and the electromagnetic field $F^\varepsilon$ associated with \eqref{eq:KGMscintro} converge in Lebesgue norms to the momentum $\textbf{J}=\textbf{U}\rho$, the density $\rho$, and the electromagnetic field $F$ associated with \eqref{eq:REMintro}. In the appendix \ref{section:WKB}, we derive formally \eqref{eq:REMintro} from \eqref{eq:KGMscintro} using the WKB expansion method.

In mathematics and physics, it is natural to study the semi-classical limit of quantum systems, mainly in the non-relativistic setting. For mathematical works on semi-classical limits, we refer to \cite{zbMATH01132286} and \cite{zbMATH05243173} on the convergence of the nonlinear Schrödinger (NLS) equation to the classical compressible Euler equations, to \cite{zbMATH00482230,zbMATH00203353,zbMATH01860573,zbMATH05590720,zbMATH01987564,zbMATH05124485,Carles_2009,Lafleche_2021,zbMATH07713387,zbMATH07680771} on the convergence of the Schrödinger-Poisson (SP) system to the Vlasov-Poisson system (or Euler-Poisson in the monokinetic case), and to \cite{zbMATH01877186,zbMATH01970480,zbMATH05150102,zbMATH06101438,zbMATH07570763,zbMATH07699568,möller2023paulipoissonequationsemiclassicallimit,yang2024semiclassicallimitpaulipoisswelleulerpoisswell,leopold2024derivationvlasovmaxwellmaxwellschrodingerequations} for various related works that consider double limits or different settings. Most of these works are put in context in Section \ref{section:conclusion}. We also refer to \cite{NonnenmacherINTRODUCTIONTS} and \cite{zbMATH06054091} for, respectively, a rich introduction and a general review of semi-classical analysis.\\
The SP system is known to be the non-relativistic limit of the mKGM system, we refer to \cite{zbMATH02058032} and \cite{10.1155/S107379280320310X} for rigorous proofs and to Section \ref{section:conclusion} for more context. Thus, considering the previous remark, it is natural to study the semi-classical limit of mKGM and to expect the limit to be the relativistic Vlasov-Maxwell (RVM) system (presented in Section \ref{section:mRVM}), that is, the REM system \eqref{eq:REMintro} in the monokinetic case. We also point out that, as far as we know, the present work is the first to treat the semi-classical limit of (massive or not) KGM and to obtain the (massive or not) RVM system as a semiclassical limit even in the monokinetic case. 

The present work is a direct continuation of the ideas found in \cite{salvi2024semiclassicallimitkleingordonequation} (by the author) on the semi-classical limit for the massive nonlinear Klein-Gordon equation applied to the mKGM system. In \cite{salvi2024semiclassicallimitkleingordonequation}, we adapt the modulated energy method of \cite{zbMATH05243173} (to show the semi-classical limit of the nonlinear Schrodinger (NLS) equation) to the wave equation and the relativistic setting, as a modulated stress-energy method. The modulated energy method goes back to the work of \cite{zbMATH01442880}. We refer to \cite{zbMATH06101438} for a history of this method applied to the semi-classical and hydrodynamic limits of quantum mechanics equations such as Klein-Gordon, Schrodinger, and Gross-Pitaevskii. In fact, the modulated energy method has already been applied to the Schrödinger-Poisson system in \cite{zbMATH01987564}. Thus, one can consider the present paper as a relativistic version of \cite{zbMATH01987564} but with a strong convergence of the observables, this is due to a compactness argument given in Section \ref{subsection:Coerciviy}. 
\subsection{Physical context}
\label{subsection:context}
We present in this section some physics of the problem. Before that, we give a useful notation that holds throughout the paper.
\begin{nota}
\label{nota:gradNOTA1}
    The symbol $\boldsymbol{\nabla}$ alone denotes the spacetime gradient and we note $|\boldsymbol{\nabla} f|^2=\boldsymbol{\nabla}f\cdot\overline{\boldsymbol{\nabla}f}=\partial_tf\overline{\partial_tf}+\nabla f\cdot\overline{\nabla f}$ with $\nabla$ the space gradient and $\cdot$ the usual scalar products of both $\mathbb{R}^3$ and $\mathbb{R}^4$. We note with $\partial$ any space or time derivative.\\
\end{nota}
\begin{nota}
\label{nota:gradNOTA2}
    Four-vectors are noted with \textbf{bold letters} and classical three-dimensional vectors with Roman letters.
    For $\textbf{X}$, a four-vector, we note $X$, with $X^i=\textbf{X}^i$, its space components and $X^0=\textbf{X}^0$ its time component.
\end{nota}
In its physical scaling the KGM system is
\begin{equation}
\begin{cases}
\label{eq:KGMcontext}
\boldsymbol{\nabla}_\alpha F^{\alpha\beta}=-\Im
(\Phi\overline{(\boldsymbol{D})^{\beta}\Phi}),\\
\textbf{D}^\alpha\textbf{D}_\alpha\Phi=c^2m^2\Phi,
\end{cases}
\end{equation}
with $\textbf{D}_\alpha=\hbar\boldsymbol{\nabla}_\alpha+iq\boldsymbol{A}_\alpha$. The constants are $m\geq0$ the rest mass, $q$ the charge, and $\hbar$ the Planck constant divided by $2\pi$. In flat spacetime, the indices are raised with respect to the Minkowski metric $g=Diag(-c^2,1,1,1)$ with $c>0$ the speed of light.
The system \eqref{eq:KGMcontext} is a model for quantum electrodynamics, the wave function $\Phi$ describes the behavior of a relativistic massive spinless charged particle\footnote{The particle is in a mean-field regime, by analogy with the Schrödinger-Poisson system. Nonetheless, there is no derivation of mKGM as a mean-field limit.} that interacts with the electromagnetic field $F$. The latter is usually decomposed into an electric and a magnetic part
\begin{align*}  
        &F_{0i}=E_i, &^\star F_{0i}=B_i,
\end{align*}
where $^\star F$ is the Hodge dual of $F$, and has the charge current density (or momentum) $\textbf{J}=-\Im
(\Phi\overline{(\boldsymbol{D})^{\beta}\Phi})$ as a source. From the zeroth component of the Maxwell equations and the fact that $F=d\boldsymbol{A}$, we get
\begin{align}
\label{eq:maxwellconstrcontext}
    &\nabla\cdot E=-\textbf{J}^0, &\nabla\cdot B=0.
\end{align}
This is the Maxwell constraints. The energy of the system is 
\begin{equation}
\label{eq:energyKGMcontext}
\mathcal{E}_{KGM}=\int_{\mathbb{R}^3}{\frac{|\boldsymbol{D}\Phi|^2}{2}+\frac{c^2m^2|\Phi|^2}{2}+\frac{|E|^2+|B|^2}{2}dx}.
\end{equation}
We can also note that the Klein-Gordon equation 
\begin{align*}
    \textbf{D}^\alpha\textbf{D}_\alpha\Phi=c^2m^2\Phi,
\end{align*}
is derived from the relativistic energy-momentum relation 
\begin{align*}
    \textbf{p}^\alpha \textbf{p}_\alpha=-c^2m^2,
\end{align*}
with the correspondence principle $\textbf{p}_\alpha\to\frac{1}{i}\textbf{D}_\alpha$, where the covariant derivative takes into account the local gauge symmetry.
 Indeed, the KGM equations \eqref{eq:KGMcontext} are part of the gauge theory and are equivalent to the Yang-Mills-Higgs (YMH) equations with $U(1)$ as a symmetry group. The system is invariant by change of gauge. Let $\chi$ be a real function on 
$\mathbb{R}^{3+1}$ and let  
\begin{align*}
&\textbf{A}'_\alpha=\textbf{A}_\alpha+\boldsymbol{\nabla}_\alpha\chi, &\Phi'=e^{-i\chi}\Phi,
\end{align*}
then $(\textbf{A}',\Phi')$ is a solution to KGM if and only if $(\textbf{A},\Phi)$ is a solution too. 
\\ 
We recover \eqref{eq:KGMscintro} with the setting $c=1$, $q=1$, $m=1$ 
and $\hbar=\varepsilon$ for $\varepsilon$ small. It corresponds to the relativistic semi-classical physics, the relativistic effects due to the maximal speed $c$ are present and the quantum effects are negligible.
Another way to recover \eqref{eq:KGMscintro} is to set $\Phi^\varepsilon(t,x)=\Phi(\frac{t}{\varepsilon},\frac{x}{\varepsilon})$ where $\Phi$ is a solution to \eqref{eq:KGMcontext} with $\hbar=1$ (the natural units). We see that the wave function $\Phi^\varepsilon$ of the solution to \eqref{eq:KGMscintro} is high-frequency, its amplitude is small in comparison with its derivatives when $\varepsilon$ is small. In fact, it is recommended to think of $\Phi^\varepsilon$ as the monokinetic WKB ansatz
\begin{equation}
\label{eq:ansatzcontext}
    \Phi^\varepsilon\sim e^{i\frac{\omega}{\varepsilon}}\Psi,
\end{equation}
where $\Psi$ is low-frequency and where only the phase of $ \Phi^\varepsilon$ is high-frequency, we refer to Appendix \ref{section:WKB} on the WKB method and to Section \ref{subsection:initassum} for the initial data. 
\subsection{Modulated energy method} 
The modulated energy method is used when one studies the behaviors of solutions to physical systems of equations at the frontier of two regimes, typically when one parameter in a system tends to $0$ or $+\infty$. This method shows that if certain quantities associated with solutions to a starting system (here \eqref{eq:KGMscintro}) initially converge (in a sense to be precise) at this limit to their respective equivalents associated with a solution to a limiting system (here \eqref{eq:REMintro}), then this convergence holds on some interval of time\footnote{The method is often applied the other way around (as originally in \cite{zbMATH01442880}). One shows first that a certain limit exists and concludes with the modulated energy that it is solution to the limiting system.}. It tells us about the dynamic of the starting system in the limit regime. The typical examples of application are the study of the quasineutral limit, for which the method was originally developed in \cite{zbMATH01442880}, the mean-field limit as adapted first in \cite{Serfaty_2020}, and the semi-classical limit (when $\hbar$ tends to 0) as adapted first in \cite{zbMATH01987564}.\\
An energy has the property to remain constant and (in good cases) to control other physical quantities, for example, the observables of a quantum system (the momentum, the density, and the electromagnetic field). The modulated energy is constructed as a modulation of the energies of the starting systems, typically by introducing forcing terms from the
limiting system, and is meant to vanish at the limit. In this paper, the modulated energy\footnote{In fact, it is a representative of an equivalence class, see Section \ref{section:defmod}.} is 
\begin{align}
\label{eq:modulatedintro}
H^\varepsilon_0(t):=\int_{\mathbb{R}^3}{\left(\frac{|(\boldsymbol{D}^\varepsilon-i\textbf{U})\Phi^\varepsilon|^2}{2}+\frac{|E^\varepsilon-E|^2+|B^\varepsilon-B|^2}{2}\right)(t)dx},
\end{align}
where we use the electric and magnetic representation $E^\varepsilon$ and $B^\varepsilon$ (resp. $E$ and $B$) of the Faraday tensor $F^\varepsilon$ (resp. $F$), this notation is detailed in Definition \ref{defi:FelecKGM} (resp. \ref{defi:FelecREM}). The interpretation is that we cut the non-quantum\footnote{We do not say classical for non-quantum because the limiting system is relativistic.} part of the energy of \eqref{eq:KGMscintro} so that only the quantum part remains, vanishing at the semi-classical limit. Unlike the energy, the modulated energy is not constant, but its evolution is sufficiently regular to have\footnote{The Landau notation is given in \ref{nota:landaudef}.}
 \begin{align*}       
    H^\varepsilon_0(0)=O(\varepsilon^2)\implies H^\varepsilon_0(t)=O(\varepsilon^2),
\end{align*}
for $t\in[0,T]$ with $0<T$ the time for which \eqref{eq:REMintro} is well-posed for the considered initial data. This is the \textbf{propagation property}, we refer to Lemma \ref{lem:modpoint2Idea}. Then, as the energy, the modulated energy controls quantities of interest and their convergence when it converges to 0. This is the \textbf{coercivity property}, we refer to Lemma \ref{lem:modpoint1Idea}. It corresponds to
\begin{align*}
&\sup_{t\in[0,T]}H^\varepsilon(t)=O(\varepsilon^2)\implies \\
    &\lim_{\varepsilon\to0}||\textbf{J}^\varepsilon-\textbf{U}\rho||_{L^\infty([0,T],L^{1})}+||F^\varepsilon-F||_{L^\infty([0,T],L^2)}+||\rho^\varepsilon-\rho||_{L^\infty([0,T],L^1)}+||\sqrt{\rho^\varepsilon}-\sqrt{\rho}||_{L^\infty([0,T],L^2)}=0,
\end{align*}
where $\textbf{J}^\varepsilon=-\Im
(\Phi^\varepsilon\overline{(\boldsymbol{D}^\varepsilon)\Phi^\varepsilon})$ and $\rho^\varepsilon=|\Phi^\varepsilon|^2$. This holds under additional conditions on the decay and the convergence of the initial data given in the next Section and in point \ref{item:item2mainth} of Theorem \ref{unTheorem:TH1mainth} respectively. These conditions are mandatory to obtain the convergence of the density via a compactness argument. This is the main new feature of our method : the modulated energy is exploited to get extra uniform regularity on the density and then to get compactness with the uniform decay, see Section \ref{subsection:Coerciviy} for the full argument. The rest of the method is similar to \cite{salvi2024semiclassicallimitkleingordonequation}, we modulate the full stress energy tensor associated with \eqref{eq:KGMscintro} (defined in \ref{defi:energytensorKGM}) and define an equivalence class of modulated energy to prove the propagation property. 
 \subsection{Assumptions on the initial data}
 \label{subsection:initassum}
 We give the Definitions of well-prepared initial data for \eqref{eq:KGMscintro} and \eqref{eq:REMintro} to have solutions that match the requirements of Theorem \ref{unTheorem:TH1mainth}. The compatibility between these assumptions and the convergence assumptions \ref{item:item2mainth} of Theorem \ref{unTheorem:TH1mainth} is largely studied in Section \ref{section:InitialData}.\\
 For \eqref{eq:KGMscintro}, the well-posedness result we use is gauge-dependent, so we must prescribe initial data in the form of $(\varphi^\varepsilon,\mathscr{A}^\varepsilon,\pi^\varepsilon,\mathscr{E}^\varepsilon)$ to fit the temporal gauge $((A^\varepsilon)^0=0)$ requirement. The tuple $(\pi^\varepsilon,\mathscr{E}^\varepsilon)$ is the conjugate momentum of $(\varphi^\varepsilon,\mathscr{A}^\varepsilon)$. The physical quantities are then $(\Phi^\varepsilon,B^\varepsilon,\textbf{D}^\varepsilon_0\Phi^\varepsilon,E^\varepsilon)|_{t=0}=(\varphi^\varepsilon,\nabla\times\mathscr{A}^\varepsilon,\pi^\varepsilon,\mathscr{E}^\varepsilon)$.
 \begin{defi}
    \label{defi:wellprepKGMsc}
    Let $(\varphi^\varepsilon,\mathscr{A}^\varepsilon,\pi^\varepsilon,\mathscr{E}^\varepsilon)_{0<\varepsilon<1}$ be a family of initial data for \eqref{eq:KGMscintro} such that $\forall\varepsilon\in(0,1),$ $(\varphi^\varepsilon,\mathscr{A}^\varepsilon,\pi^\varepsilon,\mathscr{E}^\varepsilon)\in H^5\times H^5\times H^4\times H^4$. The initial data are well-prepared if    \begin{itemize}
\item   The weighted energy is uniformly bounded 
    \begin{align}
    \label{eq:wellprepunifdecay1}
        \int_{\mathbb{R}^3}{(1+|x|^2)^{\kappa}\left(\frac{|\pi^\varepsilon|^2}{2}+\frac{|(\varepsilon\nabla+i\mathscr{A})\varphi^\varepsilon|^2}{2}+\frac{|\varphi^\varepsilon|^2}{2}+\frac{|\mathscr{E}^\varepsilon|^2+|\nabla\times\mathscr{A}^\varepsilon|^2}{2}\right)dx}\leq c_0,
    \end{align}
    for some constants $c_0>0$ and $\kappa>0$ independent of $\varepsilon$.
     \item The following constraint holds 
    \begin{align}
    \label{eq:maxwellconstrwellprep}
    \nabla\cdot\mathscr{E}^\varepsilon=\Im{(\varphi^\varepsilon\overline{\pi^\varepsilon})}.
    \end{align}
      \end{itemize}
\end{defi}
The regularity and the temporal gauge are mandatory to be able to apply the global existence theorem of \cite{zbMATH03781718}. In fact, $(\Phi^\varepsilon,F^\varepsilon)\in H^2\times H^1$ is sufficient for the rest of our argument. \\
The high-frequency behavior of the solutions to mKGM $\eqref{eq:KGMscintro}$ is inherent to the semi-classical limit, but the monokinetic one must be prescribed initially\footnote{Then, our theorem shows that the monokinetic behavior propagates and is well approximated by the REM system 
\eqref{eq:REMintro}, that is, a monokinetic weak solution to the relativistic Vlasov-Maxwell system presented in Section \ref{section:mRVM}}. This prescription is direct with the assumption on the convergence \ref{item:item2mainth} of the modulated energy in Theorem \ref{unTheorem:TH1mainth}, and it typically corresponds to the initial data 
\begin{equation}
\label{eq:initialansatzcontext}
    \varphi^\varepsilon\sim e^{i\frac{\omega}{\varepsilon}}\psi.
\end{equation}
The initial time derivative, $\pi^\varepsilon$, must be given too, as the Klein-Gordon equation is a wave equation. With \eqref{eq:initialansatzcontext}, we see that the part of the kinetic energy $\frac{|(\varepsilon\nabla+i\mathscr{A})\varphi^\varepsilon|^2}{2}$ is uniformly\footnote{In general, uniformly means uniformly with respect to $\varepsilon$.} bounded as the high-frequency behavior is compensated by the presence of $\varepsilon$ in front of each derivative. Moreover, if $\psi$ is decaying, then $\varphi^\varepsilon$ is also uniformly decaying. In \ref{defi:wellprepKGMsc}, we ask for the whole energy to be uniformly decaying. In our proof, we only need the uniform decay of the density, but the argument we use to propagate the decay\footnote{We could not adapt to the Klein-Gordon equation the argument of \cite{zbMATH00482230} to propagate the uniform decay of the density of the Schrödinger equation.} without losing derivatives (and so the uniform boundedness) only works if the whole energy (that includes the density) is uniformly decaying. Nonetheless, the required weight can be taken as small as we want, and the electric energy is decaying\footnote{See \cite{zbMATH07160722} and \cite{zbMATH07370993}.} like $\frac{1}{x^4}$, thus this assumption is not very limiting. More generally, our assumption implies the uniform control of the energy for all time, this is a very natural feature.\\
Finally, the constraint \eqref{eq:maxwellconstrwellprep} corresponds to the first equation in the Maxwell constraints \eqref{eq:maxwellconstrcontext}, the second equation is automatically verified for $B^\varepsilon|_{t=0}=\nabla\times\mathscr{A}$. \\\\
For \eqref{eq:REMintro}, the local well-posedness is proved in the present paper and is gauge independent. Thus, we prescribe initial data under the form 
$(U^0,U,E,B,\rho)|_{t=0}=(\mathscr{U}^0,\mathscr{U},\mathscr{E},\mathscr{B},\varrho)$.
\begin{defi}
    \label{defi:wellprepREM}
    Let $(\mathscr{U}^0,\mathscr{U},\mathscr{E},\mathscr{B},\varrho)\in H^4\times H^4\times H^4\times H^4\times H^3$ be initial data for \eqref{eq:REMintro}. The initial data are well-prepared if the following constraints hold
    \begin{align}
           \label{eq:wellprepconstr1}
        &-\mathscr{U}^0\mathscr{U}^0+\mathscr{U}^i\mathscr{U}_i=-1,\\
         \label{eq:wellprepconstr2}
        & \mathscr{U}^0>0,\\
      	  \label{eq:wellprepconstr3}
        &\varrho\geq0,\\
           \label{eq:wellprepconstr4}
        &\nabla\cdot\mathscr{E}=-\mathscr{U}^0\varrho, 
           &\nabla\cdot\mathscr{B}=0.
    \end{align}
\end{defi}
To show that the modulated energy propagates its size, we need the solution to the limiting system \eqref{eq:REMintro} to be regular, it allows us to use Sobolev embeddings, see the proof of Proposition \ref{propal:mainpropalPropagation}. 
The normalization \eqref{eq:wellprepconstr1} implies that the vector field $\textbf{U}$ is time-like. This is a standard normalization, see \cite{zbMATH05382453} for general results on relativistic fluids. In our context, it can be linked to the so-called eikonal equation of the WKB analysis (see Appendix \ref{section:WKB}) and comes from the fact that we consider the \textbf{massive} Klein-Gordon-Maxwell equations. This normalization and the constraint \eqref{eq:wellprepconstr2} tell us that $\textbf{U}$ is future-directed and can be locally identified with the tangent vector to the trajectory of an observer going less fast than the speed of light, $c$, set to 1. The constraint \eqref{eq:wellprepconstr3} is very natural, the density must be positive. Finally, the last  two equations \eqref{eq:wellprepconstr4} correspond to the Maxwell constraints. \\
\begin{rem}
    We can already note that the properties of the well-prepared initial data are all propagated in time, we refer to Proposition \ref{propal:GWPKGM} and \ref{propal:LWPREM}.
\end{rem}
\begin{nota}
\label{nota:constantdef}
We also note\footnote{Without loss of generality, we use the same constant as in \eqref{eq:wellprepunifdecay1}.} $c_0>0$ the constant such that
   \begin{align}   
   \label{eq:wellprepineq} ||\mathscr{U}||_{H^4}+||\mathscr{U}^0||_{H^4}+||\mathscr{E}||_{H^4}+||\mathscr{B}||_{H^4}+||\varrho||_{H^3}\leq c_0. 
    \end{align}
    In general, we note $C_0$ all the constants that only depend on $c_0$ and universal constants\footnote{It can also depend on the time of existence $T$ for \eqref{eq:REMintro}, defined in \ref{propal:LWPREM}, as it only depends on $c_0$.}. In particular, $C_0$ does not depend on $\varepsilon$. The symbol $\lesssim$ refers to "$\leq$ up to a universal constant". \\
\end{nota}
\begin{nota}
\label{nota:landaudef}
   We use the big $O$ notation $O(\varepsilon^m)$ for $O_{\varepsilon\to0}(\varepsilon^m)$. We have \begin{align*}
       f^\varepsilon=O(\varepsilon^m)\Leftrightarrow \exists \varepsilon_0\in(0,1),\;\exists C_0,\;\forall0<\varepsilon<\varepsilon_0 \;|f^\varepsilon|\leq C_0\varepsilon^m.
   \end{align*} 
   In particular $f^\varepsilon=O(1)$ is equivalent to the uniform bound $|f^\varepsilon|\leq C_0$ for all $\varepsilon$ small enough.\\
\end{nota}
\begin{nota}
    We note $L^2_\delta$ for the weighted Lebesgue space defined as the closure of $C^\infty_c$ for the norm 
    \begin{align*}
        ||u||_{L^2_\delta}=||u(1+|x|^2)^{\frac{\delta}{2}}||_{L^2}.
    \end{align*}
\end{nota}
\subsection{Acknowledgment}
The author thanks Cécile Huneau for her valuable advices and guidance during the elaboration of this paper and Daniel Han-Kwan for the helpful discussions and recommendations. 
\subsection{Outline of the paper}
\label{subsection:outline}
\begin{itemize}
 \item \label{itm:sect1outline}In Section \ref{section:mainth}, we state the main Theorem \ref{unTheorem:TH1mainth} and sketch the proof.
    \item \label{itm:sect2outline}In Section \ref{section:KGM}, we give the information on the mKGM system.
      \item \label{itm:sect3outline}In Section \ref{section:REM}, we give the information on the REM system. In particular, we prove its local well-posedness.
      \item \label{itm:sect4outline}In Section \ref{section:defmod}, we give all details on the modulated stress-energy.
      \item  \label{itm:sect5outline}In Section \ref{section:InitialData}, we show the compatibility between the well-prepared assumption \ref{item:item1mainth} and the convergence assumption \ref{item:item2mainth} of the main Theorem.  
        \item \label{itm:sect6outline}In Section \ref{section:PROOF}, we give the proof of the main Theorem. 
         \item \label{itm:sect7outline}In Section \ref{section:mRVM}, we show how the solutions to REM are weak solutions to the relativistic Vlasov-Maxwell equations and how the semi-classical limit of mKGM is the relativistic Vlasov-Maxwell system. 
         \item \label{itm:sect8outline} In Section \ref{section:conclusion}, we sum up everything and give a broad discussion on the context of the present work.  
        \item \label{itm:sect8outline} In Appendix \ref{section:WKB}, we give a formal derivation of REM at the semi-classical limit of mKGM via the WKB method.   
\end{itemize}
\section{Main results}
\label{section:mainth}
In this Section, we set the main theorem, we comment it and give the main ideas of the proof. 
\begin{unTheorem}
\label{unTheorem:TH1mainth}
    Let $(\varphi^\varepsilon,\mathscr{A}^\varepsilon,\pi^\varepsilon,\mathscr{E}^\varepsilon)_{0<\varepsilon<1}$ be a family of initial data for \eqref{eq:KGMscintro} and let $(\mathscr{U}^0,\mathscr{U},\mathscr{E},\mathscr{B},\varrho)$ be initial data for \eqref{eq:REMintro}. We assume that :
\begin{enumerate}
\label{enumerate:listmainth}
    \item \label{item:item1mainth}We have well-prepared initial data from Definitions \ref{defi:wellprepKGMsc} and \ref{defi:wellprepREM}.
    \item \label{item:item2mainth}We have the assumption of convergence of the initial data at the limit, that is,  
    \begin{align}     
H^\varepsilon_{0}(\varphi^\varepsilon,\pi^\varepsilon,\mathscr{A}^\varepsilon,\mathscr{E}^\varepsilon,\mathscr{U}^0,\mathscr{U},\mathscr{E},\mathscr{B},\varrho)+|||\varphi^\varepsilon|-\sqrt{\varrho}||^2_{L^2}=O(\varepsilon^2).
    \end{align}
\end{enumerate}
Then, 
\begin{enumerate}[label=\roman*)]
 \label{itemize:pointsmainth}
	\item  \label{item:point1mainth} There exists $(\textbf{U},F,\rho)$ a solution to \eqref{eq:REMintro} with initial data $(\mathscr{U}^0,\mathscr{U},\mathscr{E},\mathscr{B},\varrho)$ and a family of solution $(\Phi^\varepsilon,\textbf{A}^\varepsilon)_{0<\varepsilon<1}$ to \eqref{eq:KGMscintro} with initial data $(\varphi^\varepsilon,\mathscr{A}^\varepsilon,\pi^\varepsilon,\mathscr{E}^\varepsilon)_{0<\varepsilon<1}$. Both solutions are defined on some time interval $[0,T]$ for $T<+\infty$ and satisfy 
 \begin{align}
     &\sum_{k=0}^4||F||_{C^{4-k}([0,T],H^k)}+\sum_{k=0}^4||\textbf{U}||_{C^{4-k}([0,T],H^k)}+\sum_{k=0}^3||\rho||_{C^{3-k}([0,T],H^k)}\leq C_0,\\
     &\nonumber \\
     &\forall\varepsilon\in(0,1),\;(\textbf{A}^\varepsilon,\Phi^\varepsilon)\in(\cap_{j=0}^5C^j([0,T],H^{5-j}))^2,\\ 
     &\nonumber\\
     & \int_{\mathbb{R}^3}{(1+|x|^2)^{\kappa}\left(\frac{|\textbf{D}^\varepsilon\Phi^\varepsilon|^2}{2}+\frac{|\Phi^\varepsilon|^2}{2}+\frac{|E^\varepsilon|^2+|B^\varepsilon|^2}{2}\right)dx}\leq C_0.
 \end{align}
    \item \label{item:point2mainth} 
For all $t\in[0,T]$, the modulated energy verifies
\begin{align}      
H^\varepsilon_0(\Phi^\varepsilon,\textbf{D}_0^\varepsilon\Phi^\varepsilon,\textbf{A}^\varepsilon,\textbf{U},F,\rho)(t)=O(\varepsilon^2),
    \end{align}
    and so
     \begin{equation}
        \begin{split}
            & \lim_{\varepsilon\to0}||\textbf{J}^\varepsilon-\textbf{U}\rho||_{L^\infty([0,T],L^{1})}+||F^\varepsilon-F||_{L^\infty([0,T],L^2)}=0,\\
     &\lim_{\varepsilon\to0}||\rho^\varepsilon-\rho||_{L^\infty([0,T],L^1)}+||\sqrt{\rho^\varepsilon}-\sqrt{\rho}||_{L^\infty([0,T],L^2)}=0,
        \end{split}
       \end{equation}    
       for $\rho^\varepsilon=|\Phi^\varepsilon|^2$ and $\textbf{J}^\varepsilon=-\Im
(\Phi^\varepsilon\overline{(\boldsymbol{D})^{\beta}\Phi^\varepsilon})$.
\end{enumerate}
\end{unTheorem}
\subsection{Comment}
\label{subsection:comment}
The theorem states that if the initial data are well-prepared (Definitions \ref{defi:wellprepKGMsc} and \ref{defi:wellprepREM}) and if the assumption of convergence \ref{item:item2mainth} holds, then, locally in time\footnote{It holds for times for which the solution to \eqref{eq:REMintro} are well-defined.}, the modulated energy convergences to $0$ at the rate $O(\varepsilon^2)$ and implies the convergence of the density $\rho^\varepsilon=|\Phi^\varepsilon|^2$, the momentum $\textbf{J}^\varepsilon$, and the electromagnetic field $F^\varepsilon$ to the density $\rho$, the momentum $\textbf{J}$ and the electromagnetic field $F$. Moreover, the dynamics of $(\textbf{U},\rho,F)$ is driven by the REM equations \eqref{eq:REMeq}. The assumption of convergence \ref{item:item2mainth} implies that the initial data for $\Phi^\varepsilon$ are monokinetic\footnote{See Section \ref{subsection:context} and Appendix \ref{section:WKB}.}, it is \textbf{only highly oscillatory in one direction} captured by $\textbf{U}$. The short statement is that the semi-classical monokinetic limit of the mKGM equations \eqref{eq:KGMscintro} is the REM system \eqref{eq:REMintro} if the initial data for mKGM have uniformly decaying energies with respect to $\varepsilon$ and if the initial data of both mKGM and REM are regular. We show in Section \ref{section:mRVM} how the semi-classical limit of mKGM can also be understood as the relativistic Vlasov-Maxwell system. \\
More generally, we note that our adaptation of the modulated energy method is gauge invariant. In particular, the quantity of interest (the momentum, the density and the electromagnetic field) are the physical quantities that do not depend on the gauge.\\ 

\subsection{Idea of the proof}
\label{subsection:Idea}
We give here a sketch of the main ideas of the proof. Firstly, we need the existence of sufficiently regular solutions to mKGM \eqref{eq:KGMscintro} and REM \eqref{eq:REMeq} on a common interval of time.\\
For the mKGM system \eqref{eq:KGMscintro}, we apply the standard result of \cite{zbMATH03781717} and \cite{zbMATH03781718} on the global well-posedness of Yang-Mills-Higgs (of which mKGM is a subcase) for large and regular initial data in the temporal gauge. Nonetheless, these results do not give information on any decay. We independently prove in Proposition \ref{propal:GWPKGM} that the uniform decay of the energy is propagated. We point out that more recent results are lowering the necessary regularity for the global well-posedness of mKGM and give more details on the decay rates and the asymptotic behavior of the solution for the cost of the smallness of the initial data, see \cite{zbMATH07370993} or \cite{zbMATH07160722}. We also refer to \cite{10.1215/S0012-7094-94-07402-4} and \cite{zbMATH05758747} for proofs of global well-posedness of low regularity large initial data in the massless case of the KGM equations (in the Coulomb gauge and the Lorenz gauge respectively).
\\ 
For the REM system \eqref{eq:REMeq}, we adapt the method of \cite{zbMATH06445364}, on the well-posedness of the Euler-Einstein equation, to get a priori energy estimates and to be able to use the classical energy method. This technique relies on both elliptic estimates and some clever use of the commutator between the derivatives along the flow and the wave operator. Indeed, there is a loss of derivative in the scheme for the REM system \eqref{eq:REMeq} that needs to be handled to obtain the desired energy estimates, as explained in the proof of Proposition \ref{propal:LWPREM}. To compensate for this apparent loss, we need to estimate with care the top order derivatives of the electromagnetic field $F$. For that, we use the fact that the derivatives along the flow ($\boldsymbol{\nabla}_\textbf{U}$) of $\textbf{U}$ and $\rho$ are better than standard derivatives with respect to the number of derivatives of $\rho$, $F$, and $\textbf{U}$ needed to control them. This is due to the transport equation in \eqref{eq:REMeq}. Then, by commuting the Maxwell equation with $\boldsymbol{\nabla}_\textbf{U}$ and using the classical energy estimates, we get heuristically that the $\boldsymbol{\nabla}_\textbf{U}\partial F$ are at the level of $\boldsymbol{\nabla}_\textbf{U}\partial\rho$ and $\boldsymbol{\nabla}_\textbf{U}\partial\textbf{U}$ and so are better than two standard derivatives of $F$. In the same way, $\Box F$ is better than two standard derivatives of $F$ due to a rearrangement \eqref{eq:waveeqREM} of the Maxwell equations \eqref{eq:REMintro} with \eqref{eq:bianchiREM}. Moreover, we know that $\Box=-\partial^2_{tt}+\Delta$ so that we control the space derivatives (with elliptic estimates via $\Delta$) if we control the time derivatives and the d'Alembertian. Because the vector field $\textbf{U}$ is time-like, we can do a similar operation and rewrite the d'Alembertian $\Box$ as $\boldsymbol{\nabla}_\textbf{U}$ derivatives plus an elliptic operator, this is done in Lemma \ref{lem:MellipticoperatorEM}. This implies that we control the space derivatives if we control the $\boldsymbol{\nabla}_\textbf{U}$ derivatives and the d'Alembertian. Once we get the better control on the space derivatives we automatically get the better control on the time derivatives $\partial_t$ using the control on the $\boldsymbol{\nabla}_\textbf{U}$ derivatives. Overall, with $\Box F$ and $\boldsymbol{\nabla}_\textbf{U}\partial F$ we control all the two standard derivatives of $F$ with better estimate. Applying this reasoning to the top order derivatives of $F$ gives the a priori energy estimate and thus the local well-posedness result. This existence result is gauge invariant as we only see the electromagnetic field and not the potential.\\
Then, the rest of the proof of the Theorem concerns the convergence of the momentum, the density, and the electromagnetic field, everything is gauge invariant here too. The convergence is given by the two following Lemmas.
\begin{lem}
     \label{lem:modpoint1Idea}
     The modulated energy $H^\varepsilon_0$ controls the convergence in the sense that 
      \begin{align}
    &||\sqrt{\rho^\varepsilon}|_{t=0}-\sqrt{\varrho}||^2_{L^2}=O(\varepsilon^2),\;\sup_{t\in[0,T]}H^\varepsilon_0(t)=O(\varepsilon^2)\implies \\
    &\lim_{\varepsilon\to0}||\textbf{J}^\varepsilon-\textbf{U}\rho||_{L^\infty([0,T],L^{1})}+||F^\varepsilon-F||_{L^\infty([0,T],L^2)}+||\rho^\varepsilon-\rho||_{L^\infty([0,T],L^1)}+||\sqrt{\rho^\varepsilon}-\sqrt{\rho}||_{L^\infty([0,T],L^2)}=0,\nonumber
    \end{align}
    where $\rho^\varepsilon=|\Phi^\varepsilon|$ and $\varrho$ is the initial datum for $\rho$. 
    This is the strong statement of the \textbf{coercivity property}. 
\end{lem}

\begin{rem}
\label{rem:modpoint1Ideaweak}
The weaker version of it corresponds to 
    \begin{align}
    &\sup_{t\in[0,T]}H^\varepsilon_0(t)=O(\varepsilon^2)\implies \exists \sqrt{\rho'}\in C^{0}([0,T],L^2) \text{ s.t.}\\
    &\lim_{\varepsilon\to0}||(\textbf{J}^\varepsilon-\textbf{U}\rho')(t)||_{L^{1}}+||(F^\varepsilon-F)(t)||_{L^2}+||(\rho^\varepsilon-\rho')(t)||_{L^1}+||(\sqrt{\rho^\varepsilon}-\sqrt{\rho'})(t)||_{L^2}=0.\nonumber
    \end{align}
\end{rem}
\begin{lem}
\label{lem:modpoint2Idea}
    The modulated energy $H^\varepsilon_0$ propagates its size in the sense that 
        \begin{align}       
    H^\varepsilon_0(0)=O(\varepsilon^2)\implies H^\varepsilon_0(t)=O(\varepsilon^2),
    \end{align}
    for $t\in[0,T]$. This is the \textbf{propagation property}.
\end{lem}
For the first Lemma, we show in Proposition \ref{propal:mainpropalCoerciviy} that $H^\varepsilon_0$ is coercive. For the second Lemma, the traditional way is to compute the time derivative of $H^\varepsilon_0$ and use the Gronwall Lemma, this is done in \cite{zbMATH01987564}, \cite{zbMATH05243173}, or \cite{zbMATH06101438} for example. Here, we were not able to perform such a calculation and close the argument. The general idea (that we already applied in \cite{salvi2024semiclassicallimitkleingordonequation}) to fix this is to construct a new "modulated stress energy" $H^\varepsilon$ based on the full stress energy tensors (defined in Section \ref{section:defmod}) that is equivalent to the previous one in terms of coercivity but with a better structure for the propagation property. In fact, using the "modulated stress energy", we can build an equivalence class\footnote{In the sense of the equivalence of norms \eqref{eq:equivrelatdefmod}.} of functionals that are all coercive. Each representative corresponds to a local reference frame, a future-directed time-like vector field (see Definition \ref{defi:HmodulenergDefmoddef}). Then, it remains to show that one of the representatives of the class satisfies the propagation property. As in \cite{salvi2024semiclassicallimitkleingordonequation}, the representative that we use to show the propagation corresponds to the vector field $\textbf{U}$, this is done in \ref{propal:mainpropalPropagation}.\\
\begin{rem}
    From a physical point of view, this tells us that the modulated energy, like the energy, can be defined in any time-like reference frame and has the same control on the system. In particular, the Theorem holds if the assumption on $H^\varepsilon_0$ is made on another representative of the equivalence class.\\
\end{rem}
At a more technical level, we point out that the coercivity is less direct to get than it is in \cite{salvi2024semiclassicallimitkleingordonequation}. To be more precise, the modulated energy\footnote{We mean by that, the equivalent class of the modulated energies.} controls in a direct way the convergence of the electromagnetic field and a type of "velocity" associated with the mKGM system. It also gives a uniform bound in $L^2$ on the spacetime gradient of the square root of the density, $\sqrt{\rho^\varepsilon}$. Moreover, from a different argument, we are able to propagate a uniform decay for $\sqrt{\rho^\varepsilon}$, i.e., a uniform bound of the $C^0([0,T],L^2_\kappa)$ norm for some $\kappa>0$. This implies, by compactness, the strong convergence of $\sqrt{\rho^\varepsilon}$ in $C^0([0,T],L^2)$ toward a function $\sqrt{\rho'}$. Once we get the convergence of the density, and because we have the convergence of the "velocity", we get the convergence of the momentum. This leads to the weak statement of the cercivity property of Remark \ref{rem:modpoint1Ideaweak}. To ensure that the density $\rho'$ is in fact $\rho$ and to get the strong statement of Lemma \ref{lem:modpoint1Idea}, we show that $\rho'$ is a weak solution to the same transport equation as $\rho$ and conclude with the unicity of the solution for such an equation knowing that the initial data are assumed to coincide in Theorem \ref{unTheorem:TH1mainth}. This method does not require a polynomial potential as in \cite{zbMATH05243173}, \cite{zbMATH06101438}, or \cite{salvi2024semiclassicallimitkleingordonequation}, but relies on the assumption of uniform decay for the initial data of $\Phi^\varepsilon$ given in Definition \ref{defi:wellprepKGMsc}. For the propagation property of Lemma \ref{lem:modpoint2Idea}, we basically show that the representative $H^\varepsilon_\textbf{U}$ of the equivalent class $H^\varepsilon$ satisfies
\begin{equation}
\label{eq:modderivIdea}
    \frac{d}{dt}H^\varepsilon_\textbf{U}\leq H^\varepsilon_\textbf{U}+O(\varepsilon^2).
\end{equation}
This relies on the structure of the equations listed in Sections \ref{section:KGM} and \ref{section:REM}, see Proposition \ref{propal:mainpropalPropagation}. 
\section{Klein-Gordon-Maxwell system}
\label{section:KGM}
We give the details on the semi-classical version of the massive Klein-Gordon-Maxwell equations
\begin{equation}
\label{eq:KGMKGM}
\begin{cases}
    \boldsymbol{\nabla}_\alpha (F^\varepsilon)^{\alpha\beta}=-\Im
(\Phi^\varepsilon\overline{(\boldsymbol{D}^\varepsilon)^{\beta}\Phi^\varepsilon})=(\textbf{J}^\varepsilon)^\beta,\\
(\boldsymbol{D}^\varepsilon)^\alpha(\boldsymbol{D}^\varepsilon)_\alpha\Phi^\varepsilon=\Phi^\varepsilon.
\end{cases}
\end{equation}
The quantity $\Phi^\varepsilon$ and $F^\varepsilon$ are the wave function and the Faraday tensor (the electromagnetic tensor). We refer to the physical context Section \ref{subsection:context}.

\begin{defi}
\label{defi:FelecKGM}
    We recall that $F^\varepsilon=d\textbf{A}^\varepsilon$ ($ F^\varepsilon_{\alpha\beta}=\boldsymbol{\nabla}_\alpha\textbf{A}^\varepsilon_\beta-\boldsymbol{\nabla}_\beta\textbf{A}^\varepsilon_\alpha$) where $\textbf{A}^\varepsilon$ is the electromagnetic four-potential with its associated connection $\textbf{D}^\varepsilon=\varepsilon\boldsymbol{\nabla}+i\textbf{A}^\varepsilon$. The Bianchi identity
\begin{equation}
    \label{eq:bianchiKGM}
    \boldsymbol{\nabla}_\alpha F^\varepsilon_{\beta\gamma}+\boldsymbol{\nabla}_\beta F^\varepsilon_{\gamma\alpha}+\boldsymbol{\nabla}_\gamma F^\varepsilon_{\alpha\beta}=0
\end{equation}
holds from the fact that $F^\varepsilon$ is an exact form. Moreover, we write
\begin{align}  
        &F^\varepsilon_{0i}=E^\varepsilon_i &^\star F^\varepsilon_{0i}=B^\varepsilon_i
\end{align}
where $^\star F^\varepsilon$ is the Hodge dual of $F^\varepsilon$ and where $E^\varepsilon$ and $B^\varepsilon$ are the electric and the magnetic field respectively. This can be written as
    \[F^\varepsilon=
\begin{pmatrix} 
	0 & E^\varepsilon_1 & E^\varepsilon_2 & E^\varepsilon_3\\
	-E^\varepsilon_1 & 0 & -B^\varepsilon_3 & B^\varepsilon_2\\
	-E^\varepsilon_2 & B^\varepsilon_3 & 0 & -B^\varepsilon_1\\
	-E^\varepsilon_3 & -B^\varepsilon_2 & B^\varepsilon_1 & 0\\
	\end{pmatrix}
	\quad
	\]
 and we have $F^\varepsilon_{\mu\nu}(F^\varepsilon)^{\mu\nu}=2|B^\varepsilon|^2-2|E^\varepsilon|^2$ so that $F^\varepsilon_{0\mu}(F^\varepsilon)_{0}^{~\mu}+\frac{1}{4}F^\varepsilon_{\mu\nu}(F^\varepsilon)^{\mu\nu}=\frac{1}{2}(|B^\varepsilon|^2+|E^\varepsilon|^2)$.\\
\end{defi}
\begin{defi}
\label{defi:momentumKGM}
    The momentum (or charge current density) associated with \eqref{eq:KGMKGM} is
\begin{equation}
\label{eq:momentumequationKGM}
\textbf{J}^\varepsilon_\alpha=-\Im
(\Phi^\varepsilon\overline{(\boldsymbol{D}^\varepsilon)_{\alpha}\Phi^\varepsilon}).
\end{equation}
We also set $J^\varepsilon$ from $\mathbb{R}^{1+3}$ to $\mathbb{R}^{3}$ as  $J_i^\varepsilon=\textbf{J}^\varepsilon_i$ and $J_0^\varepsilon=\textbf{J}^\varepsilon_0$. We have the equation
\begin{equation}
\label{eq:momentumJ0KGM}
\partial_tJ_0^\varepsilon=\nabla\cdot J^\varepsilon,
\end{equation}
written equivalently as 
\begin{equation}
\label{eq:momentumbisKGM}
\boldsymbol{\nabla}_\alpha (\textbf{J}^\varepsilon)^\alpha=0.\\
\end{equation}
  We have the conservation law 
\begin{equation}
\label{eq:conservchargeKGM}
	\frac{d}{dt}\int_{\mathbb{R}^3}{J^\varepsilon_0dx}=0.
\end{equation} 
\end{defi}
\begin{nota}
\label{nota:densityKGM}
To simplify the notation we set the density as $\rho^\varepsilon=|\Phi^\varepsilon|^2$.\\
\end{nota}
\begin{defi}
\label{defi:energytensorKGM}
The stress energy tensor associated with \eqref{eq:KGMKGM} is 
\begin{equation}
\begin{split}
        &T^{\varepsilon}_{mKGM}[\Phi^\varepsilon,\textbf{A}^\varepsilon]_{\alpha\beta}=\frac{1}{2}(\boldsymbol{D}^\varepsilon_\alpha\Phi^\varepsilon \boldsymbol{D}^\varepsilon_\beta\overline{\Phi^\varepsilon}+\overline{\boldsymbol{D}^\varepsilon_\alpha\Phi^\varepsilon}\boldsymbol{D}^\varepsilon_\beta\Phi^\varepsilon)-\frac{1}{2}g_{\alpha\beta}(\overline{\boldsymbol{D}^\varepsilon_\gamma\Phi^\varepsilon}(\boldsymbol{D}^\varepsilon)^\gamma\Phi^\varepsilon+|\Phi^\varepsilon|^2)\\
        &+F^\varepsilon_{\alpha\mu}(F^\varepsilon)_{\beta}^{~\mu}-\frac{1}{4}g_{\alpha\beta}F^\varepsilon_{\mu\nu}(F^\varepsilon)^{\mu\nu},
\end{split}
\end{equation}
its divergence is 0
\begin{equation}
\label{eq:divstressenergyKGM}
    \boldsymbol{\nabla}_\alpha (T_{mKGM}^{\varepsilon})^{\alpha}~_\beta=0,
\end{equation}
and 
\begin{equation}
\label{eq:EstressenergyKGM}
   \int{T^{\varepsilon}_{mKGM}[\Phi^\varepsilon,\textbf{A}^\varepsilon]_{00}dx}=\mathcal{E}_{mKGM}^\varepsilon[\Phi^\varepsilon],
\end{equation}
where $\mathcal{E}^{\varepsilon}_{mKGM}$ is the associated conserved energy
\begin{equation}
\label{eq:EenergyKGM}	
	\mathcal{E}^{\varepsilon}_{mKGM}[\Phi^\varepsilon,\textbf{A}^\varepsilon]=\int_{\mathbb{R}^3}{\frac{|\boldsymbol{D}^\varepsilon\Phi^\varepsilon|^2}{2}+\frac{|\Phi^\varepsilon|^2}{2}+|E^\varepsilon|^2+|B^\varepsilon|^2dx}.\\
\end{equation}
\end{defi}
We give the result on the well-posedness of \eqref{eq:KGMKGM}.
\begin{propal}
\label{propal:GWPKGM}
    Let $(\varphi^\varepsilon,\mathscr{A}^\varepsilon,\pi^\varepsilon,\mathscr{E}^\varepsilon)_{0<\varepsilon<1}$ be a family of initial data for \eqref{eq:KGMKGM}. If the initial data are well-prepared (from Definition \ref{defi:wellprepKGMsc}) then,
\begin{enumerate}
\item \label{item:GWP1}There exists a family of global solutions $(\textbf{A}^\varepsilon,\Phi^\varepsilon)_{0<\varepsilon<1}$ to \eqref{eq:REMeq} in $H^5\times H^5$ and for any $T'>0$ we have 
\begin{equation}
\label{eq:GWPweighted}
          \int_{\mathbb{R}^3}{(1+|x|^2)^{\kappa}\left(\frac{|\textbf{D}^\varepsilon\Phi^\varepsilon|^2}{2}+\frac{|\Phi^\varepsilon|^2}{2}+\frac{|E^\varepsilon|^2+|B^\varepsilon|^2}{2}\right)dx}\leq C(c_0,T')
\end{equation}
\item \label{item:GWP2}The following equations hold 
    \begin{align}
    \label{eq:GWPconstr}
        &\nabla\cdot E^\varepsilon=-(\textbf{J}^\varepsilon)^0, &\nabla\cdot B^\varepsilon=0.
    \end{align}
\end{enumerate}
\end{propal}
\begin{rem}
    In fact, for the proof of Theorem \ref{unTheorem:TH1mainth}, we only deal with the solutions to \eqref{eq:KGMKGM} on $[0,T]$ where $T>0$ is the time of existence for the solution to REM defined in \ref{propal:LWPREM}. In that case, when $T'=T$, we see that $C(c_0,T')$ can be replaced by $C_0$ from the Notation \ref{nota:constantdef}.
    \end{rem}
\begin{proof}
    For the existence, we use the result of \cite{zbMATH03781717} and \cite{zbMATH03781718} on the global well-posedness of mKGM regardless of the size of the initial data. In particular, we match the required regularity. In \cite{zbMATH03781717} and \cite{zbMATH03781718}, the authors do not consider the semi-classical version of mKGM, i.e., the constant $\varepsilon$ is set to 1, but the result applies for any $\varepsilon>0$. Indeed, we can see the solutions to \eqref{eq:KGMKGM} as rescaled solutions to \eqref{eq:KGMKGM} with $\varepsilon=1$, as explained in Section \ref{subsection:context}.\\
    For \eqref{eq:GWPweighted}, we use equation \eqref{eq:divstressenergyKGM} and compute 
    \begin{align*}
       \frac{d}{dt}\int_{\mathbb{R}^3}{(1+|x|^2)^{\kappa}(T^\varepsilon_{mKGM})_{00}dx}&=\int_{\mathbb{R}^3}{(1+|x|^2)^{\kappa}\nabla_i(T^\varepsilon_{mKGM})^i_{~0}dx}\\
       &=-\int_{\mathbb{R}^3}{2\kappa(1+|x|^2)^{\kappa-1}x_i(T^\varepsilon_{mKGM})^i_{~0}dx}\\
       &\leq \sum_{i=1}^3\int_{\mathbb{R}^3}{2\kappa(1+|x|^2)^{\kappa-1}|x||(T^\varepsilon_{mKGM})^i_{~0}|dx}\\
       &\lesssim \int_{\mathbb{R}^3}{\kappa(1+|x|^2)^{\kappa}(T^\varepsilon_{mKGM})_{00}dx},\\
    \end{align*}
    then we conclude with the Gronwall Lemma. For the equations \eqref{eq:GWPconstr}, we use the Maxwell and the Bianchi equations.
\end{proof}
 Finally, we add the following useful equations to our list.
\begin{propal}
    \label{propal:splitequationKGM}
   Let $\Phi^\varepsilon$ be a wave function then we have 
    \begin{equation}
       \label{eq:splitJJ1KGM}
       \frac{(\textbf{J}^\varepsilon)^\alpha(\textbf{J}^\varepsilon)_\alpha}{\rho^\varepsilon}=-\varepsilon^2\boldsymbol{\nabla}_\alpha\sqrt{\rho^\varepsilon}\boldsymbol{\nabla}^\alpha\sqrt{\rho^\varepsilon}+\boldsymbol{D}^\varepsilon_\alpha\Phi^\varepsilon\overline{(\boldsymbol{D}^\varepsilon)^\alpha\Phi^\varepsilon},
    \end{equation}
     \begin{equation}
         \label{eq:splitJJ2KGM}
         \frac{\textbf{J}^\varepsilon\cdot\textbf{J}^\varepsilon}{\rho^\varepsilon}=-\varepsilon^2\boldsymbol{\nabla}\sqrt{\rho^\varepsilon}\cdot\boldsymbol{\nabla}\sqrt{\rho^\varepsilon}+|\boldsymbol{D}^\varepsilon\Phi^\varepsilon|^2,
    \end{equation}
   and if $\Phi^\varepsilon$ is a solution to \eqref{eq:KGMKGM}, then the momentum $\textbf{J}^\varepsilon$ and the density $\rho^\varepsilon$ are solutions to
    \begin{equation}
         \label{eq:splitJJ3KGM}
\frac{(\textbf{J}^\varepsilon)^\alpha(\textbf{J}^\varepsilon)_\alpha}{\rho^\varepsilon}=\varepsilon^2\sqrt{\rho^\varepsilon}\Box\sqrt{\rho^\varepsilon}-\rho^\varepsilon,
            \end{equation}
            which recall the standard normalization of four-velocity vectorfields such as in \eqref{eq:REMintro}.
\end{propal}
\begin{proof}
    By direct calculations, we get  
        \begin{align*}
\frac{i\varepsilon(\Phi^\varepsilon\overline{\boldsymbol{\nabla}^\alpha\Phi^\varepsilon}-\overline{\Phi^\varepsilon}\boldsymbol{\nabla}^\alpha\Phi^\varepsilon)i\varepsilon(\Phi^\varepsilon\overline{\boldsymbol{\nabla}_\alpha\Phi^\varepsilon}-\overline{\Phi^\varepsilon}\boldsymbol{\nabla}_\alpha\Phi^\varepsilon)}{4|\Phi^\varepsilon|^2}
         &=-\frac{\varepsilon^2(\Phi^\varepsilon\overline{\boldsymbol{\nabla}_\alpha\Phi^\varepsilon}\Phi^\varepsilon\overline{\boldsymbol{\nabla}^\alpha\Phi^\varepsilon}+\overline{\Phi^\varepsilon}\boldsymbol{\nabla}_\alpha\Phi^\varepsilon\overline{\Phi^\varepsilon}\boldsymbol{\nabla}^\alpha\Phi^\varepsilon)}{4|\Phi^\varepsilon|^2}\\
         &+\varepsilon^2\frac{2\boldsymbol{\nabla}^\alpha\Phi^\varepsilon\overline{\boldsymbol{\nabla}_\alpha\Phi^\varepsilon}|\Phi^\varepsilon|^2}{4|\Phi^\varepsilon|^2}\\
     &=-\varepsilon^2(\frac{\boldsymbol{\nabla}_\alpha|\Phi^\varepsilon|\Phi^\varepsilon\overline{\boldsymbol{\nabla}^\alpha\Phi^\varepsilon}}{2|\Phi^\varepsilon|}-\frac{\overline{\Phi^\varepsilon}\boldsymbol{\nabla}^\alpha\Phi^\varepsilon\Phi^\varepsilon\overline{\boldsymbol{\nabla}_\alpha\Phi^\varepsilon}}{4|\Phi^\varepsilon|^2})\\
     &-\varepsilon^2(\frac{\boldsymbol{\nabla}_\alpha|\Phi^\varepsilon|\overline{\Phi^\varepsilon}\boldsymbol{\nabla}^\alpha\Phi^\varepsilon}{2|\Phi^\varepsilon|}-\frac{\overline{\Phi^\varepsilon}\boldsymbol{\nabla}^\alpha\Phi^\varepsilon\Phi^\varepsilon\overline{\boldsymbol{\nabla}_\alpha\Phi^\varepsilon}}{4|\Phi^\varepsilon|^2})\\
     &+\varepsilon^2\frac{\boldsymbol{\nabla}^\alpha\Phi^\varepsilon\overline{\boldsymbol{\nabla}_\alpha\Phi^\varepsilon}|\Phi^\varepsilon|^2}{2|\Phi^\varepsilon|^2}\\
     &=-\varepsilon^2\boldsymbol{\nabla}_\alpha|\Phi^\varepsilon|\boldsymbol{\nabla}^\alpha|\Phi^\varepsilon|+\varepsilon^2\boldsymbol{\nabla}_\alpha\Phi^\varepsilon\overline{\boldsymbol{\nabla}^\alpha\Phi^\varepsilon},
    \end{align*}
    so that by Definition of $\textbf{J}^\varepsilon=-\Im
(\Phi^\varepsilon\overline{\boldsymbol{D}^\varepsilon\Phi^\varepsilon})=\frac{i\varepsilon(\Phi^\varepsilon\overline{\boldsymbol{\nabla}\Phi^\varepsilon}-\overline{\Phi^\varepsilon}\boldsymbol{\nabla}\Phi^\varepsilon)}{2}+\textbf{A}|\Phi^\varepsilon|^2$ we have 
     \begin{align*}
\frac{(\textbf{J}^\varepsilon)^\alpha(\textbf{J}^\varepsilon)_\alpha}{\rho^\varepsilon}&=-\varepsilon^2\boldsymbol{\nabla}_\alpha|\Phi^\varepsilon|\boldsymbol{\nabla}^\alpha|\Phi^\varepsilon|+\varepsilon^2\boldsymbol{\nabla}_\alpha\Phi^\varepsilon\overline{\boldsymbol{\nabla}^\alpha\Phi^\varepsilon}+i\varepsilon(\Phi^\varepsilon\overline{\boldsymbol{\nabla}^\alpha\Phi^\varepsilon}-\overline{\Phi^\varepsilon}\boldsymbol{\nabla}^\alpha\Phi^\varepsilon)\textbf{A}^\varepsilon_\alpha+(\textbf{A}^\varepsilon)^\alpha\textbf{A}^\varepsilon_\alpha|\Phi^\varepsilon|^2\\
         &=-\varepsilon^2\boldsymbol{\nabla}_\alpha\sqrt{\rho^\varepsilon}\boldsymbol{\nabla}^\alpha\sqrt{\rho^\varepsilon}+\boldsymbol{D}^\varepsilon_\alpha\Phi^\varepsilon\overline{(\boldsymbol{D}^\varepsilon)^\alpha\Phi^\varepsilon},
    \end{align*}
    which is the equation \eqref{eq:splitJJ1KGM}.
    The second equation \eqref{eq:splitJJ2KGM} is obtained the same way. Then, for the third one \eqref{eq:splitJJ3KGM}, we have
    \begin{align*}
         -\varepsilon^2\boldsymbol{\nabla}_\alpha|\Phi^\varepsilon|\boldsymbol{\nabla}^\alpha|\Phi^\varepsilon|+\varepsilon^2\boldsymbol{\nabla}_\alpha\Phi^\varepsilon\overline{\boldsymbol{\nabla}^\alpha\Phi^\varepsilon}&=-\varepsilon^2\boldsymbol{\nabla}^\alpha(\boldsymbol{\nabla}_\alpha|\Phi^\varepsilon||\Phi^\varepsilon|)+\varepsilon^2\Box|\Phi^\varepsilon||\Phi^\varepsilon|\\
         &+\varepsilon^2\frac{\boldsymbol{\nabla}^\alpha(\boldsymbol{\nabla}_\alpha\Phi^\varepsilon\overline{\Phi^\varepsilon})+\boldsymbol{\nabla}^\alpha(\Phi^\varepsilon\boldsymbol{\nabla}_\alpha\overline{\Phi^\varepsilon})}{2}-\varepsilon^2\frac{\Box\Phi^\varepsilon\overline{\Phi^\varepsilon}+\Phi^\varepsilon\Box\overline{\Phi^\varepsilon}}{2}\\
         &=-\varepsilon^2\frac{\boldsymbol{\nabla}^\alpha(\boldsymbol{\nabla}_\alpha\Phi^\varepsilon\overline{\Phi^\varepsilon}+\Phi^\varepsilon\boldsymbol{\nabla}_\alpha\overline{\Phi^\varepsilon})}{2}+\varepsilon^2\Box|\Phi^\varepsilon||\Phi^\varepsilon|\\
         &+\varepsilon^2\frac{\boldsymbol{\nabla}^\alpha(\boldsymbol{\nabla}_\alpha\Phi^\varepsilon\overline{\Phi^\varepsilon}+\Phi^\varepsilon\boldsymbol{\nabla}_\alpha\overline{\Phi^\varepsilon})}{2}-\varepsilon^2\frac{\Box\Phi^\varepsilon\overline{\Phi^\varepsilon}+\Phi^\varepsilon\Box\overline{\Phi^\varepsilon}}{2}\\
          &=\varepsilon^2\sqrt{\rho^\varepsilon}\Box\sqrt{\rho^\varepsilon}-\varepsilon^2\frac{\Box\Phi^\varepsilon\overline{\Phi^\varepsilon}+\Phi^\varepsilon\Box\overline{\Phi^\varepsilon}}{2},
    \end{align*}
    so that with the previous calculation and if $\Phi^\varepsilon$ is a solution to \eqref{eq:KGMKGM} we recover
     \begin{align*}
\frac{(\textbf{J}^\varepsilon)^\alpha(\textbf{J}^\varepsilon)_\alpha}{\rho^\varepsilon}&=\varepsilon^2\sqrt{\rho^\varepsilon}\Box\sqrt{\rho^\varepsilon}-\varepsilon^2\frac{\Box\Phi^\varepsilon\overline{\Phi^\varepsilon}+\Phi^\varepsilon\Box\overline{\Phi^\varepsilon}}{2}+i\varepsilon(\Phi^\varepsilon\overline{\boldsymbol{\nabla}^\alpha\Phi^\varepsilon}-\overline{\Phi^\varepsilon}\boldsymbol{\nabla}^\alpha\Phi^\varepsilon)\textbf{A}^\varepsilon_\alpha+(\textbf{A}^\varepsilon)^\alpha\textbf{A}^\varepsilon_\alpha|\Phi^\varepsilon|^2\\ 
&=\varepsilon^2\sqrt{\rho^\varepsilon}\Box\sqrt{\rho^\varepsilon}-\rho^\varepsilon,
    \end{align*}
 which is the equation \eqref{eq:splitJJ3KGM}.    
\end{proof}
\section{Relativistic Euler Maxwell system}
\label{section:REM}
We introduce the relativistic Euler-Maxwell (REM) system
\begin{equation}
\label{eq:REMeq}
        \begin{cases}
               \boldsymbol{\nabla}_\alpha F^{\alpha\beta}=\textbf{U}^\beta\rho,\\               \textbf{U}^\alpha\boldsymbol{\nabla}_\alpha\rho+\boldsymbol{\nabla}_\alpha\textbf{U}^\alpha\rho=0,\\              \textbf{U}^\alpha\boldsymbol{\nabla}_\alpha\textbf{U}_\beta=F_{\alpha\beta}\textbf{U}^\alpha,\\
            \textbf{U}_\alpha\textbf{U}^\alpha=-1.
        \end{cases}
    \end{equation}
where $\rho$ is the charge density, $F$ the Faraday tensor and $\textbf{U}$ the four-velocity vector field. We denote by $U$ the vector field from $[0,T]\times\mathbb{R}^{3}$ to $\mathbb{R}^{3}$ with $U_i=\textbf{U}_i$, the space components, and by $U_0=\textbf{U}_0$ the time component.\\
This system is a model for massive pressureless charged fluids interacting with an electromagnetic field, this electromagnetic field has the charge current flux of the fluid as a source. \\
\begin{defi}
\label{defi:FelecREM}
    The Faraday tensor $F$ can be decomposed in an electric part $E$ and a magnetic part $B$ as done in \ref{defi:FelecKGM}. We have $F_{\alpha\beta}=\boldsymbol{\nabla}_\alpha\textbf{A}_\beta-\boldsymbol{\nabla}_\beta\textbf{A}_\alpha$ for some\footnote{The four-potential $\textbf{A}$ is determined up to a gauge, the four-potential $\textbf{A}'=\textbf{A}+\boldsymbol{\nabla}\chi$ gives the same Faraday tensor.} $\textbf{A}$, F is antisymmetric, and we also have the Bianchi identity 
    \begin{equation}
    \label{eq:bianchiREM}
            \boldsymbol{\nabla}_\alpha F_{\beta\gamma}+\boldsymbol{\nabla}_\beta F_{\gamma\alpha}+\boldsymbol{\nabla}_\gamma F_{\alpha\beta}=0.
    \end{equation}
\end{defi}
\begin{defi}
\label{defi:momentumREM}
    The momentum (or charge current density) associated with \eqref{eq:REMeq} is
\begin{equation}
\textbf{J}_\alpha=\rho \textbf{U}_\alpha,
\end{equation}
with
\begin{equation}
\label{eq:momentumREM}
\boldsymbol{\nabla}_\alpha \textbf{J}^\alpha=0.
\end{equation}
We also set $J$ from $\mathbb{R}^{1+3}$ to $\mathbb{R}^{3}$ as  $J_i=\textbf{J}_i$ and $J_0=\textbf{J}_0$. 
   We have the conservation law 
\begin{equation}
\label{eq:conservchargeREM}
	\frac{d}{dt}\int_{\mathbb{R}^3}{J_0dx}=\frac{d}{dt}\int_{\mathbb{R}^3}{U_0\rho dx}=0.
\end{equation} 
\end{defi}
\begin{defi}
\label{defi:energytensorREM}
The stress energy tensor associated with \eqref{eq:REMeq} is 
\begin{equation}
\label{eq:stressenergyREM}
    T_{REM}[\textbf{U},F,\rho]_{\alpha\beta}=\rho \textbf{U}_\alpha \textbf{U}_\beta+F_{\alpha\mu}F_{\beta}^{~\mu}-\frac{1}{4}g_{\alpha\beta}F_{\mu\nu}F^{\mu\nu},
\end{equation}
its divergence is equal to 0 
\begin{equation}
\label{eq:divstressenergyREM}
    \boldsymbol{\nabla}_\alpha (T_{REM})^{\alpha}~_\beta=0.
\end{equation}
We also have
\begin{equation}
\label{eq:EenergystressenergyREM}
   \int{T_{REM}[\textbf{U},F,\rho]_{00}dx}=\mathcal{E}_{REM}[\textbf{U},\rho,F],
\end{equation}
where $\mathcal{E}_{REM}$ is the associated conserved energy
\begin{equation}
\label{eq:EenergyREM}	
	\mathcal{E}_{REM}[\textbf{U},F,\rho]=\int_{\mathbb{R}^3}{\textbf{U}_0\textbf{U}_0\rho+|E|^2+|B|^2dx}.
\end{equation}
\end{defi}
We give a useful wave equation in what follows for the local well posedness result.
\begin{propal}
    Let $(\textbf{U},F,\rho)$ be solution to the REM system \eqref{eq:REMeq}, then we have 
    \begin{equation}
    \label{eq:waveeqREM}
        \Box F_{\alpha\beta}=\boldsymbol{\nabla}_\alpha\textbf{J}_\beta-\boldsymbol{\nabla}_\beta\textbf{J}_\alpha.
    \end{equation}
\end{propal}
\begin{proof}
    We have 
    \begin{align*}
        \boldsymbol{\nabla}^\gamma\boldsymbol{\nabla}_\gamma F_{\alpha\beta}&=\boldsymbol{\nabla}^\gamma\boldsymbol{\nabla}_\alpha F_{\gamma\beta}+\boldsymbol{\nabla}^\gamma\boldsymbol{\nabla}_\beta F_{\alpha\gamma}\\
&=\boldsymbol{\nabla}_\alpha\textbf{J}_\beta-\boldsymbol{\nabla}_\beta\textbf{J}_\alpha,
    \end{align*}
 using the first equation of \eqref{eq:REMeq}, the Bianchi equation \eqref{eq:bianchiREM} and the antisymmetry of $F$.    
\end{proof}
We give the result on the local existence of solutions to \eqref{eq:REMeq}.
\begin{propal}
\label{propal:LWPREM}
Let $(\mathscr{U}^0,\mathscr{U},\mathscr{E},\mathscr{B},\varrho)$  be initial data for the REM system \eqref{eq:REMeq}. If the initial data are well-prepared (from Definition \ref{defi:wellprepREM}) then,
\begin{enumerate}
\item \label{item:LWP1}There exists a time $T>0$ for which there exists a solution $(\textbf{U},\rho,F)$ to \eqref{eq:REMeq} with the regularity\footnote{The Definition of the constant is given in Notation \ref{nota:constantdef}} 
\begin{equation}
\label{eq:backgroundineqyproof}
     \sum_{k=0}^4||F||_{C^{4-k}([0,T],H^k)}+\sum_{k=0}^4||\textbf{U}||_{C^{4-k}([0,T],H^k)}+\sum_{k=0}^3||\rho||_{C^{3-k}([0,T],H^k)}\leq C_0.
\end{equation}
\item \label{item:LWP2}The following equations hold 
    \begin{align}
    \label{eq:LWPpropagconstr1}
        &\rho\geq0,\;\; \textbf{U}^0\geq1, \\
         \label{eq:LWPpropagconstr2}
        &\textbf{U}^\alpha\textbf{U}_\alpha=-1, \\
         \label{eq:LWPpropagconstr3}
        &\nabla\cdot E=-\textbf{U}^0\rho, &\nabla\cdot B=0.
    \end{align}
\end{enumerate}
\end{propal}
\begin{proof}
We start with point \ref{item:LWP2} and we assume that $(\textbf{U},\rho,F)$ is solution to \eqref{eq:REMeq}. If $\rho$ is a regular solution to the conservative transport 
\begin{align*}
    \textbf{U}^\alpha\boldsymbol{\nabla}_\alpha \rho+\boldsymbol{\nabla}_\alpha \textbf{U}^\alpha\rho=0\\
\end{align*}
then
\begin{equation}
\label{eq:implicitrhoREM}
    \rho(\chi(\tau,y))=\rho(0,y)e^{-\int^\tau_0{\boldsymbol{\nabla}_\alpha \textbf{U}^\alpha(\chi(\theta,y))d\theta}},
\end{equation}
where $\chi$ represents the flow lines of $\textbf{U}$, it is solution to 
\begin{equation}
    \begin{cases}
\label{eq:flowlineREM}\dot{\chi}^\alpha(\tau,y)=\textbf{U}^\alpha(\chi(\tau,y)),\\
        \chi(0,y)=(0,y).
    \end{cases}
\end{equation}
We see that $\rho$ remains positive for all time. \\
For the normalization of $\textbf{U}$, we know that 
\begin{align*}
\textbf{U}^\alpha\boldsymbol{\nabla}_\alpha\textbf{U}_\beta=F_{\alpha\beta}\textbf{U}^\alpha.
\end{align*}
This implies that 
\begin{align*}
 \textbf{U}^\alpha\boldsymbol{\nabla}_\alpha( \textbf{U}^\beta \textbf{U}_\beta)=2F_{\alpha\beta}\textbf{U}^\alpha\textbf{U}^\beta=0,
\end{align*}
the quantity $\textbf{U}^\alpha \textbf{U}_\alpha$ remains constant along the flow lines.\\
Thus, $\textbf{U}^\alpha \textbf{U}_\alpha=-1$ and $|\textbf{U}^0|\geq 1$, we recover that $\textbf{U}^0>1$ by continuity.\\
    Finally, for the equations \eqref{eq:LWPpropagconstr3} (the propagation of the constraints \eqref{eq:wellprepconstr3}) we use the Maxwell and the Bianchi equations.\\
Now we deal with point \ref{item:LWP1}, we want to use the energy method to show the local well-posedness of \eqref{eq:REMeq}. We search for closed a priori estimates, for that we use the wave equation \eqref{eq:waveeqREM}. In terms of control of derivative in $L^2$, this gives us 
\begin{align*}
&\partial F\sim \partial\textbf{U},\partial\rho,\\
&\rho\sim \rho,\partial \textbf{U},\\
&\textbf{U}\sim F,\textbf{U}.
\end{align*}
We can observe that there is a loss of derivative in the scheme. Indeed, let $N\in\mathbb{N}$ be greater than or equal to $4$. To estimate $\textbf{U}$ in a $H^N$ norm we need to control the $H^N$ norm of $F$ which requires to control the $H^N$ norm of $\rho$ which itself requires to control $\partial\textbf{U}$ in $H^N$. 
To compensate for the apparent loss of derivative, we use the same type of method as in \cite{zbMATH06445364} for the well-posedness of the Euler-Einstein equations. The goal is to control the norms 
 \begin{align}
 \label{eq:NNREM}
    \mathscr{N}_N(t)=\mathscr{N}_{<N}(t)+||\nabla_{N}F(t)||_{L^2}+||\nabla_{N-1} \partial_t F(t)||_{L^2}+||\nabla_{N-2} \partial^2_{tt} F(t)||_{L^2},
 \end{align}
 with 
 \begin{align}
  \label{eq:N-NREM}
     \mathscr{N}_{<N}(t)=\sum_{k=0}^2||\partial_t^k\textbf{U}(t)||_{H^{N-k}}+\sum_{k=0}^1||\partial_t^kF(t)||_{H^{N-k-1}}+\sum_{k=0}^1||\partial_t\rho(t)||_{H^{N-k-1}},
 \end{align}
 on some time interval. We call $C_{<N}$ the constants that only depend on $\mathscr{N}_{<N}(t)$.\\
The top order derivatives (the Nth derivatives) of $F$ are the only ones that need to be handled with care, the lowest order ones are controlled with classical energy estimates and vector field commutation. We can "lose derivatives" and be less careful if we are not at the top order. All the derivatives of $\textbf{U}$ and $\rho$ are also controlled with classical method. Thus, we can already set the following result 
 \begin{equation}
 \label{eq:energyintLWPREM}
     \mathscr{N}_{<N}(t)\leq \mathscr{N}_{<N}(0)+\int^t_0{(\mathscr{N}_{<N}(s))^2+\mathscr{N}_{<N}(s)|| \nabla_{N-2} \partial^2 F(s)||_{L^2}ds},
 \end{equation}
 where $\partial\partial F$ represents any couple of (space or time) derivatives of $F$.
 In particular, to get this inequality, we use the fact that $N-2>n/2=3/2$ is sufficiently large to use the necessary Sobolev embeddings. Now, we want to control the top order derivative of $F$. For the top order time derivative we use the derivatives along the flow of $\textbf{U}$, that is $ \boldsymbol{\nabla}_{\textbf{U}}=\textbf{U}^\alpha\boldsymbol{\nabla}_\alpha$, and for the top order space derivatives we use elliptic estimates.  \\
 Firstly, we rewrite the equations for $F$, $\partial\textbf{U}$ and $\partial\rho$ schematically, we have 
 \begin{equation}
 \nonumber
     \begin{cases}
     \Box F=\partial\textbf{U}\rho+\textbf{U}\partial\rho,\\  \boldsymbol{\nabla}_{\textbf{U}}\partial\textbf{U}=\partial\textbf{U}\partial\textbf{U}+\partial\textbf{U}F+\textbf{U}\partial F,\\    \boldsymbol{\nabla}_{\textbf{U}}\partial\rho=\partial\textbf{U}\partial\rho+\partial\partial\textbf{U}\rho,
     \end{cases}
 \end{equation}
 from \eqref{eq:waveeqREM} and \eqref{eq:REMeq}. Then, we differentiate the wave equation \eqref{eq:waveeqREM} for $F$ with respect to the space derivatives $N-2=m$ times and with respect to $ \boldsymbol{\nabla}_{\textbf{U}}$ one time, we obtain
 \begin{align}
\label{eq:commutusEm}
     \Box  \boldsymbol{\nabla}_{\textbf{U}}\nabla_{m}F&\approx[\Box,  \boldsymbol{\nabla}_{\textbf{U}}]\nabla_{m}F+  \boldsymbol{\nabla}_{\textbf{U}}(\sum_{k,l\in\mathbb{N},k+l=m}(\nabla_k\partial\textbf{U}\nabla_l\rho+\nabla_k\textbf{U}\nabla_l\partial\rho))\nonumber\\
     &\approx \partial\textbf{U}\nabla_{m}\partial^2F+ \partial^2\textbf{U}\nabla_{m}\partial F+l.o.t.\nonumber\\
     &\approx \partial\textbf{U}\nabla_{m}\partial^2F+l.o.t.,
 \end{align}
 where the $l.o.t.$ designates terms that are controllable in $L^2$ by the $\mathscr{N}_{<N}$ norm. In particular, we use the transport equations to replace the $ \boldsymbol{\nabla}_{\textbf{U}}$ falling on $\rho$ and $\textbf{U}$ by their respective RHS and we observe that the top order for the derivatives falling on $\textbf{U}$ is higher than for $\rho$. This aligns with the $\mathscr{N}_{<N}$ norm in equation \eqref{eq:N-NREM}.\\
 Then, we multiply the equation \eqref{eq:commutusEm} by $ \boldsymbol{\nabla}_{\textbf{U}}\partial_t\nabla_{m}F$ and with standard methods we get schematically the estimate
 \begin{equation}
 \label{eq:controldUderivEM}
    || \boldsymbol{\nabla}_{\textbf{U}} \partial\nabla_m F(t)||_{L^2}\lesssim  || \boldsymbol{\nabla}_{\textbf{U}}\partial\nabla_m F(0)||_{L^2}+\int^t_0{||\nabla_{m} \partial^2 F(s)||_{L^2}\mathscr{N}_{<N}(s)+\mathscr{N}_{<N}(s)^2+\mathscr{N}_{<N}(s)^3ds}.
 \end{equation} 
 This gives us the control on the top order derivatives (Nth derivatives) of $F$ that contains the direction $ \boldsymbol{\nabla}_{\textbf{U}}$ for the cost of only $N-1$ derivatives of $\rho$ (one less derivatives than with standard method) if we have the control of the top order space derivative of $F$. We show how to have it in what follows.\\
 Firstly, we notice that 
 \begin{align*}
     \partial_t=\frac{ \boldsymbol{\nabla}_{\textbf{U}}}{U^0}-\frac{U^i\nabla_i}{U^0}
 \end{align*}
 and so 
 \begin{equation}
  \label{eq:laplacewaveEM}
     \Box F=(-\partial^2_{tt}+\Delta)F=(\delta^{ij}-\frac{U^iU^j}{U^0U^0})\nabla_i\nabla_jF-\frac{1}{U^0} \boldsymbol{\nabla}_{\textbf{U}}\partial_tF+\frac{U^i}{(U^0)^2} \boldsymbol{\nabla}_{\textbf{U}}\nabla_iF.\\
 \end{equation}
 Secondly, we have the following Lemma from \cite{zbMATH06445364} on elliptic estimates.
 \begin{lem}
     \label{lem:MellipticoperatorEM}
     The operator $M^{ij}\nabla_i\nabla_j$ with $M^{ij}=(\delta^{ij}-\frac{U^iU^j}{U^0U^0})$ is a uniformly elliptic second order operator on $\mathbb{R}^3$. Moreover, we have the inequality 
     \begin{equation}
         \label{eq:ineqellipticEM}
         ||f||_{H^2}\leq K_{<N}(t)||M^{ij}\nabla_i\nabla_jf||_{L^2},
     \end{equation}
     for a certain constant $K_{<N}$ that only depends on $\mathscr{N}_{<N}(t)$.
 \end{lem}
 \begin{proof}
     We remark that the eigenvalues of $M$ are $1$, $1$ and $\frac{1}{(U^0)^2}$ and are uniformly bounded from above and from below on a fixed time interval. Indeed, we have $U^0(t)\geq1$ and $||U^0(t)||_{C^0}\leq \mathscr{N}_{<N}(t)$. Thus, $M^{ij}\nabla_i\nabla_j$ is a uniform second order elliptic operator. The inequality follows from Theorem 8.12 of \cite{zbMATH06445364}.\\
 \end{proof}
Up to a commutation with $\nabla_m$ for $m=N-2$, the equation \eqref{eq:laplacewaveEM} together with \eqref{eq:ineqellipticEM} gives us 
\begin{equation}
    \label{eq:controlspacederivEM}
   \sum_{i,j=1}^{3} ||\nabla_i\nabla_j\nabla_mF(t)||_{L^2}\leq K_{<N}(t)(||\Box \nabla_mF(t)||_{L^2}+\mathscr{N}_{<N}(t)^2+\mathscr{N}_{<N}(t)^3+|| \boldsymbol{\nabla}_{\textbf{U}}\nabla_m\partial F(t)||_{L^2}(1+\mathscr{N}_{<N}(t))),
\end{equation}
with $K_{<N}(t)$ from \ref{lem:MellipticoperatorEM}. Then, using the wave equation \eqref{eq:waveeqREM} we get 
\begin{equation}
    \label{eq:controlspacederivfinalEM}
   \sum_{i,j=1}^{3} ||\nabla_i\nabla_j\nabla_mF(t)||_{L^2}\leq C_{<N}(t)(1+|| \boldsymbol{\nabla}_{\textbf{U}}\nabla_m\partial F(t)||_{L^2}),
\end{equation}\\
that is, the control of the top order space derivatives by the $\boldsymbol{\nabla}_{\textbf{U}}\partial$ derivatives.
In particular, the term $||\Box \nabla_mF(t)||_{L^2}$ only requires the control of $m+1=N-1$ derivatives of $\textbf{U}$ and $\rho$ which is given by $\mathscr{N}_N$.\\
Then, from \eqref{eq:LWPpropagconstr1} and \eqref{eq:N-NREM} we observe that 
\begin{align*}
||\partial_t\nabla\nabla_mF(t)||_{L^2}&\leq||\frac{ \boldsymbol{\nabla}_{\textbf{U}}}{U^0}\nabla\nabla_mF(t)||_{L^2}+||\frac{U^j\nabla_j}{U^0}\nabla\nabla_mF(t)||_{L^2}\\
&\leq ||\boldsymbol{\nabla}_{\textbf{U}}\nabla\nabla_mF(t)||_{L^2}+C_{<N}(t)||\nabla\nabla\nabla_mF(t)||_{L^2}\\
\end{align*}
and 
\begin{align*}
||\partial_t\partial_t\nabla_mF(t)||_{L^2}&\leq||\frac{ \boldsymbol{\nabla}_{\textbf{U}}}{U^0}\partial_t\nabla_mF(t)||_{L^2}+||\frac{U^j\nabla_j}{U^0}\partial_t\nabla_mF(t)||_{L^2}\\
&\leq ||\boldsymbol{\nabla}_{\textbf{U}}\partial_t\nabla_mF(t)||_{L^2}+C_{<N}(t)||\nabla\partial_t\nabla_mF(t)||_{L^2}\\
&\leq C_{<N}(t)||\boldsymbol{\nabla}_{\textbf{U}}\partial\nabla_mF(t)||_{L^2}+C_{<N}(t)||\nabla\nabla\nabla_mF(t)||_{L^2}\\
\end{align*}
so that all the top order derivatives are controlled by the top order space derivatives (via \eqref{eq:controlspacederivfinalEM}) and $\boldsymbol{\nabla}_{\textbf{U}}\partial$ (via \eqref{eq:controldUderivEM}), i.e.,
\begin{align}
||\partial^2\nabla_{m}  F(s)||_{L^2}\leq C_{<N}||\boldsymbol{\nabla}_{\textbf{U}}\partial\nabla_mF(t)||_{L^2}+C_{<N}||\nabla\nabla\nabla_mF(t)||_{L^2}.
\end{align}
With the previous calculations, we can control the RHS term $|| \nabla_{m}\partial^2 F(s)||_{L^2}$ in the integral of \eqref{eq:energyintLWPREM} and so we can close our estimates for $\mathscr{N}_N(t)$. This implies point \eqref{item:LWP1}. \\
This ends the proof of Proposition \ref{propal:LWPREM}.
\end{proof}
\section{Modulated stress-energy method}
\label{section:defmod}
We give useful Definitions for the equivalent class of modulated energy (or modulated stress-energy). It uses the same principles as in \cite{salvi2024semiclassicallimitkleingordonequation} but also controls the electromagnetic field.
\begin{propal}
\label{propal:tensordiffdefmod}Let $(\Phi^\varepsilon,\textbf{A}^\varepsilon)_{0<\varepsilon<1}$ be solutions to \eqref{eq:KGMKGM} given by Proposition \ref{propal:GWPKGM} and $(\textbf{U},F,\rho)$ be a solution to \eqref{eq:REMeq} given by Proposition \ref{propal:LWPREM}, let $T^\varepsilon_{mKGM}$ and $T_{REM}$ be their respective stress energy tensor field from Definitions \ref{defi:energytensorKGM} and \ref{defi:energytensorREM}, then we have
    \begin{equation}
    \label{eq:tensordiffeqdefmod}
        (T^\varepsilon_{mKGM})_{\alpha\beta}-(T_{REM})_{\alpha\beta}=h^\varepsilon_{\alpha\beta}+I^\varepsilon_{\alpha\beta},
    \end{equation}
with 
\begin{equation}
\label{eq:hdefdefmod}
\begin{split}
    &h_{\alpha\beta}^\varepsilon=\frac{1}{2}((\boldsymbol{D}^\varepsilon_\alpha-i\textbf{U}_\alpha)\Phi^\varepsilon\overline{(\boldsymbol{D}^\varepsilon_\beta-i\textbf{U}_\beta)\Phi^\varepsilon}+\overline{(\boldsymbol{D}^\varepsilon_\alpha-i\textbf{U}_\alpha)\Phi^\varepsilon}(\boldsymbol{D}^\varepsilon_\beta-i\textbf{U}_\beta)\Phi^\varepsilon)-\frac{1}{2}g_{\alpha\beta}((\boldsymbol{D}^\varepsilon_\gamma-i\textbf{U}_\gamma)\Phi^\varepsilon\overline{((\boldsymbol{D}^\varepsilon)^\gamma-i\textbf{U}^\gamma )\Phi^\varepsilon})\\
&+(F^\varepsilon_{\alpha\mu}-F_{\alpha\mu})((F^\varepsilon)_{\beta}^{~\mu}-F_{\beta}^{~\mu})-\frac{1}{4}g_{\alpha\beta}((F^\varepsilon_{\mu\nu}-F_{\mu\nu})((F^\varepsilon)^{\mu\nu}-F^{\mu\nu})),
\end{split}
\end{equation}
and 
\begin{equation}
\label{eq:Idefdefmod}
\begin{split}
    &I^\varepsilon_{\alpha\beta}=-\textbf{U}_\alpha \textbf{U}_\beta (\rho-|\Phi^\varepsilon|^2)+\textbf{U}_\alpha (\textbf{J}^\varepsilon_\beta-\textbf{U}_\beta|\Phi^\varepsilon|^2)+\textbf{U}_\beta (\textbf{J}^\varepsilon_\alpha-\textbf{U}_\alpha |\Phi^\varepsilon|^2)+F_{\alpha\mu}((F^\varepsilon)_{\beta}^{~\mu}-F_{\beta}^{~\mu})+(F^\varepsilon_{\alpha\mu}-F_{\alpha\mu})F_{\beta}^{~\mu}\\
    &-g_{\alpha\beta}(\textbf{J}^\varepsilon_\gamma-\textbf{U}_\gamma |\Phi^\varepsilon|^2)\textbf{U}^\gamma-\frac{1}{2}g_{\alpha\beta}(F_{\mu\nu}((F^\varepsilon)^{\mu\nu}-F^{\mu\nu})).\\
\end{split}
\end{equation}
\end{propal}
\begin{proof}
   For the kinetic energy part we get
\begin{align*}
	&\frac{1}{2}(\boldsymbol{D}^\varepsilon_\alpha\Phi^\varepsilon \boldsymbol{D}^\varepsilon_\beta\overline{\Phi^\varepsilon}+\boldsymbol{D}^\varepsilon_\alpha\overline{\Phi^\varepsilon}\boldsymbol{D}^\varepsilon_\beta\Phi^\varepsilon)-\frac{1}{2}g_{\alpha\beta}(\boldsymbol{D}^\varepsilon_\gamma\overline{\Phi^\varepsilon}(\boldsymbol{D}^\varepsilon)^\gamma\Phi^\varepsilon+|\Phi^\varepsilon|^2)-\rho\textbf{U}_\alpha \textbf{U}_\beta\\
&=\frac{1}{2}(\boldsymbol{D}^\varepsilon_\alpha\Phi^\varepsilon \boldsymbol{D}^\varepsilon_\beta\overline{\Phi^\varepsilon}+\boldsymbol{D}^\varepsilon_\alpha\overline{\Phi^\varepsilon}\boldsymbol{D}^\varepsilon_\beta\Phi^\varepsilon)+\textbf{U}_\alpha \textbf{U}_\beta |\Phi^\varepsilon|^2-2\textbf{U}_\alpha \textbf{U}_\beta |\Phi^\varepsilon|^2-\textbf{U}_\alpha \textbf{U}_\beta (\rho-|\Phi^\varepsilon|^2)\\
&-\frac{1}{2}g_{\alpha\beta}(\boldsymbol{D}^\varepsilon_\gamma\overline{\Phi^\varepsilon}(\boldsymbol{D}^\varepsilon)^\gamma\Phi^\varepsilon+|\Phi^\varepsilon|^2-(\textbf{U}^\gamma \textbf{U}_\gamma+1)|\Phi|^2)\\
&=\frac{1}{2}((\boldsymbol{D}^\varepsilon_\alpha-i\textbf{U}_\alpha)\Phi^\varepsilon\overline{(\boldsymbol{D}^\varepsilon_\beta-i\textbf{U}_\beta)\Phi^\varepsilon}+\overline{(\boldsymbol{D}^\varepsilon_\alpha-i\textbf{U}_\alpha)\Phi^\varepsilon}(\boldsymbol{D}^\varepsilon_\beta-i\textbf{U}_\beta)\Phi^\varepsilon)+\textbf{J}^\varepsilon_\alpha \textbf{U}_\beta+\textbf{J}^\varepsilon_\beta \textbf{U}_\alpha-2\textbf{U}_\alpha \textbf{U}_\beta |\Phi^\varepsilon|^2\\
&-\textbf{U}_\alpha\textbf{U}_\beta (\rho-|\Phi^\varepsilon|^2)-\frac{1}{2}g_{\alpha\beta}((\boldsymbol{D}^\varepsilon_\gamma-i\textbf{U}_\gamma)\Phi^\varepsilon\overline{((\boldsymbol{D}^\varepsilon)^\gamma-i\textbf{U}^\gamma )\Phi^\varepsilon}+2\textbf{J}^\varepsilon_\gamma \textbf{U}^\gamma -2\textbf{U}_\gamma \textbf{U}^\gamma |\Phi^\varepsilon|^2),\\
\end{align*}
where we use the normalization \eqref{eq:LWPpropagconstr2}, and for the electromagnetic part we get
\begin{align*}
	&F^\varepsilon_{\alpha\mu}(F^\varepsilon)_{\beta}^{~\mu}-\frac{1}{4}g_{\alpha\beta}(F^\varepsilon_{\mu\nu}(F^\varepsilon)^{\mu\nu})-F_{\alpha\mu}F_{\beta}^{~\mu}+\frac{1}{4}g_{\alpha\beta}(F_{\mu\nu}F^{\mu\nu})\\
 &=(F^\varepsilon_{\alpha\mu}-F_{\alpha\mu})((F^\varepsilon)_{\beta}^{~\mu}-F_{\beta}^{~\mu})-\frac{1}{4}g_{\alpha\beta}((F^\varepsilon_{\mu\nu}-F_{\mu\nu})((F^\varepsilon)^{\mu\nu}-F^{\mu\nu}))\\
&+F_{\alpha\mu}((F^\varepsilon)_{\beta}^{~\mu}-F_{\beta}^{~\mu})+(F^\varepsilon_{\alpha\mu}-F_{\alpha\mu})F_{\beta}^{~\mu}-\frac{1}{2}g_{\alpha\beta}(F_{\mu\nu}((F^\varepsilon)^{\mu\nu}-F^{\mu\nu})),
 \end{align*}
 which gives us the $h^\varepsilon$ and $I^\varepsilon$ decomposition of $T^\varepsilon_{mKGM}-T_{REM}$.
\end{proof}
\begin{defi}
\label{defi:xiXidefmod}
     We write\footnote{We write $\boldsymbol{\xi}$ and $\Xi$ instead of $\boldsymbol{\xi}^\varepsilon$ and $\Xi^\varepsilon$ to lighten the notation.} $\boldsymbol{\xi}_\alpha=(\boldsymbol{D}^\varepsilon_\alpha-i\textbf{U}_\alpha)\Phi^\varepsilon$ and $\Xi(F^\varepsilon,F)_{\alpha\beta}=(F^\varepsilon_{\alpha\beta}-F_{\alpha\beta})$ so that 
\begin{equation}
\label{eq:defh2Cyproof}
h_{\alpha\beta}^\varepsilon=\frac{1}{2}(\boldsymbol{\xi}_\alpha\overline{\boldsymbol{\xi}_\beta}+\overline{\boldsymbol{\xi}_\alpha}\boldsymbol{\xi}_\beta)-\frac{1}{2}g_{\alpha\beta}(\boldsymbol{\xi}_\gamma\overline{\boldsymbol{\xi}^\gamma})+\Xi_{\alpha\mu}\Xi_{\beta}^{~\mu}-\frac{1}{4}g_{\alpha\beta}(\Xi_{\mu\nu}\Xi^{\mu\nu}).
\end{equation}
We also use the electric and magnetic field notation for $\Xi$, that is \[\Xi=
\begin{pmatrix} 
	0 & \mathcal{E}_1 &\mathcal{E}_2 &\mathcal{E}_3\\
	-\mathcal{E}_1 & 0 & -\mathcal{B}_3 & \mathcal{B}_2\\
	-\mathcal{E}_2 & \mathcal{B}_3 & 0 & -\mathcal{B}_1\\
	-\mathcal{E}_3 & -\mathcal{B}_2 & \mathcal{B}_1 & 0\\
	\end{pmatrix}
	\quad
	\]
 for\footnote{We also drop the index $\varepsilon$ to lighten the notation.} $\mathcal{E}=E^\varepsilon-E$ and $\mathcal{B}=B^\varepsilon-B$.\\
\end{defi}
\begin{defi}
\label{defi:etadefmod}
    For $h^\varepsilon_{\alpha\beta}$ defined as in \ref{propal:tensordiffdefmod}, we define $\eta(\textbf{X})^\varepsilon=h^\varepsilon_{\alpha0}\textbf{X}^\alpha$. It corresponds to looking at the modulated energy flux in the reference frame of $\textbf{X}$.\\
\end{defi}
\begin{defi}
\label{defi:acceptablefieldmod}
 Let $\textbf{X}$ be a vector field on $\mathbb{R}^{3+1}$, we say that it is acceptable if it is a time-like future-directed vector field and if it satisfies $|\textbf{X}|\leq\frac{1}{\nu}$, $\textbf{X}^0\geq\nu$ and $-\textbf{X}^\alpha \textbf{X}_\alpha\geq\nu$ for some $\nu>0$.
 \\
    \end{defi}
\begin{propal}
\label{propal:c1c2defmod}
    Let $\textbf{X}_1$ and $\textbf{X}_2$ be acceptable vector fields from Definition \ref{defi:acceptablefieldmod} with their respective $\nu_1$ and $\nu_2$ constants, then there exist $c_1(\nu_1,\nu_2)$ and $c_2(\nu_1,\nu_2)$ such that 
\begin{equation}
\label{eq:c1c2defmod}
		c_1\eta(\textbf{X}_1)<\eta(\textbf{X}_2)<c_2\eta(\textbf{X}_1).
\end{equation}
\end{propal}
\begin{proof}
First, we start by showing
\begin{align*}
    c_3(\nu)\eta(\partial_t)<\eta(\textbf{X})<c_4(\nu)\eta(\partial_t),
\end{align*}
for any $\textbf{X}$ acceptable. We directly have that
\begin{align*}
\eta(\partial_t)=h^\varepsilon_{00}=\frac{|\boldsymbol{\xi}|^2}{2}+\frac{|\mathcal{E}|^2+|\mathcal{B}|^2}{2},
\end{align*}
 where we use the notations of Definition \ref{defi:xiXidefmod}. This directly implies the existence of $c_4(\textbf{X})$ considering the fact that $\textbf{X}$ is uniformly bounded by $\frac{1}{\nu}$. Then, we have 
\begin{align*}
h_{0\beta}\textbf{X}^\beta&=\frac{|\boldsymbol{\xi}|^2+|\mathcal{B}|^2+|\mathcal{E}|^2}{2}\textbf{X}^0+\frac{\boldsymbol{\xi}_0\overline{\boldsymbol{\xi}_i}+\overline{\boldsymbol{\xi}_0}\boldsymbol{\xi}_i}{2}\textbf{X}^i+\textbf{X}^i\Xi_{i j}\Xi_{0j}\\
&=\frac{|\boldsymbol{\xi}|^2+|\mathcal{B}|^2+|\mathcal{E}|^2}{2}\textbf{X}^0+\frac{\boldsymbol{\xi}_0\overline{\xi_i}+\overline{\boldsymbol{\xi}_0}\xi_i}{2}X^i+
X^i(\mathcal{E}\times \mathcal{B})_i\\
&\geq \frac{|\boldsymbol{\xi}|^2+|\mathcal{B}|^2+|\mathcal{E}|^2}{2}\textbf{X}^0-|\boldsymbol{\xi}^0||\xi||X|-|X|| \mathcal{E}|| \mathcal{B}|\\
&\geq \frac{|\boldsymbol{\xi}|^2+|\mathcal{B}|^2+|\mathcal{E}|^2}{2}(X^0-|X|)\\
&\geq c_1(\textbf{X})h_{00}^\varepsilon,
\end{align*}
where we use the Young inequality. We have the existence of $c_3(\textbf{X})$. \\
This leads to 
\begin{equation}
   \frac{c_3(\nu_1)}{c_4(\nu_2)}\eta(\textbf{X}_2)<c_3(\nu_1)\eta(\partial_t)<\eta(\textbf{X}_1)<c_4(\nu_1)\eta(\partial_t)< \frac{c_4(\nu_1)}{c_3(\nu_2)}\eta(\textbf{X}_2),
\end{equation}
which is equivalent to \eqref{eq:c1c2defmod}.
\end{proof}
\begin{defi}
\label{defi:HmodulenergDefmoddef}
    Let $(\Phi^\varepsilon,\textbf{A}^\varepsilon)_{0<\varepsilon<1}$ be solutions to \eqref{eq:KGMKGM} given by Proposition \ref{propal:GWPKGM} and $(\textbf{U},F,\rho)$ be a solution to \eqref{eq:REMeq} given by Proposition \ref{propal:LWPREM} on $[0,T]$. For all $\textbf{X}$ time-like future-directed, we define for all $t\in[0,T]$
    \begin{equation}
        H^\varepsilon_{\textbf{X}}(\Phi^\varepsilon,\textbf{D}^\varepsilon_0\Phi^\varepsilon,\textbf{A}^\varepsilon,\textbf{U},F,\rho)(t):=\int_{\mathbb{R}^3}{\eta^\varepsilon(\textbf{X})dx}.
    \end{equation}
    In particular, we have
     \begin{align*}
    \label{eq:Hepsfulldefmod} 
    H^\varepsilon_{\textbf{U}}(t)=\int_{\mathbb{R}^3}{\textbf{U}^\alpha\frac{1}{2}((\boldsymbol{D}^\varepsilon_\alpha-i\textbf{U}_\alpha)\Phi^\varepsilon\overline{(\boldsymbol{D}^\varepsilon_0-i\textbf{U}_0)\Phi^\varepsilon}+\overline{(\boldsymbol{D}^\varepsilon_\alpha-i\textbf{U}_\alpha)\Phi^\varepsilon}(\boldsymbol{D}^\varepsilon_0-i\textbf{U}_0)\Phi^\varepsilon)dx}\\
    +\int_{\mathbb{R}^3}{\textbf{U}^\alpha(F^\varepsilon_{\alpha\mu}-F_{\alpha\mu})((F^\varepsilon)_{0}^{~\mu}-F_{0}^{~\mu})-\frac{1}{2}\textbf{U}_0((\boldsymbol{D}^\varepsilon_\gamma-i\textbf{U}_\gamma)\Phi^\varepsilon\overline{((\boldsymbol{D}^\varepsilon)^\gamma-i\textbf{U}^\gamma )\Phi^\varepsilon}+\frac{1}{2}(F^\varepsilon_{\mu\nu}-F_{\mu\nu})((F^\varepsilon)^{\mu\nu}-F^{\mu\nu}))dx},
    \end{align*}
    and 
    \begin{align*}
H^\varepsilon_0(t):=H^\varepsilon_{\partial_t}=\int_{\mathbb{R}^3}{\frac{|(\boldsymbol{D}^\varepsilon_\alpha-i\textbf{U}_\alpha)\Phi^\varepsilon|^2}{2}+\frac{|E^\varepsilon-E|^2+|B^\varepsilon-B|^2}{2}dx}.
    \end{align*}
\end{defi}
\begin{defi}
\label{defi:canonicalequivdefmod}
   We define the equivalence class $[H^\varepsilon_{\textbf{X}}]=\{H^\varepsilon_{\textbf{Y}}|H^\varepsilon_{\textbf{Y}}\sim H^\varepsilon_{\textbf{X}}\}$, for the equivalence relation 
   \begin{equation}
   \label{eq:equivrelatdefmod}
      H^\varepsilon_{\textbf{X}_1}\sim    H^\varepsilon_{\textbf{X}_2}\Leftrightarrow  \exists c_1,\;c_2,\; \forall t\in[0,T],\;c_1H^\varepsilon_{\textbf{X}_1}(t)<H^\varepsilon_{\textbf{X}_2}(t)<c_2H^\varepsilon_{\textbf{X}_1}(t).
   \end{equation}
Finally, we define $H^\varepsilon$ as the canonical representative of the class of $H^\varepsilon_0$, with
\begin{equation}
\label{eq:canonicalequivdefmod}
    [H^\varepsilon]=[H^\varepsilon_0].
\end{equation}
\end{defi}
\begin{nota}
\label{nota:canonicalequivdefmod}
    In what follows, the modulated energy designates the equivalence class of $H^\varepsilon$.\\
\end{nota}
\begin{propal}
\label{propal:acceptimplyequivdefmod}
    For any acceptable $\textbf{X}$ from Definition \ref{defi:acceptablefieldmod}, we have 
    \begin{align}
        [H^\varepsilon_{\textbf{X}}]=H^\varepsilon.
    \end{align}
\end{propal}
\begin{proof}
    By direct calculations using Proposition \ref{propal:c1c2defmod} and Definition \ref{defi:canonicalequivdefmod}.
\end{proof}
We make some Remarks on these Definitions and Propositions.
\begin{rem}
\label{rem:H0candidatedefmod}
    The modulate energy $H^\varepsilon_0$ is the direct adaptation of the usual energy modulated energy of \cite{zbMATH01987564}, \cite{zbMATH05243173} or \cite{zbMATH06101438} (in the reference frame of $\partial_t$) and the modulated energy of the present paper, $H^\varepsilon$, is one of its equivalents.\\
\end{rem}
\begin{rem}
\label{rem:kineticelectrodefmod}
   The modulated energy $H^\varepsilon$ can be separated in two different parts respectively equivalent to, the kinetic part 
    \begin{equation}
    \label{eq:kineticdefmod}
        K^\varepsilon_0=\int_{\mathbb{R}^3}{\frac{|(\boldsymbol{D}^\varepsilon_\alpha-i\textbf{U}_\alpha)\Phi^\varepsilon|^2}{2}dx}
    \end{equation}
    and the electromagnetic part 
       \begin{equation}
        \label{eq:electrodefmod}
        P^\varepsilon_0=\int_{\mathbb{R}^3}{\frac{|E^\varepsilon-E|^2+|B^\varepsilon-B|^2}{2}dx}.
    \end{equation}
\end{rem}

\section{Compatibility of the assumptions on the initial data}
\label{section:InitialData}
This Section is not mandatory for the proof of Theorem \ref{unTheorem:TH1mainth}, we give here some Propositions about the compatibility between the assumptions of the Theorem. These assumptions concern the initial data $(\mathscr{U}^0,\mathscr{U},\mathscr{E},\mathscr{B},\varrho)$ for the REM system \eqref{eq:REMeq} and $(\varphi^\varepsilon,\pi^\varepsilon,\mathscr{A}^\varepsilon,E^\varepsilon)_{0<\varepsilon<1}$ for mKGM \eqref{eq:KGMKGM} and their proximity, point \ref{item:item2mainth} of Theorem \ref{unTheorem:TH1mainth}. \\
In the next two Propositions we demonstrate that the main constraints in the well-preparedness Definitions \ref{defi:wellprepKGMsc} and \ref{defi:wellprepREM} of the initial data follow\footnote{It is only the case for the Maxwell constraints. This brings a bit of redundancy in the Theorem but makes things simpler to differentiate.} from or are acceptable under the assumption of convergence, point \ref{item:item2mainth}.
The first Proposition states that under the assumption of convergence, the property of being a solution to the Maxwell constraints of mKGM can be passed to the limiting system. 
\begin{propal}
\label{propal:proxiimplycontsraint1Init}
    Under the assumption of convergence \ref{item:item2mainth} of Theorem \ref{unTheorem:TH1mainth} and the well-preparedness \ref{defi:wellprepKGMsc} of\\ $(\varphi^\varepsilon,\pi^\varepsilon,\mathscr{A}^\varepsilon,E^\varepsilon)_{0<\varepsilon<1}$ the following equations hold for the initial data $(\mathscr{U}^0,\mathscr{U},\mathscr{E},\mathscr{B},\varrho)$
    \begin{align}
          &\nabla\cdot\mathscr{E}=-\mathscr{U}^0\varrho, &    \nabla\cdot\mathscr{B}=0,
    \end{align}
    in the sense of distribution\footnote{Or in the classical sense if the initial data are regular enough.} at time $t=0$. The limiting electromagnetic field is solution to the constraint equation. 
\end{propal}
\begin{proof}
The result is obtained by linearity of the Maxwell equation in \eqref{eq:REMeq}. To give more details, for all $\varphi\in C_c^\infty(\mathbb{R}^3)$ 
    \begin{align*}
        \int_{\mathbb{R}^3}{\mathscr{E}\nabla\varphi dx}=\int_{\mathbb{R}^3}{(\mathscr{E}-\mathscr{E}^\varepsilon)\nabla\varphi dx}-\int_{\mathbb{R}^3}{(\mathscr{U}^0\varrho-(-\Im(\varphi^\varepsilon\overline{\pi^\varepsilon}))\varphi dx}+\int_{\mathbb{R}^3}{\mathscr{U}^0\varrho\varphi dx},
    \end{align*}
    where $-\Im(\varphi^\varepsilon\overline{\pi^\varepsilon})=(\textbf{J}^\varepsilon)^0|_{t=0}$. Then, with arguments from Lemma \ref{lem:elecvelocityconvergeCoerc} we know that with the assumption of convergence \ref{item:item2mainth} we have $||\mathscr{U}^0\varrho-(-\Im(\varphi^\varepsilon\overline{\pi^\varepsilon}))||_{L^1}=O(\varepsilon)$ and $||\mathscr{E}-\mathscr{E}^\varepsilon||_{L^2}=O(\varepsilon)$ so that 
     \begin{align*}
        \int_{\mathbb{R}^3}{\mathscr{E}\nabla\varphi dx}+\int_{\mathbb{R}^3}{\mathscr{U}^0\rho\varphi dx}=0
    \end{align*}
    must hold for all $\varphi\in C_c^\infty(\mathbb{R}^3)$. Then, if the data are regular enough (as in Definition \ref{defi:wellprepREM}), we deduce that we have strong solutions to the constraints. The same result holds for the second equation.\\
\end{proof}
In this second Proposition, we show that if a solution to REM is close to a solution to mKGM (in the sense of the modulated energy) on a time interval, then the vector field $\textbf{U}$ must be normalized to $-1$. Thus, the normalization of $\textbf{U}$ in \eqref{eq:wellprepconstr1} is, a fortiori, mandatory if the interval contains the time $t=0$. 
\begin{propal}
\label{propal:proxiimplycontsraint2Init}
    Let $(\Phi^\varepsilon,\textbf{A}^\varepsilon)_{0<\varepsilon<1}$ be solutions to \eqref{eq:KGMKGM} given by Proposition \ref{propal:GWPKGM}. Let $(\textbf{U},F,\rho)$ be a solution to \eqref{eq:REMeq} without the normalization condition but with the regularity of Proposition \ref{propal:LWPREM} on $[0,T]$ and such that $\forall t\in[0,T]$
    \begin{align*}      &H^\varepsilon(\Phi^\varepsilon,\textbf{A}^\varepsilon,\textbf{U},F,\rho)(t)=O(\varepsilon^2), &\rho^\varepsilon(t)>0 \text{  a.e.},
    \end{align*}
    then 
    \begin{equation}
        \textbf{U}^\alpha\textbf{U}_\alpha=-1.
    \end{equation}
\end{propal}
\begin{proof}
First, we have 
\begin{align*}  \rho^\varepsilon\textbf{U}^\alpha\textbf{U}_\alpha+\rho^\varepsilon&=\frac{(\textbf{J}^\varepsilon)^\alpha\textbf{J}^\varepsilon_\alpha}{\rho^\varepsilon}+(\textbf{U}^\alpha\sqrt{\rho^\varepsilon}-\frac{(\textbf{J}^\varepsilon)^\alpha}{\sqrt{\rho^\varepsilon}})(\textbf{U}_\alpha\sqrt{\rho^\varepsilon}+\frac{\textbf{J}^\varepsilon_\alpha}{\sqrt{\rho^\varepsilon}})+\rho^\varepsilon\\
    &=\varepsilon^2\sqrt{\rho^\varepsilon}\Box\sqrt{\rho^\varepsilon}+(\textbf{U}^\alpha\sqrt{\rho^\varepsilon}-\frac{(\textbf{J}^\varepsilon)^\alpha}{\sqrt{\rho^\varepsilon}})(\textbf{U}_\alpha\sqrt{\rho^\varepsilon}+\frac{\textbf{J}^\varepsilon_\alpha}{\sqrt{\rho^\varepsilon}}),
\end{align*}
using equation \eqref{eq:splitJJ1KGM}. Then, for every $t_1,t_2\in[0,T]$ and every $\varphi\in C^\infty_c([t_1,t_2]\times\mathbb{R}^{3})$ we have
 \begin{align*}
        \int_{[t_1,t_2]\times\mathbb{R}^{3}}{\varepsilon^2\sqrt{\rho^\varepsilon}\Box\sqrt{\rho^\varepsilon}\varphi dx}&=-\int_{[t_1,t_2]\times\mathbb{R}^{3}}{\varepsilon^2\boldsymbol{\nabla}_\alpha\sqrt{\rho^\varepsilon}\boldsymbol{\nabla}^\alpha\sqrt{\rho^\varepsilon}\varphi+ \varepsilon^2\sqrt{\rho^\varepsilon}\boldsymbol{\nabla}^\alpha\sqrt{\rho^\varepsilon}\boldsymbol{\nabla}_\alpha\varphi dx}\\
        &+\int_{\mathbb{R}^{3}}{\varepsilon^2(\sqrt{\rho^\varepsilon}\partial_t\sqrt{\rho^\varepsilon}\varphi)(t_1)dx}-\int_{\mathbb{R}^{3}}{\varepsilon^2(\sqrt{\rho^\varepsilon}\partial_t\sqrt{\rho^\varepsilon}\varphi)(t_2)dx}\\
        &\lesssim\sup_{[0,T]}||\varepsilon\boldsymbol{\nabla}\sqrt{\rho^\varepsilon}||^2_{L^2}+\sup_{[0,T]}||\varepsilon\boldsymbol{\nabla}\sqrt{\rho^\varepsilon}||_{L^2}||\varepsilon\sqrt{\rho^\varepsilon}||_{L^2}\\
        &\lesssim \sup_{[0,T]}H^\varepsilon+\sup_{[0,T]}(H^\varepsilon)^{1/2}\varepsilon C_0\\
        &=O(\varepsilon^2)
\end{align*}
and 
\begin{align*}
        \int_{[t_1,t_2]\times\mathbb{R}^{3}}{(\textbf{U}^\alpha\sqrt{\rho^\varepsilon}-\frac{(\textbf{J}^\varepsilon)^\alpha}{\sqrt{\rho^\varepsilon}})(\textbf{U}_\alpha\sqrt{\rho^\varepsilon}+\frac{\textbf{J}^\varepsilon_\alpha}{\sqrt{\rho^\varepsilon}})\varphi dx}&\lesssim ||\frac{\textbf{U}_\alpha\rho^\varepsilon-\textbf{J}^\varepsilon_\alpha}{\sqrt{\rho}^\varepsilon}||_{L^2}(||\frac{\textbf{J}^\varepsilon_\alpha}{\sqrt{\rho^\varepsilon}}||+||\textbf{U}_\alpha\sqrt{\rho^\varepsilon}||_{L^2})\\
        &\lesssim (H^\varepsilon)^{1/2}(||\frac{|\boldsymbol{D}^\varepsilon\Phi^\varepsilon||\Phi^\varepsilon|}{|\Phi^\varepsilon|}||_{L^2}+||\textbf{U}_\alpha||_{L^\infty}||\sqrt{\rho^\varepsilon}||_{L^2})\\
        &\lesssim (H^\varepsilon)^{1/2}C_0\\
        &=O(\varepsilon),
    \end{align*}
    where we use the Definition \ref{defi:momentumKGM} for $\textbf{J}^\varepsilon$, the Lemma \ref{lem:elecvelocityconvergeCoerc} and the uniform bound on the energy for mKGM from Proposition \ref{propal:GWPKGM}. The latter controls uniformly the mass $||\sqrt{\rho^\varepsilon}||_{L^2}$ and the kinetic energy $||\boldsymbol{D}^\varepsilon\Phi^\varepsilon||_{L^2}$. \\\
    We deduce that for every $\varphi\in C^\infty_c(\mathbb{R}^{3+1})$ we have 
    \begin{align*}
        \int_{\mathbb{R}^{3+1}}{(\rho^\varepsilon\textbf{U}^\alpha\textbf{U}_\alpha+\rho^\varepsilon)\varphi dx}=0
    \end{align*}
    and so that 
    \begin{align*}
        \textbf{U}^\alpha\textbf{U}_\alpha=-1
    \end{align*}
    by continuity of $\textbf{U}$ and because we assumed that $\rho^\varepsilon>0$ a.e..
    \end{proof}
The final Proposition works the other way around. It states that if we assume the well preparedness \ref{defi:wellprepREM} and \ref{defi:wellprepKGMsc} and if we have the assumption of convergence on the free part of the initial data, then we have the full assumption of convergence \ref{item:item2mainth} of Theorem \ref{unTheorem:TH1mainth}. More precisely, if initially the Maxwell constraints \eqref{eq:wellprepconstr4} are satisfied and if we have the smallness of the kinetic part (see Remark \ref{rem:kineticelectrodefmod}) of the modulated energy, the convergence of the density and the convergence of the curl part\footnote{See next Definition \ref{defi:Helmoltz}.} of the electric field (the part that is not constrained by the equation \eqref{eq:wellprepconstr4}) then we automatically get the convergence of the electric field. Before that, we recall the definition of Helmholtz decomposition.
\begin{defi}
\label{defi:Helmoltz}
    Let $X$ be a regular vector field from $\mathbb{R}^{3+1}$ to $\mathbb{R}^3$, then 
    \begin{equation}
        X=X^{df}+X^{cf}
    \end{equation}
    where
    \begin{align*}
        &X^{df}=\mathscr{P}(X)=\mathcal{F}^{-1}(\hat{X}-\frac{\hat{X}\cdot\xi}{|\xi|^2}\xi) &X^{cf}=(1-\mathscr{P})(X)=\mathcal{F}^{-1}(\frac{\hat{X}\cdot\xi}{|\xi|^2}\xi)
    \end{align*}
    are respectively the curl free and the divergence free parts of $X$.
\end{defi}
\begin{propal}
\label{propal:constraintimplyproxiInit}
Under the well-preparedness assumption \ref{item:item1mainth}, if we assume that\footnote{The kinetic part of the modulated energy is defined in Remark \ref{rem:kineticelectrodefmod}.}
    \begin{equation}
    \label{eq:assuminitialdatapropal}
        K^\varepsilon_0(0)+||\sqrt{\varrho^\varepsilon}-\sqrt{\varrho}||^2_{L^2}=O(\varepsilon^2)
    \end{equation}
    for $\varrho^\varepsilon=|\varphi^\varepsilon|^2$ and that\footnote{This assumption corresponds to the convergence of the free part of the electric field. Indeed, the constraint equation only controls the curl free part.}
    \begin{equation}
         ||(\mathscr{E}^\varepsilon)^{df}-\mathscr{E}^{df}||^2_{L^2}=O(\varepsilon^2)
    \end{equation}
    then we have the convergence of the electric field
    \begin{equation}
        ||\mathscr{E}^\varepsilon-\mathscr{E}||^2_{L^2}=O(\varepsilon^2).
    \end{equation}
\end{propal} 
\begin{proof}
Firstly, with argument from Lemma \ref{lem:elecvelocityconvergeCoerc} we know that 
\begin{align*}
    K^\varepsilon_0(0)=O(\varepsilon^2)
\end{align*}
implies
\begin{align}
\label{eq:impliesKineticInit}
    ||\frac{(-\Im(\varphi^\varepsilon\overline{\pi^\varepsilon}))-\mathscr{U}^0\varrho^\varepsilon}{\sqrt{\varrho^\varepsilon}}||_{L^2}+   ||\varepsilon\nabla\sqrt{\varrho^\varepsilon|}||_{L^2}=O(\varepsilon),
\end{align}
for $-\Im(\varphi^\varepsilon\overline{\pi^\varepsilon})=(\textbf{J}^\varepsilon)^0|_{t=0}$, and so, with the uniform bound assumption of the (weighted) energy \eqref{eq:GWPweighted}, we get
 \begin{align}
 \label{eq:uniformH1initial}
  ||\sqrt{\varrho^\varepsilon}||_{H^{1}}=O(1).
\end{align}
       Then, from the well preparedness \ref{defi:wellprepREM} we have the constraint equation \eqref{eq:wellprepconstr3} and so 
       \begin{equation}
       \nabla\cdot(\mathscr{E}^{cf})=\nabla\cdot \mathscr{E}=-\mathscr{U}^0\varrho.
       \end{equation}
       On the other hand we know that there exists a function $\zeta$ on $\mathbb{R}^3$ such that
       \begin{equation}
           \mathscr{E}^{cf}=\nabla\zeta.
       \end{equation}
      By doing the same operation on $E^\varepsilon$ we get
           \begin{equation}
       \Delta (\zeta^\varepsilon-\zeta)=-\Im(\varphi^\varepsilon\overline{\pi^\varepsilon})-\mathscr{U}^0\varrho.
       \end{equation}
       Then, by direct calculations, we find that
       \begin{align*}
           ||\nabla(\zeta^\varepsilon-\zeta)||^2_{L^2}&=-\int_{\mathbb{R}^3}{(-\Im(\varphi^\varepsilon\overline{\pi^\varepsilon})-\mathscr{U}^0\varrho)(\zeta^\varepsilon-\zeta) dx}\\
           &\leq||\zeta^\varepsilon-\zeta||_{\dot{H}^1}||-\Im(\varphi^\varepsilon\overline{\pi^\varepsilon})-\mathscr{U}^0\varrho||_{\dot{H}^{-1}},
       \end{align*}
       so that 
             \begin{align*}
           ||\nabla(\zeta^\varepsilon-\zeta)||_{L^2}\leq||-\Im(\varphi^\varepsilon\overline{\pi^\varepsilon})-\mathscr{U}^0\varrho||_{\dot{H}^{-1}}.
       \end{align*}
       To control the right hand side, we separate it into low and high frequencies. We have        
        \begin{align*}
           ||-\Im(\varphi^\varepsilon\overline{\pi^\varepsilon})-\mathscr{U}^0\varrho||^2_{\dot{H}^{-1}}&=\int_{\mathscr{B}(0,1)}{\frac{1}{|\xi|^2}[\mathcal{F}(-\Im(\varphi^\varepsilon\overline{\pi^\varepsilon}))-\mathscr{U}^0\varrho)]^2d\xi}+\int_{\mathbb{R}^3\backslash\mathscr{B}(0,1)}{\frac{1}{|\xi|^2}[\mathcal{F}(-\Im(\varphi^\varepsilon\overline{\pi^\varepsilon}))-\mathscr{U}^0\varrho)]^2d\xi}\\
           &\lesssim \int_{\mathscr{B}(0,1)}{\frac{1}{|\xi|^2}d\xi}||\mathcal{F}(-\Im(\varphi^\varepsilon\overline{\pi^\varepsilon}))-\mathscr{U}^0\varrho)||^2_{L^\infty}+\int_{\mathbb{R}^3\backslash\mathscr{B}(0,1)}{\frac{1}{(1+|\xi|^2)}[\mathcal{F}(-\Im(\varphi^\varepsilon\overline{\pi^\varepsilon}))-\mathscr{U}^0\varrho)]^2d\xi}\\
           &\lesssim ||-\Im(\varphi^\varepsilon\overline{\pi^\varepsilon})-\mathscr{U}^0\varrho||^2_{L^1}+||-\Im(\varphi^\varepsilon\overline{\pi^\varepsilon})-\mathscr{U}^0\varrho||^2_{H^{-1}}.\\
        \end{align*}
        We show in Lemma \ref{lem:densityconvergestrongCoerc} how to get $||-\Im(\varphi^\varepsilon\overline{\pi^\varepsilon})-\mathscr{U}^0\varrho||^2_{L^1}=O(\varepsilon^2)$ with the assumption on the modulated energy and the density of the present Proposition \ref{propal:constraintimplyproxiInit}. We do the same type of operations on the other term to get 
        \begin{align*}
            ||-\Im(\varphi^\varepsilon\overline{\pi^\varepsilon})-\mathscr{U}^0\varrho||_{H^{-1}}&\leq ||\frac{-\Im(\varphi^\varepsilon\overline{\pi^\varepsilon})-\mathscr{U}^0\varrho^\varepsilon}{\sqrt{\varrho^\varepsilon}}\sqrt{\varrho^\varepsilon}||_{H^{-1}}+||\mathscr{U}^0(\varrho^\varepsilon-\varrho)||_{H^{-1}}\\  
            &\leq||\frac{-\Im(\varphi^\varepsilon\overline{\pi^\varepsilon})-\mathscr{U}^0\varrho^\varepsilon}{\sqrt{\varrho^\varepsilon}}||_{L^2}||\sqrt{\varrho^\varepsilon}||_{H^{1/2}}+||\sqrt{\varrho^\varepsilon}-\sqrt{\varrho}||_{L^2}||\mathscr{U}^0(\sqrt{\varrho^\varepsilon}+\sqrt{\varrho})||_{H^{1/2}}\\
             &\leq O(\varepsilon)(||\sqrt{\varrho^\varepsilon}||_{H^{1}})^{1/2}(||\sqrt{\varrho^\varepsilon}||_{L^2})^{1/2}+O(\varepsilon)||\mathscr{U}^0||_{H^{1}}||(\sqrt{\varrho^\varepsilon}+\sqrt{\varrho})||_{H^{1}}\\
            &\leq O(\varepsilon),
        \end{align*}
       where we use the Sobolev product inequality of \cite{zbMATH01584967} and the uniform bound on the $H^1$ norm of $\sqrt{\varrho^\varepsilon}$ and $\sqrt{\varrho}$ given by \eqref{eq:uniformH1initial} and \eqref{eq:wellprepineq}. Finally, gathering everything, we get 
         \begin{align*}
          ||\mathscr{E}^\varepsilon-\mathscr{E}||_{L^2}&\leq  ||(\mathscr{E}^\varepsilon)^{df}-\mathscr{E}^{df}||_{L^2}+||(\mathscr{E}^\varepsilon)^{cf}-\mathscr{E}^{cf}||_{L^2}\\
          &\leq||(\mathscr{E}^\varepsilon)^{df}-\mathscr{E}^{df}||_{L^2}+||\nabla(\zeta^\varepsilon-\zeta)||_{L^2}\\
          &=O(\varepsilon),
       \end{align*}
       which is the desired convergence.
\end{proof}

\section{Proof of Theorem \ref{unTheorem:TH1mainth}}
\label{section:PROOF}
This Section is dedicated to the proof of Theorem \ref{unTheorem:TH1mainth}. 
\subsection{Existence and regularity of the solutions}
\label{subsection:Existreg}
 The existence of a family of global solutions $(\Phi^\varepsilon,\textbf{A}^\varepsilon)_{0<\varepsilon<1}$ to mKGM is given by Proposition \ref{propal:GWPKGM} under the well-preparedness assumption \ref{defi:wellprepKGMsc}. The existence of a solution $(\textbf{U},F,\rho)$ to REM is given by Proposition \ref{propal:LWPREM} under the well-preparedness assumption \ref{defi:wellprepREM}. Moreover, these Propositions give us $\forall t\in\mathbb{R}\;\forall 0<\varepsilon<1\;(\Phi^\varepsilon,\textbf{A}^\varepsilon)(t)\in H^5\times H^5$ and the following bounds 
\begin{align}
\label{eq:weightedenergyproof}
  &\forall t\in[0,T]\;\int_{\mathbb{R}^3}{(1+|x|^2)^{\kappa}\left(\frac{|\textbf{D}^\varepsilon\Phi^\varepsilon|^2}{2}+\frac{|\Phi^\varepsilon|^2}{2}+\frac{|E^\varepsilon|^2+|B^\varepsilon|^2}{2}\right)dx}\leq C_0,\\
  \label{eq:backgroundineqyproof}
      &\sum_{k=0}^4||F||_{C^{4-k}([0,T],H^k)}+\sum_{k=0}^4||\textbf{U}||_{C^{4-k}([0,T],H^k)}+\sum_{k=0}^3||\rho||_{C^{3-k}([0,T],H^k)}\leq C_0,
\end{align}
for $C_0>0$ and $\kappa>0$. Thus, we have the point 1 of Theorem \ref{unTheorem:TH1mainth}. In particular, from \eqref{eq:weightedenergyproof}, we have 
 \begin{align}
 \label{eq:energyandweightproof}
      &\sup_{t\in[0,T]}\mathcal{E}^\varepsilon_{mKGM}[\Phi^\varepsilon,\textbf{A}^\varepsilon](t)+ \sup_{t\in[0,T]}||\sqrt{\rho^\varepsilon}(t)||_{L^2_\kappa}\leq C_0.
 \end{align}
\subsection{Proof of coercivity property}
\label{subsection:Coerciviy}
In this Section, we consider that the existence and regularity properties of $(\Phi^\varepsilon,\textbf{A}^\varepsilon)_{0<\varepsilon<1}$ and $(\textbf{U},F,\rho)$ are those of Section \eqref{subsection:Existreg} under the well-preparedness assumption \ref{item:item1mainth} of Theorem \ref{unTheorem:TH1mainth}. The following Proposition is the main result of this Section. It states that if the modulated energy\footnote{We mean by that the equivalence class defined in Section \ref{section:defmod}.} $H^\varepsilon$ is in $O(\varepsilon^2)$ on some time interval then we have the desired convergence, we have the weak coercitivity property of Remark \ref{rem:modpoint1Ideaweak}. In fact, to recover the strong statement of the coercivity property we need to use fully the assumption of convergence \ref{item:item2mainth} of Theorem \ref{unTheorem:TH1mainth} on the initial data with the convergence of the density. 
\begin{propal}
\label{propal:mainpropalCoerciviy}
    Under the well-preparedness assumption \ref{item:item1mainth}, if we have 
    \begin{equation}
\sup_{t\in[0,T]}H^\varepsilon(t)=O(\varepsilon^2),
    \end{equation}
    then there exists $\sqrt{\rho'}\in C^{0}([0,T],L^2)$ such that 
    \begin{equation}
        \begin{split}
            & \lim_{\varepsilon\to0}||\textbf{J}^\varepsilon-\textbf{U}\rho'||_{L^\infty([0,T],L^{1})}+||F^\varepsilon-F||_{L^\infty([0,T],L^2)}=0,\\
     &\lim_{\varepsilon\to0}||\rho^\varepsilon-\rho'||_{L^\infty([0,T],L^1)}+||\sqrt{\rho^\varepsilon}-\sqrt{\rho'}||_{L^\infty([0,T],L^2)}=0.
        \end{split}
       \end{equation}
       We recover the weak coercivity property of Remark \ref{rem:modpoint1Ideaweak}.
        Moreover, if the initial data satisfy\footnote{This assumption is a part of the assumption of convergence \ref{item:item2mainth} of Theorem \ref{unTheorem:TH1mainth}}
        \begin{equation}
          \lim_{\varepsilon\to0}||\sqrt{\varrho^\varepsilon}-\sqrt{\varrho}||_{L^2}=0,
       \end{equation}
        for $\varrho^\varepsilon=|\varphi^\varepsilon|^2$, then $\rho'=\rho$ and so
            \begin{equation}
        \begin{split}
            & \lim_{\varepsilon\to0}||\textbf{J}^\varepsilon-\textbf{U}\rho||_{L^\infty([0,T],L^{1})}+||F^\varepsilon-F||_{L^\infty([0,T],L^2)}=0,\\
     &\lim_{\varepsilon\to0}||\rho^\varepsilon-\rho||_{L^\infty([0,T],L^1)}+||\sqrt{\rho^\varepsilon}-\sqrt{\rho}||_{L^\infty([0,T],L^2)}=0.
        \end{split}
       \end{equation}
    The modulated energy $H^\varepsilon$ (and so the representative $H^\varepsilon_0$) satisfies the strong statement of the coercivity property of Lemma \ref{lem:modpoint1Idea}. 
\end{propal}
To demonstrate the Proposition, we break the proof into different Propositions and Lemmas. 
We start by defining useful quantities and their respective properties. We consider the quantity $h^\varepsilon_{\alpha\beta}$ from Proposition \ref{propal:tensordiffdefmod}.
\begin{propal}
\label{propal:detailh00Coerc}
We have 
\begin{equation}
\label{eq:h00simpleCyproof}
    h^\varepsilon_{00}=\frac{|\boldsymbol{\xi}|^2}{2}+\frac{|\mathcal{E}|^2+|\mathcal{B}|^2}{2},
\end{equation}
and 
\begin{equation}
\label{eq:hgoodshapeCoer}
h^\varepsilon_{00}=\varepsilon^2\frac{|\boldsymbol{\nabla}\sqrt{\rho^\varepsilon}|^2}{2}+\frac{|\textbf{J}^\varepsilon-\rho^\varepsilon \textbf{U}|^2}{2\rho^\varepsilon}+\frac{|\mathcal{E}|^2+|\mathcal{B}|^2}{2}.
\end{equation}    
\end{propal}
\begin{proof}
We have \eqref{eq:h00simpleCyproof} by direct calculation. We give a proof of the identity \eqref{eq:hgoodshapeCoer} for the sake of completeness, it resembles the proof found in \cite{salvi2024semiclassicallimitkleingordonequation}. We have
\begin{align*}
    \frac{|(\boldsymbol{D}^\varepsilon-i\textbf{U})\Phi^\varepsilon|^2}{2}&=\frac{|\boldsymbol{D}^\varepsilon\Phi|^2}{2}+\frac{|\textbf{U}|^2|\Phi^\varepsilon|^2}{2}-\textbf{J}^\varepsilon\cdot\textbf{U}=\frac{\textbf{J}^\varepsilon\cdot\textbf{J}^\varepsilon}{2\rho^\varepsilon}+\frac{\varepsilon^2\boldsymbol{\nabla}\sqrt{\rho^\varepsilon}\cdot\boldsymbol{\nabla}\sqrt{\rho^\varepsilon}}{2}+\frac{|\textbf{U}|^2\rho^\varepsilon}{2}-\textbf{J}^\varepsilon\cdot\textbf{U}\\    &=\varepsilon^2\frac{|\boldsymbol{\nabla}\sqrt{\rho^\varepsilon}|^2}{2}+\frac{|\textbf{J}^\varepsilon-\rho^\varepsilon \textbf{U}|^2}{2\rho^\varepsilon},
\end{align*}
where we use the equation \eqref{eq:splitJJ2KGM}.
\end{proof}
\begin{rem}
\label{rem:recallcanonicalCoerc}
    We recall that we define the modulated energy $H^\varepsilon$ in Section \ref{section:defmod} as the canonical representative of the equivalence class of 
    \begin{align*}
H^\varepsilon_0(\Phi^\varepsilon,\partial_t\Phi^\varepsilon,\textbf{A}^\varepsilon,\textbf{U},\rho,\textbf{A})=\int_{\mathbb{R}^3}{h^\varepsilon_{00}dx}.
    \end{align*}
\end{rem}
Firstly, we can state that $H^\varepsilon$ controls the convergence of the electromagnetic field and the convergence of some "velocity quantity", the momentum $\textbf{J}^\varepsilon$ divided by $\sqrt{\rho^\varepsilon}$.
\begin{lem}
\label{lem:elecvelocityconvergeCoerc}
    Under the assumptions of Proposition \ref{propal:mainpropalCoerciviy} we have
    \begin{equation}
       \label{eq:HcontrolCoerc}
       ||(\frac{\textbf{J}^\varepsilon-\rho^\varepsilon \textbf{U}}{\sqrt{\rho^\varepsilon}})(t)||^2_{L^{2}}+||(F^\varepsilon-F)(t)||^2_{L^2}\leq C_0H^\varepsilon(t),
    \end{equation}
    and so 
    \begin{equation}
        \begin{split}
            & \lim_{\varepsilon\to0}||\frac{\textbf{J}^\varepsilon-\rho^\varepsilon \textbf{U}}{\sqrt{\rho^\varepsilon}}||_{L^\infty([0,T],L^{2})}+||F^\varepsilon-F||_{L^\infty([0,T],L^2)}=0.\\
        \end{split}
    \end{equation}
\end{lem}
\begin{proof}
    We get the result with Proposition \ref{propal:detailh00Coerc} and Definition \ref{defi:canonicalequivdefmod}. In particular, all the component of the tensor $F^\varepsilon-F$ are controlled by $\mathcal{E}$ and $\mathcal{B}$ defined in  \ref{defi:xiXidefmod}.
\end{proof}
Now, we state that the quantity $\sqrt{\rho^\varepsilon}$ must converge strongly in $L^2$.
\begin{lem}
\label{lem:densityconvergeCoerc}
    Under the assumptions of Proposition \ref{propal:mainpropalCoerciviy} there exists $\sqrt{\rho'}\in C^{0}([0,T],L^2)$ such that 
    \begin{equation}
        \begin{split}
            & \lim_{\varepsilon\to0}(\sup_{t\in[0,T]}||\rho^\varepsilon-\rho'||_{L^{1}}+\sup_{t\in[0,T]}||\sqrt{\rho^\varepsilon}-\sqrt{\rho'}||_{L^2})=0,\\
        \end{split}
    \end{equation}
\end{lem}
\begin{proof}
    From the uniform bound \eqref{eq:energyandweightproof} we get the control of  
    \begin{equation}
        \sup_{t\in[0,T]}||\sqrt{\rho^\varepsilon}(t)||_{L^2}+||\sqrt{\rho^\varepsilon}(t)||_{L^2_\kappa}\leq C_0,
    \end{equation}
     for $\kappa>0$. This implies that for any $R>0$
    \begin{equation}
    \label{eq:decayequation}
         \sup_{t\in[0,T]}||\sqrt{\rho^\varepsilon}(t)||_{L^2(\mathbb{R}^3\backslash\mathscr{B}(0,R))}\lesssim \sup_{t\in[0,T]}\frac{1}{R^\kappa}||\sqrt{\rho^\varepsilon}(t)||_{L^2_\delta}\leq \frac{1}{R^\kappa}C_0.
    \end{equation}
    Then, from the assumption on $H^\varepsilon$ and the Proposition \ref{propal:detailh00Coerc}, we also have 
    \begin{equation}     \sup_{t\in[0,T]}||\nabla\sqrt{\rho^\varepsilon}(t)||_{L^2}\leq \sup_{t\in[0,T]} (\varepsilon^{-2}H^\varepsilon_0)^{1/2}\lesssim\sup_{t\in[0,T]} (\varepsilon^{-2}H^\varepsilon)^{1/2}\leq C_0
    \end{equation}
    and 
    \begin{equation}     \sup_{t\in[0,T]}||\partial_t\sqrt{\rho^\varepsilon}(t)||_{L^2}\leq \sup_{t\in[0,T]} (\varepsilon^{-2}H^\varepsilon_0)^{1/2}\lesssim\sup_{t\in[0,T]} (\varepsilon^{-2}H^\varepsilon)^{1/2}\leq C_0.
    \end{equation}
    This implies that 
    \begin{equation}
    \label{eq:regularequation}
         \sup_{t\in[0,T]}(||\sqrt{\rho^\varepsilon}(t)||_{H^1}+||\partial_t\sqrt{\rho^\varepsilon}(t)||_{L^2})\leq C_0.
    \end{equation}
    With \eqref{eq:decayequation} and \eqref{eq:regularequation}, we can use the Fréchet-Kolmogorov Theorem\footnote{The sequence lives in a closed set of $L^2$ whose elements are uniformly decaying with a uniform $H^1$ regularity.} to show that for all $t\in[0,T]$ the sequence $(\sqrt{\rho^\varepsilon}(t))_{\varepsilon\in(0,1)}$ is in a compact set of $L^2$. Then, with the Ascoli Theorem\footnote{The sequence lives in a closed and equicontinuous set of $C^0([0,T],X)$ (for $X$ a compact set of $L^2$).}, we get the existence\footnote{We abuse the square root notation as we only know for now that $\sqrt{\rho'}\geq0$ a.e. by the fact that $\sqrt{\rho^\varepsilon}\geq0$.} of a sub sequence converging to some $\sqrt{\rho'}\in C^0([0,T],L^2)$, such that 
    \begin{align*}        \lim_{\varepsilon\rightarrow0}\sup_{t\in[0,T]}||\sqrt{\rho^\varepsilon}(t)-\sqrt{\rho'}(t)||_{L^2}=0.
    \end{align*}
    Moreover, we have
    \begin{align*}        
    ||\rho^\varepsilon(t)-\rho'(t)||_{L^1}\leq||\sqrt{\rho^\varepsilon}(t)-\sqrt{\rho'}(t)||_{L^2}||\sqrt{\rho^\varepsilon}(t)+\sqrt{\rho'}(t)||_{L^2}\leq ||\sqrt{\rho^\varepsilon}(t)-\sqrt{\rho'}(t)||_{L^2}C_0
    \end{align*}
    and so 
    \begin{equation}
        \lim_{\varepsilon\to0}\sup_{t\in[0,T]}||\rho^\varepsilon-\rho'||_{L^{1}}=0,
    \end{equation}
    which ends the proof of the Lemma.
\end{proof}
The previous Lemma gives the existence of a certain $\rho'$ to which $\rho^\varepsilon$ converges but there is no argument yet to show that $\rho'$ is equal to $\rho$, we do not have directly the strong coercivity property. To recover this statement, we need to ensure that the initial data coincide for $\rho'$ and $\rho$. Then, we show that both are transported along the flow of $\textbf{U}$ and so remain equal. 
\begin{lem}
\label{lem:densityconvergestrongCoerc}
       Under the assumptions of Proposition \ref{propal:mainpropalCoerciviy} the limit $\rho'$ defined in Lemma \ref{lem:densityconvergeCoerc} is equal to $\rho$ the solution to \eqref{eq:REMeq} if initially we have
    \begin{equation}
            \lim_{\varepsilon\to0}|||\varphi^\varepsilon|-\sqrt{\varrho}||_{L^2}=0.\\
    \end{equation}
\end{lem}
\begin{proof}
    First, we show that $\rho'$ is a weak solution to the transport equation 
    \begin{equation}   
    \label{eq:transportaloneCoerc}
    \boldsymbol{\nabla}_\alpha(\textbf{U}^\alpha\rho')=0.
    \end{equation}
    Indeed, we have $\forall\varphi\in C^\infty_c([0,T]\times\mathbb{R}^3)$
    \begin{align*}
&|\int^T_0{\int{\textbf{U}^\alpha\rho'\boldsymbol{\nabla}_\alpha\varphi dx}dt}+\int{\textbf{U}^0(0)\rho'(0)\varphi (0)-\textbf{U}^0(T)\rho'(T)\varphi (T)dx}|\\
&=|\int^T_0{\int{(\textbf{U}^\alpha(\rho'-\rho^\varepsilon)
+(\textbf{U}^\alpha\rho^\varepsilon-\textbf{J}^\alpha))\boldsymbol{\nabla}_\alpha\varphi dx}dt}\\
&+\int{(\textbf{U}^0(\rho'-\rho^\varepsilon)(0)
+(\textbf{U}^0\rho^\varepsilon-\textbf{J}^\alpha)(0))\varphi (0)-(\textbf{U}^0(\rho'-\rho^\varepsilon)(T)
+(\textbf{U}^0\rho^\varepsilon-\textbf{J}^\alpha)(T))\varphi (T)dx}|\\
&\leq C_0( \sup_{t\in[0,T]}||\textbf{U}(t)||_{L^\infty}\sup_{t\in[0,T]}||\rho'-\rho^\varepsilon||_{L^1}+\sup_{t\in[0,T]}||\textbf{U}^\alpha\rho^\varepsilon-\textbf{J}^\alpha||_{L^1})\\
&\leq C_0\sup_{t\in[0,T]}||\rho'-\rho^\varepsilon||_{L^1}+\sup_{t\in[0,T]}||\frac{\textbf{U}^\alpha\rho^\varepsilon-\textbf{J}^\alpha}{\sqrt{\rho^\varepsilon}}||_{L^2}\sup_{t\in[0,T]}||\sqrt{\rho^\varepsilon}||_{L^2},
    \end{align*}
    where we use the fact that 
    \begin{equation}     \boldsymbol{\nabla}_\alpha(\textbf{J}^\varepsilon)^\alpha=0.
    \end{equation}
Then, we use Propositions \ref{lem:densityconvergeCoerc} and \ref{lem:elecvelocityconvergeCoerc} to deduce that the right-hand side of the inequality is as small as we want. This implies that
$\forall\varphi\in C^\infty_c([0,T]\times\mathbb{R}^3)$
    \begin{equation}
\int^T_0{\int{\textbf{U}^\alpha\rho'\boldsymbol{\nabla}_\alpha\varphi dx}dt}+\int{\textbf{U}^0(0)\rho'(0)\varphi (0)-\textbf{U}^0(T)\rho'(T)\varphi (T)dx}=0,
\end{equation}
we have a weak solution to \eqref{eq:transportaloneCoerc}.\\
 We know that we have unicity of the solution for such a transport equation because the vector field $\textbf{U}$ is regular, see Proposition \ref{propal:LWPREM} for the regularity. Then, because we assume that the initial data coincide, we must have $\rho'=\rho$ as the only solution to \eqref{eq:transportaloneCoerc}. 
\end{proof}
Now, we have enough material to demonstrate Proposition \ref{propal:mainpropalCoerciviy} fully.
\begin{proof}[Proof of Proposition \ref{propal:mainpropalCoerciviy}]
With Lemmas \ref{lem:elecvelocityconvergeCoerc}, \ref{lem:densityconvergeCoerc}, and \ref{lem:densityconvergestrongCoerc} most of the proof is already done, we have the convergence of the electromagnetic field and the density. It remains to show the convergence of the momentum.
We do it only for the strong statement of the coercivity property of Lemma \ref{lem:modpoint1Idea} because the weak statement uses the same argument. We have 
\begin{align*}
    \lim_{\varepsilon\to0}||\textbf{J}^\varepsilon-\textbf{U}\rho||_{L^\infty([0,T],L^{1})}&=\lim_{\varepsilon\to0}||\textbf{J}^\varepsilon-\textbf{U}\rho^\varepsilon||_{L^\infty([0,T],L^{1})}+||\textbf{U}(\rho^\varepsilon-\rho)||_{L^\infty([0,T],L^{1})}\\
    &\leq\lim_{\varepsilon\to0}||\frac{\textbf{J}^\varepsilon-\textbf{U}\rho^\varepsilon}{\sqrt{\rho^\varepsilon}}||_{L^\infty([0,T],L^{2})}||\sqrt{\rho^\varepsilon}||_{L^\infty([0,T],L^{2})}+||\rho^\varepsilon-\rho||_{L^\infty([0,T],L^{1})}||\textbf{U}||_{L^\infty([0,T],L^\infty)}
\end{align*}
and so, using the inequalities \eqref{eq:backgroundineqyproof}, \eqref{eq:energyandweightproof} and the result of Lemmas \ref{lem:elecvelocityconvergeCoerc} and \ref{lem:densityconvergestrongCoerc}, we get that 
\begin{align*}
    \lim_{\varepsilon\to0}||\textbf{J}^\varepsilon-\textbf{U}\rho||_{L^\infty([0,T],L^{1})}=0.
\end{align*}
We clearly see that the same argument holds for $\rho'$ if we do not have the information of Lemma \ref{lem:densityconvergestrongCoerc} but only \ref{lem:densityconvergeCoerc}.\\
This ends the proof of Proposition \ref{propal:mainpropalCoerciviy}, the modulated energy $H^\varepsilon$ (and so $H^\varepsilon_0$) satisfies the \textbf{coercivity property} of Lemma \ref{lem:modpoint1Idea}. 
\end{proof}
Before going to the next Section, we show that the modulated energy associated with the reference frame of $\textbf{U}$, $H^\varepsilon_{\textbf{U}}$, is in the equivalent class of $H^\varepsilon$. The former is useful because it allows us to show the propagation property of Lemma \ref{lem:modpoint2Idea} directly.
\begin{propal}
\label{propal:HUgoodfinalCoerc}
 The quantity $H^\varepsilon_\textbf{U}=\int_{\mathbb{R}^3}{\eta^\varepsilon(\textbf{U})dx}$ (defined in \ref{defi:etadefmod}) is in the equivalence class of $H^\varepsilon$, that is,
 \begin{align}
     \label{eq:equivHUH0Coerc}
     H^\varepsilon_\textbf{U}\sim H^\varepsilon.
 \end{align}
\end{propal}
\begin{proof}
    We want to show that the vector field $\textbf{U}$ is acceptable \ref{defi:acceptablefieldmod}, then we can use Proposition \ref{propal:acceptimplyequivdefmod}. We pick $\nu=\min(1,\frac{1}{C_0})$ so that we have  $|\textbf{U}|=||\textbf{U}||_{C^0([0,T]\times\mathbb{R}^3)}\leq \frac{1}{\nu}$ from \eqref{eq:backgroundineqyproof}, $\textbf{U}^0\geq\nu$ from \eqref{eq:LWPpropagconstr1} and $-\textbf{U}^\alpha\textbf{U}_\alpha\geq\nu$ from \eqref{eq:LWPpropagconstr2}. We can apply Proposition \ref{propal:acceptimplyequivdefmod} to have 
    \begin{equation}
		c_1(\nu)H^\varepsilon_{\textbf{U}}<H^\varepsilon<c_2(\nu)H^\varepsilon_{\textbf{U}},
\end{equation}
and end the proof. 
\end{proof}
\subsection{Proof of propagation property}
\label{subsection:Propagation}
In this Section, we consider that the existence and the regularity properties of $(\Phi^\varepsilon,\textbf{A}^\varepsilon)_{0<\varepsilon<1}$ and $(F,\textbf{U},\rho)$ are those of Section \ref{subsection:Existreg}. Moreover, we assume that the assumption of convergence \ref{item:item2mainth} of Theorem \ref{unTheorem:TH1mainth} is satisfied. The following Proposition is the main result of this Section. It states that the modulated energy, defined in Section \ref{section:defmod}, satisfies the propagation property of Lemma \ref{lem:modpoint2Idea}, it propagates its smallness. 
\begin{propal}
\label{propal:mainpropalPropagation}
 Under the assumption \ref{item:item1mainth} and \ref{item:item2mainth} of Theorem \ref{unTheorem:TH1mainth}, we have 
\begin{equation}
\sup_{t\in[0,T]}H^\varepsilon(t)=O(\varepsilon^2).
\end{equation}
\end{propal}
\begin{proof}
We demonstrate the Proposition for the representative $H^\varepsilon_{\textbf{U}}$ which is in the equivalence class of $H^\varepsilon$, see Definition \ref{defi:canonicalequivdefmod} and Proposition \ref{propal:HUgoodfinalCoerc}.
In the following calculations, the dependence in $t$ is implicit. First, we calculate that 
\begin{align}
\label{eq:etaUPgproof}
\eta^\varepsilon(\textbf{U})&=h^\varepsilon_{0\alpha}\textbf{U}^\alpha=((T_{mKGM}^\varepsilon)_{0\alpha}-(T_{REM})_{0\alpha})\textbf{U}^\alpha -I^\varepsilon_{0\alpha}\textbf{U}^\alpha \\
&=((T_{mKGM}^\varepsilon)_{0\alpha}-(T_{REM})_{0\alpha})\textbf{U}^\alpha+\textbf{U}_0 \textbf{U}_\alpha (\rho-\rho^\varepsilon)\textbf{U}^\alpha -\textbf{U}_0 (\textbf{J}^\varepsilon_\alpha-\textbf{U}_\alpha\rho^\varepsilon)\textbf{U}^\alpha-\textbf{U}_\alpha (\textbf{J}^\varepsilon_0-\textbf{U}_0\rho^\varepsilon)\textbf{U}^\alpha  \nonumber\\
&-F_{0\mu}((F^\varepsilon)_{\alpha}^{~\mu}-F_{\alpha}^{~\mu})\textbf{U}^\alpha-(F^\varepsilon_{0\mu}-F_{0\mu})F_{\alpha}^{~\mu}\textbf{U}^\alpha+\textbf{U}_0(\textbf{J}^\varepsilon_\gamma-\textbf{U}_\gamma \rho^\varepsilon)\textbf{U}^\gamma+\frac{1}{2}\textbf{U}_0F_{\mu\nu}((F^\varepsilon)^{\mu\nu}-F^{\mu\nu})\nonumber\\
&=((T_{mKGM}^\varepsilon)_{0\alpha}-(T_{REM})_{0\alpha})\textbf{U}^\alpha-(\textbf{U}_0\rho-\textbf{J}^\varepsilon_0)-F_{0\mu}((F^\varepsilon)_{\alpha}^{~\mu}-F_{\alpha}^{~\mu})\textbf{U}^\alpha-(F^\varepsilon_{0\mu}-F_{0\mu})F_{\alpha}^{~\mu}\textbf{U}^\alpha\\
&+\frac{1}{2}\textbf{U}_0F_{\mu\nu}((F^\varepsilon)^{\mu\nu}-F^{\mu\nu}),\nonumber
\end{align}
using the Proposition \ref{propal:tensordiffdefmod} and the normalization \eqref{eq:LWPpropagconstr2}.  Then, we calculate the derivative
\begin{align*}
		\frac{d}{dt}H^\varepsilon_\textbf{U}&=\frac{d}{dt}\int_{\mathbb{R}^3}{\eta^\varepsilon(\textbf{U})dx}\\
		&=-\int_{\mathbb{R}^3}{((T_{mKGM}^\varepsilon)_{\alpha\beta}-(T_{REM})_{\alpha\beta})\boldsymbol{\nabla}^\alpha \textbf{U}^\beta dx}-\int_{\mathbb{R}^3}{\partial_t(F_{0\mu}((F^\varepsilon)_{\alpha}^{~\mu}-F_{\alpha}^{~\mu})\textbf{U}^\alpha)+\partial_t((F^\varepsilon_{0\mu}-F_{0\mu})F_{\alpha}^{~\mu}\textbf{U}^\alpha))dx}\\
		&+\int_{\mathbb{R}^3}{\partial_t(\frac{1}{2}\textbf{U}_0F_{\mu\nu}((F^\varepsilon)^{\mu\nu}-F^{\mu\nu}))dx},\\
\end{align*}
where we used the fact that $(T_{mKGM}^\varepsilon)_{0\alpha}$ and $(T_{REM})_{0\alpha}$ have $0$ divergence, from \eqref{eq:divstressenergyKGM} and \eqref{eq:divstressenergyREM}, and that the total charges are conserved, from \eqref{eq:conservchargeKGM} and \eqref{eq:conservchargeREM}.\\
Then, with the Proposition \ref{propal:tensordiffdefmod}, the normalization \eqref{eq:LWPpropagconstr2} and the transport equation for $\textbf{U}$ in \eqref{eq:REMeq} we get 
\begin{align*}
		\frac{d}{dt}H^\varepsilon_\textbf{U}&=-\int_{\mathbb{R}^3}{h^\varepsilon_{\alpha\beta}\boldsymbol{\nabla}^\alpha \textbf{U}^\beta +F_{\alpha\beta}\textbf{U}^\alpha(\textbf{J}^\varepsilon)^\beta+\boldsymbol{\nabla}^\alpha  \textbf{U}^\beta [F_{\alpha\mu}((F^\varepsilon)_{\beta}^{~\mu}-F_{\beta}^{~\mu})+(F^\varepsilon_{\alpha\mu}-F_{\alpha\mu})F_{\beta}^{~\mu}]dx}\\
		&+\int_{\mathbb{R}^3}{\boldsymbol{\nabla}^\alpha \textbf{U}_\alpha[(\textbf{J}^\varepsilon_\gamma- \textbf{U}_\gamma \rho^\varepsilon) \textbf{U}^\gamma+\frac{1}{2}(F_{\mu\nu}((F^\varepsilon)^{\mu\nu}-F^{\mu\nu}))] dx}\\
		&+\int_{\mathbb{R}^3}{\boldsymbol{\nabla}_\beta (F^{\beta\mu}(F^\varepsilon_{\alpha\mu}-F_{\alpha\mu})\textbf{U}^\alpha)+\boldsymbol{\nabla}_\beta (((F^\varepsilon)^{\beta\mu}-F^{\beta\mu})F_{\alpha\mu}\textbf{U}^\alpha))-\boldsymbol{\nabla}_\beta(\frac{1}{2}\textbf{U}^\beta F_{\mu\nu}((F^\varepsilon)^{\mu\nu}-F^{\mu\nu}))dx}.\\
\end{align*}
We also add a full space divergence term in the integral. Then, we simplify everything using Maxwell equations in \eqref{eq:KGMKGM} and \eqref{eq:REMeq} and the antisymmetry of $F^\varepsilon$ and $F$ to get 
\begin{align*}
		\frac{d}{dt}H^\varepsilon_\textbf{U}&=-\int_{\mathbb{R}^3}{h^\varepsilon_{\alpha\beta}\boldsymbol{\nabla}^\alpha \textbf{U}^\beta +F_{\alpha\beta}\textbf{U}^\alpha(\textbf{J}^\varepsilon)^\beta-\boldsymbol{\nabla}^\alpha \textbf{U}_\alpha(\textbf{J}^\varepsilon_\gamma- \textbf{U}_\gamma \rho^\varepsilon) \textbf{U}^\gamma dx}\\
		&+\int_{\mathbb{R}^3}{\boldsymbol{\nabla}_\beta F^{\beta\mu}(F^\varepsilon_{\alpha\mu}-F_{\alpha\mu})\textbf{U}^\alpha + ((F^\varepsilon)^{\beta\mu}-F^{\beta\mu})\boldsymbol{\nabla}_\beta F_{\alpha\mu}\textbf{U}^\alpha-\frac{1}{2}\textbf{U}^\beta \boldsymbol{\nabla}_\beta F_{\mu\nu}((F^\varepsilon)^{\mu\nu}-F^{\mu\nu})dx}\\
&+\int_{\mathbb{R}^3}{F^{\beta\mu}\boldsymbol{\nabla}_\beta(F^\varepsilon_{\alpha\mu}-F_{\alpha\mu})\textbf{U}^\alpha +\boldsymbol{\nabla}_\beta ((F^\varepsilon)^{\beta\mu}-F^{\beta\mu})F_{\alpha\mu}\textbf{U}^\alpha-\frac{1}{2}\textbf{U}^\beta F_{\mu\nu} \boldsymbol{\nabla}_\beta ((F^\varepsilon)^{\mu\nu}-F^{\mu\nu})dx}\\
&=-\int_{\mathbb{R}^3}{h^\varepsilon_{\alpha\beta}\boldsymbol{\nabla}^\alpha \textbf{U}^\beta +F_{\alpha\beta}\textbf{U}^\alpha(\textbf{J}^\varepsilon)^\beta-\boldsymbol{\nabla}^\alpha \textbf{U}_\alpha(\textbf{J}^\varepsilon_\gamma- \textbf{U}_\gamma \rho^\varepsilon) \textbf{U}^\gamma dx}\\
		&+\int_{\mathbb{R}^3}{\textbf{U}^\mu\rho(F^\varepsilon_{\alpha\mu}-F_{\alpha\mu})\textbf{U}^\alpha +((F^\varepsilon)^{\beta\mu}-F^{\beta\mu})\textbf{U}^\alpha(\boldsymbol{\nabla}_\beta F_{\alpha\mu}-\frac{1}{2}\boldsymbol{\nabla}_\alpha F_{\beta\mu})dx}\\
&+\int_{\mathbb{R}^3}{F^{\beta\mu}\textbf{U}^\alpha(\boldsymbol{\nabla}_\beta(F^\varepsilon_{\alpha\mu}-F_{\alpha\mu})-\frac{1}{2}\boldsymbol{\nabla}_\alpha(F^\varepsilon_{\beta\mu}-F_{\beta\mu})) +((\textbf{J}^\varepsilon)^\mu-\textbf{U}^\mu\rho)F_{\alpha\mu}\textbf{U}^\alpha dx}\\
&=-\int_{\mathbb{R}^3}{h^\varepsilon_{\alpha\beta}\boldsymbol{\nabla}^\alpha \textbf{U}^\beta -\boldsymbol{\nabla}^\alpha \textbf{U}_\alpha(\textbf{J}^\varepsilon_\gamma- \textbf{U}_\gamma \rho^\varepsilon) \textbf{U}^\gamma dx}\\
		&+\int_{\mathbb{R}^3}{((F^\varepsilon)^{\beta\mu}-F^{\beta\mu})\textbf{U}^\alpha(\frac{1}{2}\boldsymbol{\nabla}_\beta F_{\alpha\mu}-\frac{1}{2}\boldsymbol{\nabla}_\mu F_{\alpha\beta}-\frac{1}{2}\boldsymbol{\nabla}_\alpha F_{\beta\mu})dx}\\
&+\int_{\mathbb{R}^3}{F^{\beta\mu}\textbf{U}^\alpha(\frac{1}{2}\boldsymbol{\nabla}_\beta(F^\varepsilon_{\alpha\mu}-F_{\alpha\mu})-\frac{1}{2}\boldsymbol{\nabla}_\mu(F^\varepsilon_{\alpha\beta}-F_{\alpha\beta})-\frac{1}{2}\boldsymbol{\nabla}_\alpha(F^\varepsilon_{\beta\mu}-F_{\beta\mu}))  dx}.\\
\end{align*}
Then, from the Bianchi equalities \eqref{eq:bianchiKGM} and \eqref{eq:bianchiREM} we get 
\begin{align*}
&\frac{1}{2}\boldsymbol{\nabla}_\beta F_{\alpha\mu}-\frac{1}{2}\boldsymbol{\nabla}_\mu F_{\alpha\beta}-\frac{1}{2}\boldsymbol{\nabla}_\alpha F_{\beta\mu}=-\frac{1}{2}\boldsymbol{\nabla}_\beta F_{\mu\alpha}-\frac{1}{2}\boldsymbol{\nabla}_\mu F_{\alpha\beta}-\frac{1}{2}\boldsymbol{\nabla}_\alpha F_{\beta\mu}=0,\\
&\frac{1}{2}\boldsymbol{\nabla}_\beta F^\varepsilon_{\alpha\mu}-\frac{1}{2}\boldsymbol{\nabla}_\mu F^\varepsilon_{\alpha\beta}-\frac{1}{2}\boldsymbol{\nabla}_\alpha F^\varepsilon_{\beta\mu}=-\frac{1}{2}\boldsymbol{\nabla}_\beta F^\varepsilon_{\mu\alpha}-\frac{1}{2}\boldsymbol{\nabla}_\mu F^\varepsilon_{\alpha\beta}-\frac{1}{2}\boldsymbol{\nabla}_\alpha F^\varepsilon_{\beta\mu}=0,\\
\end{align*}
 the two last lines cancel. Finally, we have the reduced form
\begin{equation}
		\frac{d}{dt}H^\varepsilon_{\textbf{U}}=\underbrace{-\int_{\mathbb{R}^3}{h^\varepsilon_{\alpha\beta}\boldsymbol{\nabla}^\alpha \textbf{U}^\beta dx}}_\text{$\mathcal{H}_1$} +\underbrace{\int_{\mathbb{R}^3}{\boldsymbol{\nabla}^\alpha \textbf{U}_\alpha(\textbf{J}^\varepsilon_\gamma- \textbf{U}_\gamma \rho^\varepsilon) \textbf{U}^\gamma dx}}_\text{$\mathcal{H}_2$}.
\end{equation}
For the $\mathcal{H}_1$ term, we clearly have 
\begin{equation}
\label{eq:H1prop}
\mathcal{H}_1\leq C_0\int_{\mathbb{R}^3}{h^\varepsilon_{00}dx}\leq C_0H^\varepsilon_\textbf{U}
\end{equation}
with \eqref{eq:backgroundineqyproof} and with the combination of Propositions \ref{propal:detailh00Coerc} and \ref{propal:HUgoodfinalCoerc}.\\
On the other hand, it seems that the $\mathcal{H}_2$ term is not controllable by $H^\varepsilon$ but only by $(H^\varepsilon)^{1/2}$, it is only linear in the quantity we are looking at.
In fact, because of the specific structure of this term we can compensate for the apparent lack of smallness.\\
Firstly, we have
\begin{align*}
&\mathcal{H}_2=\underbrace{\int_{\mathbb{R}^3}{\boldsymbol{\nabla}^\alpha \textbf{U}_\alpha \frac{(\textbf{J}^\varepsilon_\gamma-\textbf{U}_\gamma\rho^\varepsilon)(\textbf{U}^\gamma \rho^\varepsilon-(\textbf{J}^\varepsilon)^\gamma)}{2\rho^\varepsilon}dx}}_\text{$\mathcal{H}_{2.1}$}+\underbrace{\int_{\mathbb{R}^3}{\boldsymbol{\nabla}^\alpha \textbf{U}_\alpha (\frac{(\textbf{J}^\varepsilon)_\gamma(\textbf{J}^\varepsilon)^\gamma}{2\rho^\varepsilon}-\frac{\textbf{U}_\gamma \textbf{U}^\gamma \rho^\varepsilon}{2})dx}}_\text{$\mathcal{H}_{2.2}$}.
\end{align*}
We directly get 
\begin{equation}
\label{eq:H2.1prop}
	\mathcal{H}_{2.1}\leq C_0||\frac{\textbf{J}^\varepsilon-\textbf{U}\rho^\varepsilon}{\sqrt{\rho^\varepsilon}}||^2_{L^2}\leq C_0H^\varepsilon_\textbf{U}
\end{equation}
with \eqref{eq:backgroundineqyproof} and with \ref{propal:HUgoodfinalCoerc}. Then, we use equation \eqref{eq:splitJJ1KGM} and the normalization \eqref{eq:LWPpropagconstr2} to get 
\begin{align*}
\mathcal{H}_{2.2}=\int_{\mathbb{R}^3}{\boldsymbol{\nabla}^\alpha \textbf{U}_\alpha \frac{\varepsilon^2\sqrt{\rho^\varepsilon}\Box\sqrt{\rho^\varepsilon}}{2}}dx.
\end{align*}
This leads to 
\begin{equation}
\label{eq:H2.2prop}
\begin{split}
    \mathcal{H}_{2.2}&=-\frac{d}{dt}\int_{\mathbb{R}^3}{\boldsymbol{\nabla}_\alpha \textbf{U}^\alpha\frac{\varepsilon^2\sqrt{\rho^\varepsilon}\partial_t \sqrt{\rho^\varepsilon}}{2}dx}+\int_{\mathbb{R}^3}{\boldsymbol{\nabla}_\alpha \textbf{U}^\alpha\frac{\varepsilon^2\partial_t \sqrt{\rho^\varepsilon}\partial_t \sqrt{\rho^\varepsilon}}{2}dx}+\int_{\mathbb{R}^3}{\partial_t\boldsymbol{\nabla}_\alpha \textbf{U}^\alpha\frac{\varepsilon^2\sqrt{\rho^\varepsilon}\partial_t \sqrt{\rho^\varepsilon}}{2}dx}\\
    &-\int_{\mathbb{R}^3}{\boldsymbol{\nabla}_\alpha \textbf{U}^\alpha\frac{\varepsilon^2\nabla \sqrt{\rho^\varepsilon}\nabla \sqrt{\rho^\varepsilon}}{2}dx}-\int_{\mathbb{R}^3}{\nabla\boldsymbol{\nabla}_\alpha \textbf{U}^\alpha\nabla \sqrt{\rho^\varepsilon}\frac{\varepsilon^2 \sqrt{\rho^\varepsilon}}{2}dx}\\
    &\leq -\frac{d}{dt}\int_{\mathbb{R}^3}{\boldsymbol{\nabla}_\alpha \textbf{U}^\alpha\frac{\varepsilon^2\sqrt{\rho^\varepsilon}\partial_t \sqrt{\rho^\varepsilon}}{2}dx}+C_0||\varepsilon \boldsymbol{\nabla} \sqrt{\rho^\varepsilon}||^2_{L^2}+\varepsilon C_0||\varepsilon \boldsymbol{\nabla} \sqrt{\rho^\varepsilon}||_{L^2}||\sqrt{\rho^\varepsilon}||_{L^2}\\
    &\leq \frac{d}{dt}\underbrace{\int_{\mathbb{R}^3}{-\boldsymbol{\nabla}_\alpha \textbf{U}^\alpha\frac{\varepsilon^2\sqrt{\rho^\varepsilon}\partial_t \sqrt{\rho^\varepsilon}}{2}dx}}_\text{$G^\varepsilon$}+C_0H^\varepsilon_\textbf{U}+\varepsilon C_0(H^\varepsilon_\textbf{U})^{1/2},
\end{split}
\end{equation}
where the mass term is control by \eqref{eq:energyandweightproof}.
Gathering everything, we have 
\begin{equation}
		\frac{d}{dt}H^\varepsilon_\textbf{U}=\mathcal{H}_{1}+\mathcal{H}_{2.1}+\mathcal{H}_{2.2}\leq C_0\left(H^\varepsilon_\textbf{U}+(H^\varepsilon_\textbf{U})^{1/2}\varepsilon\right)+\frac{d}{dt}G^\varepsilon.
\end{equation}
where we use \eqref{eq:H1prop}, \eqref{eq:H2.1prop} and \eqref{eq:H2.2prop}.
We integrate between $0$ and $t\in[0,T]$ and use the Young inequality to get 
\begin{align*}
		H^\varepsilon_\textbf{U}(t)\leq H^\varepsilon_\textbf{U}(0)+C_0\int^{t}_0{\left(H^\varepsilon_\textbf{U}(s)+\varepsilon^2 \right)ds}+G^\varepsilon(t)-G^\varepsilon(0).
\end{align*}
Now, using again the Young inequality properly and calculations that are similar to the previous ones, we find that for all $t\in[0,T]$
\begin{equation*}
		G^\varepsilon(t)\leq C_0(\delta^{-1}\varepsilon^2+\delta H^\varepsilon_\textbf{U}(t)),
\end{equation*}
for $\delta>0$ as small as we want. It gives us
\begin{equation}
\label{eq:almostHgoodPgproof}
				H^\varepsilon_\textbf{U}(t)\leq H^\varepsilon_\textbf{U}(0)+C_0(\int^{t}_0{\left(H^\varepsilon_\textbf{U}(s)+\varepsilon^2 \right) ds}+\delta H^\varepsilon_\textbf{U}(t)+\delta^{-1}\varepsilon^2+\delta H^\varepsilon_\textbf{U}(0)),
\end{equation}
and so for $\delta$ small in comparison to $C_0$ we have
\begin{equation}
\label{eq:finalHPgproof}
		H^\varepsilon_\textbf{U}(t)\leq C_0( H^\varepsilon_\textbf{U}(0)+\int^{t}_0{\left(H^\varepsilon_\textbf{U}(s)+\varepsilon^2 \right)ds}+\varepsilon^2).
\end{equation}
By Gronwall Lemma and under the assumption of convergence \ref{item:item2mainth} of the initial data of Theorem \ref{unTheorem:TH1mainth} we have
\begin{equation}
\label{eq:finalHsmallPgproof}
		H^\varepsilon_\textbf{U}(t)\leq C_0\varepsilon^2,
\end{equation}
for all $t\in[0,T]$ and thus the \textbf{propagation property} of Lemma \ref{lem:modpoint2Idea} for the equivalence class of $H^\varepsilon_\textbf{U}$, that is $H^\varepsilon=[H^\varepsilon_\textbf{U}]$ which includes $H^\varepsilon_0$. 
\end{proof}
\subsection{Sum up}
\label{subsection:sumup}
From Section \ref{subsection:Existreg} we have the point \ref{item:item1mainth} of Theorem \ref{unTheorem:TH1mainth}. Then, with the results of Sections \ref{subsection:Coerciviy} and \ref{subsection:Propagation} we get Lemmas \ref{lem:modpoint1Idea} and \ref{lem:modpoint2Idea} and the point \ref{item:item2mainth}. The smallness of the modulated energy propagates and implies that the semiclassical monokinetic limit of KGM is the REM system in the sense of the convergence of the density, the momentum and the electromagnetic field. 
\section{Relativistic massive Vlasov Maxwell}
\label{section:mRVM}
We show in this Section that the solution to REM is a weak solution to the relativistic massive Vlasov-Maxwell equations and give a Corollary of the main Theorem. 
\begin{propal}
\label{propal:REMtomRVM}
Let $(\textbf{U},F,\rho)$ be a solution to REM \eqref{eq:REMeq} given by Proposition \ref{propal:LWPREM}. We define the measure $\mu=\rho\delta_{\textbf{U}=\xi}\otimes\Lambda$ on $T^\star\mathcal{M}$ for $\mathcal{M}=[0,T]\times\mathbb{R}^{3}$ where $\delta_{\textbf{U}=\xi}$ is, for every $(x,t)\in\mathcal{M}$, the Dirac mass at the point $\textbf{U}(x,t)$ on the cotangent space $T^\star_{(t,x)}\mathcal{M}$ and $\Lambda$ is the usual Lebesgue measure on $\mathcal{M}$. 
Then, the couple $(F,\mu)$ is a weak solution to relativistic massive Vlasov Maxwell equations on $\mathcal{M}$ in the sense that $\mu$ is supported in momentum in the massive (time-like) region and 
\begin{equation}
\label{eq:mRVM}
\begin{cases}
\forall\varphi\in C^\infty_c(\mathcal{M}),\;\int_{\mathcal{M}}{\boldsymbol{\nabla}_\alpha F^{\alpha\beta}\varphi dxdt}=\int_{T^\star\mathcal{M}}{\boldsymbol{\xi}^\beta\varphi d\mu},\\
\forall a\in C^\infty_c(\mathcal{T^\star\mathcal{M}}),\;\int_{T^\star\mathcal{M}}{\boldsymbol{\xi}^\alpha\boldsymbol{\nabla}_\alpha a+F_{\alpha\beta}\boldsymbol{\xi}^\alpha \partial_{\boldsymbol{\xi}^\beta} ad\mu}=0.
\end{cases}
\end{equation}

\end{propal}
\begin{proof}
The support of $\mu$ in momentum is direct by the fact that $\textbf{U}^\alpha \textbf{U}_\alpha=-1$.
Then, we first have that
\begin{align*}
\forall (t,x)\in\mathcal{M},\;\boldsymbol{\nabla}_\alpha F^{\alpha\beta}(t,x)=\textbf{U}^\beta\rho(t,x)=\int_{T^\star_{(t,x)}\mathcal{M}}{\boldsymbol{\xi}^\beta \rho(x,t)\delta_{\textbf{U}=\xi}}.
\end{align*}
We integrate against a test function $\varphi$ and we directly get the first equation.\\
For the second equation, we know that 
\begin{equation}
\textbf{U}^\alpha\boldsymbol{\nabla}_\alpha\textbf{U}_\beta=F_{\alpha\beta}\textbf{U}^\alpha
\end{equation}
and
\begin{equation}
\boldsymbol{\nabla}_\alpha(\textbf{U}^\alpha\rho)=0,
\end{equation}
so that
\begin{align*}
\forall a\in C^\infty_c(\mathcal{T^\star\mathcal{M}}),\;\int_{T^\star\mathcal{M}}{\boldsymbol{\xi}^\alpha\boldsymbol{\nabla}_\alpha a+F_{\alpha\beta}\boldsymbol{\xi}^\alpha\partial_{\boldsymbol{\xi}^\beta} ad\mu}&=\int_{\mathcal{M}}{\textbf{U}^\alpha\rho\boldsymbol{\nabla}_\alpha a(t,x,\textbf{U})+F_{\alpha\beta}\textbf{U}^\alpha\rho\partial_{\boldsymbol{\xi}^\beta} a(t,x,\textbf{U})dxdt}\\
&=\int_{\mathcal{M}}{\textbf{U}^\alpha\rho(\boldsymbol{\nabla}_\alpha a(t,x,\textbf{U})+\boldsymbol{\nabla}_\alpha\textbf{U}_\beta\partial_{\boldsymbol{\xi}^\beta} a(t,x,\textbf{U}))dxdt}\\
&=\int_{\mathcal{M}}{\textbf{U}^\alpha\rho\boldsymbol{\nabla}_\alpha(a(t,x,\textbf{U}))dxdt}\\
&=-\int_{\mathcal{M}}{\boldsymbol{\nabla}_\alpha(\textbf{U}^\alpha\rho)a(t,x,\textbf{U})dxdt}\\
&=0,
\end{align*}
which is the desired equality. 
\end{proof}
\begin{defi}
\label{defi:momentumdensityRVM}
    For $(F,\mu)$ solution to \eqref{eq:mRVM}, we define the momentum $\textbf{J}$ as 
    \begin{equation}
        \label{eq:defmomentumRVM}
        \textbf{J}^\beta(t,x)=\int_{T^\star_{(t,x)}\mathcal{M}}{\boldsymbol{\xi}^\beta d\mu},
    \end{equation}
and the density $\rho$ as
      \begin{equation}
        \label{eq:defdensityRVM}
        \rho(t,x)=\int_{T^\star_{(t,x)}\mathcal{M}}{1d\mu}.
    \end{equation}
\end{defi}
\begin{cor}[of Theorem \ref{unTheorem:TH1mainth}]
\label{cor:TH2mainth} 
    Under the assumptions of Theorem \ref{unTheorem:TH1mainth}, we consider $\mu=\rho\delta_{\textbf{U}=\xi}\otimes\Lambda$ and $(\mu,F)$ the solution to RVM \eqref{eq:mRVM} defined in Proposition \ref{propal:REMtomRVM}. Then, we have
   \begin{equation}
   \label{eq:convergencescRVM}
        \begin{split}
            & \lim_{\varepsilon\to0}||\textbf{J}^\varepsilon-\textbf{J}||_{L^\infty([0,T],L^{1})}+||F^\varepsilon-F||_{L^\infty([0,T],L^2)}=0,\\
     &\lim_{\varepsilon\to0}||\rho^\varepsilon-\rho||_{L^\infty([0,T],L^1)}+||\sqrt{\rho^\varepsilon}-\sqrt{\rho}||_{L^\infty([0,T],L^2)}=0.
        \end{split}
       \end{equation}    
  where $\textbf{J}$ and $\rho$ are the momentum and the density associated with the measure $\mu$ given in Definition \ref{defi:momentumdensityRVM}.
\end{cor}
\begin{proof}
    We see that the density and the momentum of \eqref{eq:mRVM} from Definition \ref{defi:momentumdensityRVM} are equal to the density and the momentum of \eqref{eq:REMeq}. This implies the convergence \eqref{eq:convergencescRVM} directly.
\end{proof}
\section{Conclusion}
\label{section:conclusion}
Overall, we prove in Section \ref{section:PROOF} that the semi-classical monokinetic limit of the massive Klein-Gordon-Maxwell equations is the relativistic Euler-Maxwell system, which is the monokinetic case of the massive relativistic Vlasov-Maxwell system. More precisely, with the modulated energy method, we show that the dynamics of the density, the momentum, and the electromagnetic field are well approximated by the relativistic Euler-Maxwell equations at the semi-classical limit. In particular, this implies that if the initial data for $(\Phi^\varepsilon)_{0<\varepsilon<1}$ have only one direction of high oscillation (captured by $\textbf{U}$), then there remains only one direction of high oscillation on a finite interval of time.

The main constraints of this result are the following.
\begin{itemize}
    \item It only works in the massive case. See, for example, the requirement in Proposition \ref{propal:acceptimplyequivdefmod} and the compactness argument of Lemma \ref{lem:densityconvergeCoerc}. The latter relies on the uniform decay of $(\sqrt{\rho^\varepsilon})_{0<\varepsilon<1}$, the decay is propagated with the presence of the mass term in the weighted energy in Proposition \ref{propal:GWPKGM}.
    \item It requires a uniform decay for $(\sqrt{\rho^\varepsilon})_{0<\varepsilon<1}$ to have the convergence of the density on $\mathbb{R}^3$. In fact, to propagate the uniform decay, we need to assume that the weighted energy is uniformly bounded in the Definition of well-prepared initial data \ref{defi:wellprepKGMsc}.
    \item It requires regular solutions for mKGM and REM. This is due to the use of the results of \cite{zbMATH03781717} and \cite{zbMATH03781718} for the global well-posedness of mKGM and the Sobolev embeddings that give us bounds on the second derivatives of $\textbf{U}$ in the proof of the Proposition \ref{propal:mainpropalPropagation}.
    \item It only deals with the case of one direction of high-frequency oscillation (we see that there is only one vector field $\textbf{U}$), the monokinetic case. From the point of view of WKB analysis, the method does not apply to multi-phase ansatz. Nonetheless, contrary to what is suggested in \cite{salvi2024semiclassicallimitkleingordonequation} about the WKB superposition for the non-linear Klein-Gordon equation, there is no reason here to expect instability for the superposition. 
\end{itemize}

The next picture and the discussion that follows give a review of important results concerning our topic. We do not claim exhaustiveness, and we do not address the well-posedness problems associated with each equation. 
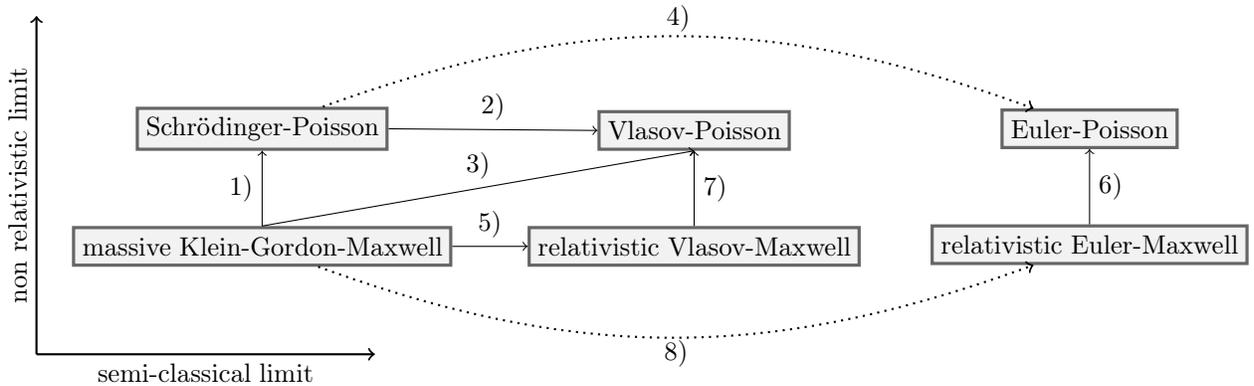
\begin{center}
\begin{tikzpicture}[
roundnode/.style={circle, draw=black!60, fill=black!5, very thick, minimum size=7mm},
squarednode/.style={rectangle, draw=black!60, fill=black!5, very thick, minimum size=5mm},
]
\node[squarednode]      (sch)   at (3,3)                       {Schrödinger-Poisson};
\node[squarednode]      (kg)        [below=of sch]                       {massive Klein-Gordon-Maxwell};
\node[squarednode]      (relat)     [right=of kg]                        {relativistic Vlasov-Maxwell};
\node[squarednode]      (vlasovpoiss)          [above=of relat]                     {Vlasov-Poisson};
\node[squarednode]      (eulerpoiss) at (14,3)                             {Euler-Poisson};
\node[squarednode]      (eulermax)          [below=of eulerpoiss]                     {relativistic Euler-Maxwell};

\draw[->] (kg.north) -- (sch.south) node[midway,left] {1)};
\draw[->] (kg.east) -- (relat.west)node[midway,above] {5)};
\draw[->] (sch.east) -- (vlasovpoiss.west)node[midway,above] {2)};
\draw[->] (relat.north) -- (vlasovpoiss.south)node[midway,right] {7)};
\draw[->] (kg.north) -- (vlasovpoiss.south)node[midway,above] {3)};
\draw[->,dotted,thick,postaction={decorate,decoration={raise=-2ex,text along path,text align=center,text={8)}}}]    (kg) to[out=-20,in=-160] (eulermax);
\draw[->,dotted,thick,postaction={decorate,decoration={raise=1ex,text along path,text align=center,text={4)}}}]    (sch) to[out=20,in=160] (eulerpoiss);
\draw[->]    (eulermax) -- (eulerpoiss)node[midway,right]{6)};
\draw[thick,->] (0,0) -- (4.5,0)node[midway,below] {semi-classical limit};
\draw[thick,->] (0,0) -- (0,4.5)node[midway,above,sloped]{non relativistic limit};;
\end{tikzpicture}
\end{center}

\begin{itemize}
   \item[--]The point 1) is proved in \cite{zbMATH02058032} and \cite{10.1155/S107379280320310X}.
  
   \item [--]The point 2) has been very much studied in various cases, we point out some major works. Firstly, the breakthrough works \cite{zbMATH00482230} and \cite{zbMATH00203353} contain the proof of the convergence of Schrödinger-Poisson to Vlasov-Poisson at the semi-classical limit for mixed states. More precisely, the authors show the weak convergence of the Wigner transform of density matrix solutions to the Hartree equation\footnote{So that the Wigner transform is solution to the Wigner-Poisson equations.} to the Vlasov-Poisson density at the semi-classical limit for mixed states with vanishing purity. Density matrices are used to represent mixed states (and so non-quantum uncertainty) and to generalize the usual wave function description. In \cite{zbMATH00482230} and \cite{zbMATH00203353}, the weak convergence in $L^2$ of the Wigner transform is required to deal with the singularity of the Coulomb interaction potential and for the nonlinearity to pass at the weak limit in the Wigner-Poisson equation. This holds if the Wigner transform is uniformly bounded in $L^2$, that is, if the purity of the states (the trace of the square density operator) vanishes in $O(\varepsilon^3)$ at the limit. Mixed states are usually associated with smoother Wigner transforms.\\
   Then, the restriction on mixed state was removed in 1 dimension in \cite{zbMATH01860573}. A discussion on this subject is given in \cite{Carles_2009}, \cite{Mauser2002SemiclassicalLO}, and a larger and more recent one can also be found in \cite{zbMATH07713387}. 
   In the latter (and in its improvement \cite{zbMATH07680771}), 
   the same semi-classical limit\footnote{In fact, both the attractive Newtonian gravitational potential and the repulsive Coulomb potential for Fermions (Hartree-Fock equation) and for Bosons (Hartree equation) are considered.} is studied for mixed states and a quantitative strong convergence of the density operator in Schatten norms (generalizing the trace norm and the Hilbert-Schmidt norm) is obtained. Still on this subject, we point out that in \cite{Lafleche_2021} the semi-classical limit is dealt with globally (and even for pure states) with weak convergence in a pseudo-distance norm based on the Wasserstein-Monge-Kantorovich distance following ideas found in \cite{zbMATH06571696} and \cite{zbMATH06688557} to study the semi-classical mean-field limit of the N-body Schrödinger equation with smoother interaction potential. See also the last paragraph below on the mean-field limit. \\

  \item [--] The point 3) is not proved yet as far as we know, but we refer to \cite{zbMATH05150102} for the proof of the semiclassical non-relativistic limit of the Dirac-Maxwell system to the Schrödinger-Poisson system. This work is related to our discussion because the Dirac equation is also an equation of relativistic quantum mechanics for the wave function. In particular, it takes into account particles' spin. 
   \item [--]The point 4) has been largely studied too. In \cite{zbMATH01987564} (presented in the introduction Section \ref{section:Intro}), the author tackles the semi-classical monokinetic limit of Schrödinger-Poisson via the modulated energy method. The result holds for pure states. The present work is a relativistic version of the latter with stronger convergence results.  In \cite{zbMATH05590720}, the same problem is studied in dimension 1, where a global existence result holds for the Euler-Poisson system, this leads to global convergence results at the semi-classical limit.
   In \cite{zbMATH05124485}, the authors study the semi-classical limit of Schrödinger-Poisson with an external potential and a doping profile (a background charge) that does not necessarily vanish at infinity. The limit is shown to be the Euler-Poisson system with an external potential and a doping profile of the same type, the argument relies on a modified WKB ansatz inspired by the one developed in \cite{zbMATH01132286} to study the semi-classical limit of the nonlinear Schrödinger equation. The Euler-Poisson system has also been obtained as a semi-classical mean field limit of the N-body Schrödinger equation with Coulomb interaction in \cite{zbMATH07570763}. The proof relies on a modulated energy method developed in \cite{Serfaty_2020} to study mean field limits, see the final paragraph of this Section. Generally, the point 4) can be understood as a sub case of the previous one 3). \\
   Finally, on a different but related topic, the semi-classical quasineutral (monokinetic) limit of Schrödinger-Poisson to incompressible Euler-Poisson is given in \cite{zbMATH01877186} and the semi-classical quasineutral strong magnetic field (monokinetic) limit of Schrödinger Poisson to Euler-Poisson is given in \cite{zbMATH01970480}. Both results rely on the modulated energy method.
   \item [--]
   The point 5) is well understood since the pioneer works \cite{zbMATH03962345}, \cite{zbMATH04004780}, and \cite{zbMATH04011069}. In the latter, the non-relativistic limit of regular solutions to the relativistic Vlasov-Maxwell equations is shown to be a solution to the Vlasov-Poisson equations. Recent improvements such as \cite{Schaeffer_2016} allow for the solution to have no decay, with infinite mass and energy, but our main interest is the result of \cite{brigouleix2020nonrelativisticlimitvlasovmaxwelluniform}. Indeed, in \cite{brigouleix2020nonrelativisticlimitvlasovmaxwelluniform} the regularity of the solutions to the relativistic Vlasov-Maxwell system is very much lowered, i.e., the (weak) solutions are required to be measures with uniformly bounded and integrable macroscopic density, finite first two moments and a moment of order $\alpha\in[0,1)$ with a controllable growth with respect to $\frac{1}{c}$. Some requirements also hold for the electromagnetic field and the convergence is understood in the weak-$\star$ sense of measures, via the Wasserstein-Monge-Kantorovich distance. The regularity of the density basically corresponds to the regularity for which one has a unique solution to the Vlasov-Poisson equations, as shown in \cite{zbMATH05125001}, of which is inspired the previously cited work \cite{brigouleix2020nonrelativisticlimitvlasovmaxwelluniform}. Its importance is clearer with the next point.
   \item [--] The point 6) is given as a particular case of \cite{brigouleix2020nonrelativisticlimitvlasovmaxwelluniform}. Indeed, the monokinetic case, such as presented in Proposition \ref{propal:REMtomRVM}, fits the regularity requirement given above as soon as the (macroscopic) density is bounded and integrable.  In that sense, we see that point 6) is a sub-case of point 5). We are not aware of another proof.
    \item [--] The point 7) is not proved yet as far as we know. In fact, the mixed states (and their smoother Wigner transform), used, for example, in the non-relativistic case as discussed in point 2), are not usually defined for the Klein-Gordon equations. There is no adaptation of \cite{zbMATH00482230}, \cite{zbMATH00203353} or other types of proofs for the relativistic case. Nonetheless, we point out recent works that deal with the "non-monokinetic" semi-classical limit of semi-relativistic or electromagnetic systems.  In \cite{zbMATH07699568} and  \cite{Aki_2008}, the semi-classical limit of the semi-relativistic Hartree-Fock equation (describing Fermions) and the Hartree equation (describing Bosons) are respectively studied. In both cases, the limiting system is shown to be the relativistic Vlasov-Poisson equations.
     In \cite{leopold2024derivationvlasovmaxwellmaxwellschrodingerequations}, the authors study the semi-classical limit of the regularized\footnote{The charged particles are extended and not seen as points or Dirac masses.} Schrödinger equations with a quantized electromagnetic field (the equation corresponding to the Pauli-Fierz Hamiltonian) in a mean-field regime and obtain the regularized non-relativistic Vlasov-Maxwell system at the limit. In \cite{möller2023paulipoissonequationsemiclassicallimit}, the author studies the semi-classical limit of the Pauli-Poisson system (with an external magnetic potential) and obtains the magnetic Vlasov-Poisson equations at the limit.
   All these works require mixed states.
     \item [--] The point 8) is proved in the present paper. 
         Generally, the point 8) can be understood as a sub case of the point 7). We add that the recent work \cite{yang2024semiclassicallimitpaulipoisswelleulerpoisswell} shows via the WKB method that the semi-classical (monokinetic) limit of the Pauli-Poisswell system (a semi-relativistic system that includes particles' spin) is the Euler-Poisswell system. The link between Pauli-Poisswell and Diarc-Maxwell together with the non-relativistic limit and the consistency of Pauli-Poisswell are discussed in \cite{zbMATH01621352}.
\end{itemize}
To enrich the discussion, we mention that the Schrödinger-Poisson system corresponds to the mean-field limit of the N-body Schrödinger equation, where the different particles interact together via the Coulomb potential with a $\frac{1}{N}$ coupling factor. Indeed, the Schrödinger-Poisson equations describe the behavior of the average particle of a N-body problem with a small correlation between particles and for N large. Mean-field limits are related to the semi-classical limit and some works, such as \cite{zbMATH07570763}, deal with both at the same time. The first mathematical proofs of the mean-field limit for smoother interaction potential are found in the pioneer work \cite{Hepp:1974vg}, in the second quantization formalism, followed by \cite{RevModPhys.52.569}. Both have been improved to include the Coulomb interaction, the former in \cite{zbMATH03691900} for specific initial data and the latter in the series of papers \cite{zbMATH01757302}, \cite{erdos2002derivationnonlinearschrodingerequation} and \cite{zbMATH01754787}. A simpler derivation is obtained in
\cite{pickl2010simplederivationmeanfield} via a functional that may recall the modulated energy one. The functional "counts" the relative number of particles of the N-body problem that are not in the "mean-field" state, it propagates its size with respect to $\frac{1}{N}$ and it is coercive as it implies the convergence of the reduced one-particle density matrix in the trace norm. \\
We are not aware of such a derivation of the KGM equations. Nonetheless, several works on mean-field limit include relativistic characteristics, we give here some examples. In \cite{elgart2005meanfielddynamicsboson} and \cite{Lee_2012}, the semi-relativistic Hartree equation is derived as a mean-field limit of the N-body Schrödinger equation with a relativistic dispersion relation and a Coulomb (or Newtonian) interaction potential. In \cite{Leopold_2020}, the Schrödinger-Maxwell equations are derived as a mean-field limit of the N-body Schrödinger equation with a quantized electromagnetic field (thus obeying the Pauli-Fierz Hamiltonian dynamics). Moreover, quitting the quantum framework, the regularized relativistic Vlasov-Maxwell system is obtained as a mean-field limit of a N-particle system in \cite{Golse_2012} and a non-relativistic mean-field limit is treated in \cite{Chen_2020} to derive the Vlasov-Poisson system. 

 \appendix
 \section{WKB analysis}
\label{section:WKB}
As explained in the introduction Section \ref{subsection:context}, solutions to \eqref{eq:KGMscintro} are intrinsically high frequency when $\varepsilon$ is small due to the scaling of the equation. The WKB method is commonly used to describe high-frequency solutions or approximate solutions. Thus, we will use WKB expansions to obtain approximate solutions to mKGM and to understand their behaviors heuristically. This is a complement to the rest of the paper, no new results are given. We refer to \cite{METIVIER2009169} for a general presentation of the WKB method and to \cite{zbMATH05590720}, \cite{zbMATH05124485}, \cite{zbMATH01132286}, and \cite{Giulini_2012} for application to the semi-classical limit. The latter gives the semi-classical WKB expansion for the Klein-Gordon equation with a fixed electromagnetic potential.  
\begin{propal}
\label{propal:approxWKB}
    Let $(\Phi_1^\varepsilon,\textbf{A}_1^\varepsilon)=(e^{i\frac{\omega}{\varepsilon}}\Psi,\textbf{A})$ with $\omega$ a smooth real phase, $\Psi$ a smooth complex amplitude and $\textbf{A}$ an electromagnetic potential with its associated Faraday tensor $F_{\alpha\beta}=\boldsymbol{\nabla}_\alpha\textbf{A}_{\beta}-\boldsymbol{\nabla}_\beta\textbf{A}_\alpha$. We assume that all the previously cited quantities are independent of $\varepsilon$. If we have 
    \begin{equation}
    \label{eq:approxsystWKB}
\begin{cases}
   \boldsymbol{\nabla}_\alpha F^{\alpha\beta}=(\boldsymbol{\nabla}^\beta\omega+\textbf{A}^\beta)|\Psi|^2,\\
    2(\boldsymbol{\nabla}^\alpha\omega+\textbf{A}^\alpha) \boldsymbol{\nabla}_\alpha\Psi+\boldsymbol{\nabla}_\alpha( \boldsymbol{\nabla}^\alpha\omega+\textbf{A}^\alpha)\Psi=0,\\
    (\boldsymbol{\nabla}^\alpha\omega+\textbf{A}^\alpha)(\boldsymbol{\nabla}_\alpha \omega+\textbf{A}_{\alpha})=-1.
    \end{cases}
\end{equation}
then, $(\Phi_1^\varepsilon,\textbf{A}_1^\varepsilon)$ is an approximate solution to \eqref{eq:KGMscintro}.
\end{propal}
\begin{proof}
    By direct calculations we get 
    \begin{align*}
\boldsymbol{\nabla}_\alpha (F^\varepsilon)^{\alpha\beta}+\varepsilon Im(\Phi^\varepsilon_1\overline{\boldsymbol{\nabla}^{\beta}\Phi_1^\varepsilon})-(\boldsymbol{A}_1^\varepsilon)^{\beta}|\Phi_1^\varepsilon|^2&=\boldsymbol{\nabla}_\alpha F^{\alpha\beta}+\varepsilon Im(\Psi^\varepsilon\overline{\boldsymbol{\nabla}^{\beta}\Psi^\varepsilon})-(\boldsymbol{\nabla}^\beta\omega+\textbf{A}^\beta)|\Psi|^2\\
&=O(\varepsilon),\\
\end{align*}
and
\begin{align*}
(\varepsilon\boldsymbol{\nabla}^\alpha+i(\boldsymbol{A}^\varepsilon)^\alpha)(\varepsilon\boldsymbol{\nabla}_\alpha+i\boldsymbol{A}^\varepsilon_\alpha)\Phi^\varepsilon-\Phi^\varepsilon&=-(\boldsymbol{\nabla}^\alpha\omega+\textbf{A}^\alpha)(\boldsymbol{\nabla}_\alpha \omega+\textbf{A}_{\alpha})e^{i\frac{\omega}{\varepsilon}}\Psi-e^{i\frac{\omega}{\varepsilon}}\Psi\\
&+ i2(\boldsymbol{\nabla}^\alpha\omega+\textbf{A}^\alpha) \boldsymbol{\nabla}_\alpha\Psi e^{i\frac{\omega}{\varepsilon}}+i\boldsymbol{\nabla}_\alpha( \boldsymbol{\nabla}^\alpha\omega+\textbf{A}^\alpha)\Psi e^{i\frac{\omega}{\varepsilon}}+\varepsilon^2\Box\Psi e^{i\frac{\omega}{\varepsilon}}\\
&=O(\varepsilon^2).
\end{align*}
\end{proof}
\begin{rem}
\label{rem:eikonalWKB}
   The equation 
   \begin{equation}
   \label{eq:eikonalWKB}
       (\boldsymbol{\nabla}^\alpha\omega+\textbf{A}^\alpha)(\boldsymbol{\nabla}_\alpha \omega+\textbf{A}_{\alpha})=-1
   \end{equation}
   is the so-called eikonal equation of the WKB analysis and also corresponds to the relativistic energy-momentum relation for a mass set to 1. It gives us the standard normalization for relativistic fluids given in Sections \ref{subsection:initassum} and \ref{section:REM}.\\
\end{rem}
 \begin{propal}
 \label{propal:REMWKB}
     Let $(\omega,\Psi,\textbf{A})$ be a solution to \eqref{eq:approxsystWKB}, then $(\textbf{U},F,\rho)=(\boldsymbol{\nabla}\omega+\textbf{A},d\textbf{A},|\Psi|^2)$ is a solution to 
    \begin{equation}
     \label{eq:REMWKB}
        \begin{cases}
               \boldsymbol{\nabla}_\alpha F^{\alpha\beta}=\textbf{U}^\beta\rho,\\               \textbf{U}^\alpha\boldsymbol{\nabla}_\alpha\rho+\boldsymbol{\nabla}_\alpha\textbf{U}^\alpha\rho=0,\\              \textbf{U}^\alpha\boldsymbol{\nabla}_\alpha\textbf{U}_\beta=F_{\alpha\beta}\textbf{U}^\alpha,\\
            \textbf{U}_\alpha\textbf{U}^\alpha=-1.
        \end{cases}
    \end{equation}
     This is the relativistic Euler-Maxwell system.
 \end{propal}
\begin{proof}
    The first equation is direct by definition of $(\textbf{U},F,\rho)$. We get the second equation of \eqref{eq:REMWKB} using the second equation of \eqref{eq:approxsystWKB} with 
    \begin{align*}
          \textbf{U}^\alpha\boldsymbol{\nabla}_\alpha \rho+\boldsymbol{\nabla}_\alpha \textbf{U}^\alpha\rho&= (\boldsymbol{\nabla}^\alpha\omega+\textbf{A}^\alpha)\boldsymbol{\nabla}_\alpha |\Psi|^2+\boldsymbol{\nabla}_\alpha( \boldsymbol{\nabla}^\alpha\omega+\textbf{A}^\alpha)|\Psi|^2\\
          &=\Psi((\boldsymbol{\nabla}^\alpha\omega+\textbf{A}^\alpha)\boldsymbol{\nabla}_\alpha \overline{\Psi}+\boldsymbol{\nabla}_\alpha( \boldsymbol{\nabla}^\alpha\omega+\textbf{A}^\alpha)\frac{\overline{\Psi}}{2})\\
          &+\overline{\Psi}((\boldsymbol{\nabla}^\alpha\omega+\textbf{A}^\alpha)\boldsymbol{\nabla}_\alpha\Psi+\boldsymbol{\nabla}_\alpha( \boldsymbol{\nabla}^\alpha\omega+\textbf{A}^\alpha)\frac{\Psi}{2})\\
          &=0,
    \end{align*}
    and the third equation of \eqref{eq:REMWKB} using the third equation of \eqref{eq:approxsystWKB} with 
    \begin{align}
          \textbf{U}^\alpha\boldsymbol{\nabla}_\alpha \textbf{U}_\beta-F_{\alpha\beta}\textbf{U}^\alpha&=(\boldsymbol{\nabla}^\alpha\omega+\textbf{A}^\alpha)\boldsymbol{\nabla}_\alpha(\boldsymbol{\nabla}_\beta\omega+\textbf{A}_\beta)-(\boldsymbol{\nabla}_\alpha\textbf{A}_\beta-\boldsymbol{\nabla}_\beta\textbf{A}_\alpha)(\boldsymbol{\nabla}^\alpha\omega+\textbf{A}^\alpha)\\
          &=\frac{1}{2}\boldsymbol{\nabla}_\beta((\boldsymbol{\nabla}^\alpha\omega+\textbf{A}^\alpha)(\boldsymbol{\nabla}_\alpha \omega+\textbf{A}_{\alpha}))\\
          &=0.
    \end{align}
    The last equation is a direct consequence of the eikonal equation \eqref{eq:eikonalWKB}. 
\end{proof}
\begin{rem}
\label{rem:observgoodWKB}
    With $(\Phi^\varepsilon_1,\textbf{A}^\varepsilon_1)$ from Proposition \ref{propal:approxWKB} and with the definition of the momentum $\textbf{J}^\varepsilon$, the density $\rho^\varepsilon$ and the electromagnetic field $F_{\alpha\beta}^\varepsilon$ from Section \ref{section:KGM}, we have 
     \begin{align}
        &\textbf{J}^\varepsilon=\frac{i}{2}(\Phi_1^\varepsilon\overline{\boldsymbol{D}\Phi_1^\varepsilon}-\overline{\Phi_1^\varepsilon}\boldsymbol{D}\Phi_1^\varepsilon)=(\boldsymbol{\nabla} \omega+\textbf{A})|\Psi|^2+O(\varepsilon)=\textbf{U}\rho+O(\varepsilon),\\
        &\rho^\varepsilon=|\Phi_1^\varepsilon|^2=|\Psi|^2=\rho,\\
        &F_{\alpha\beta}^\varepsilon=(\boldsymbol{\nabla}_\alpha\textbf{A}^\varepsilon_{1\beta}-\boldsymbol{\nabla}_\beta\textbf{A}^\varepsilon_{1\alpha})=(\boldsymbol{\nabla}_\alpha\textbf{A}_\beta-\boldsymbol{\nabla}_\beta\textbf{A}_\alpha)=F_{\alpha\beta}
    \end{align}
    so that, heuristically, $\textbf{U}\rho$, $\rho$, and $F$ are good approximations of the momentum, the density, and the electromagnetic field associated with $(\Phi_1^\varepsilon,\textbf{A}_1^\varepsilon)$ when $\varepsilon$ is small. This matches the result of Theorem \ref{unTheorem:TH1mainth}.
\end{rem}
 We also point out that the WKB analysis has been applied to the massless KGM system in \cite{zbMATH01799448} (more generally to the Yang-Mills-Higgs equations) and in \cite{salvi2024multiphasehighfrequencysolutions} (for multi-phase high-frequency ansatz) in a different context. Indeed, these works are not considering the semi-classical limit, $\varepsilon$ is set to 1. One specifies high-frequency initial data and then studies the behavior of the (approximate or exact) high-frequency solutions. This typically models inhomogeneities at small scales.\\ Such a study has been done for other related systems, for example, the Einstein equation. The high-frequency limit of the solutions to the vacuum Einstein equations is not a solution to the vacuum Einstein equations but rather to the Einstein-massless Vlasov equations. This is the Burnett's conjecture \cite{zbMATH04086482}. There is a rich literature on this subject, we refer to \cite{zbMATH07879915} for a recent review and to \cite{huneau2024burnettsconjecturegeneralizedwave}, \cite{zbMATH07009768}, and \cite{touati2024reverse} for proofs of the conjecture and its reverse counterpart with different assumptions and methods. Contrary to the Einstein equations case, it is shown in \cite{zbMATH01799448} and \cite{salvi2024multiphasehighfrequencysolutions} that the high-frequency limit of solutions to the massless KGM equations is a solution to a non-physically relevant system. From this point of view, the interesting high-frequency limit for mKGM is the semi-classical limit studied in the present paper. Indeed, the semi-classical limit is also a type of high-frequency limit and it has a stronger physical meaning.
\printbibliography
\end{document}